\documentclass[12pt]{amsart}
\usepackage{amsmath}
\usepackage{amsfonts}
\usepackage[T1]{fontenc} 
\usepackage{color}
\usepackage{url}

\usepackage{babel} \usepackage[babel=true]{microtype} \usepackage[hidelinks]{hyperref} 
\hypersetup{colorlinks=true, urlcolor=red,} 

\usepackage{a4wide}
\usepackage{verbatim}
\usepackage[parfill]{parskip}    
\usepackage{graphicx}
\usepackage{amssymb}
\usepackage{epstopdf}
\usepackage{psfrag}

\newcommand {\bea}{\begin{align}}
\newcommand {\ea}{\end{align}}
\newtheorem{proposition}{Proposition}[section]
\newtheorem{theorem}{Theorem}[section]
\newtheorem{Assumption}{Assumption}[section]
\newtheorem{lemma}{Lemma}[section]

\newtheorem{remark}{Remark}[section]

\newtheorem{corollary}{Corollary}[section]

\usepackage{xspace}

\usepackage{subfigure}

\newcommand{\thmref}[1]{{Theorem~\ref{#1}}}
\newcommand{\lemref}[1]{{Lemma~\ref{#1}}}
\newcommand{\secref}[1]{{Section~\ref{#1}}}
\newcommand{\assref}[1]{{Assumption~\ref{#1}}}

\newcommand{\rmref}[1]{{Remark~\ref{#1}}}

\newcommand{\eps}{\varepsilon}

\newcommand{\normal}{\mathbf{n}}

\newcommand{\revl}[1]{#1}
\newcommand{\revd}[1]{#1}

\begin{document}
\title[Approximation of the stochastic Cahn-Hilliard equation with white noise]
{Numerical approximation of the stochastic Cahn-Hilliard equation with space-time white noise near the sharp-interface limit}

\author{\v{L}ubom\'{i}r Ba\v{n}as}
\address{Department of Mathematics, Bielefeld University, 33501 Bielefeld, Germany}
\email{banas@math.uni-bielefeld.de}
\author{Jean Daniel Mukam}
\address{Department of Mathematics, Bielefeld University, 33501 Bielefeld, Germany}
\email{jmukam@math.uni-bielefeld.de}

\begin{abstract}
We consider the stochastic Cahn-Hilliard equation with additive space-time white noise $\eps^{\gamma}\dot{W}$ in dimension $d=2,3$, where $\eps>0$ is an interfacial width parameter.  
We study  a numerical approximation of the equation which combines a structure preserving implicit time-discretization scheme with a discrete approximation of the space-time white noise.
We derive a strong error estimate for the considered numerical approximation which  is robust with respect to the inverse 
of the interfacial width parameter $\eps$.
Furthermore,  by a splitting approach, we show that for sufficiently large scaling parameter $\gamma$, the numerical approximation of the stochastic Cahn-Hilliard equation converges uniformly to the deterministic Hele-Shaw/Mullins-Sekerka problem in the sharp-interface limit $\eps\rightarrow 0$.
\end{abstract}

\maketitle



\section{Introduction}
\label{Introduction}
We consider the stochastic Cahn-Hilliard equation with additive space-time white noise:
\begin{align}
\label{model1}
& \mathrm{d} u  = \Delta\left(-\eps\Delta u+\frac{1}{\eps}f(u)\right)dt+\eps^{\gamma}\mathrm{d}W\quad &&\text{in}\; \mathcal{D}_T:=(0, T)\times\mathcal{D},\\
& \partial_{\normal} u  = \partial_{\normal} \Delta u =0\;\;\;\;&&\mbox{on}\;\; (0,T) \times \partial \mathcal{D},\\
& u(0) = u_0^\eps\;\;\;\;&&\mbox{in}\;\;\mathcal{D},
\end{align}
where $\mathcal{D}=[0, 1]^d$, $d=2,3$, $\normal$ is the outward normal unit vector to $\partial\mathcal{D}$, $\gamma>0$ is a noise scaling parameter, $\eps>0$ is a (small) interfacial width parameter 
and $W$ is the space-time white noise.  
The nonlinearity $f$ in (\ref{model1}) is given by $f(u)=F'(u)=u^3-u$, where $F(u)=\frac{1}{4}(u^2-1)^2$ is a double-well potential.
Equation (\ref{model1}) can be interpreted as a stochastically perturbed $\mathbb{H}^{-1}$-gradient flow of the  Ginzburg-Landau free energy
 \begin{align}
 \label{energy0}
\mathcal{E}(u):=\int_{\mathcal{D}}\left(\frac{\eps}{2}\vert \nabla u\vert^2+\frac{1}{\eps}F(u)\right)dx=\frac{\eps}{2}\Vert \nabla u\Vert^2_{\mathbb{L}^2}+\frac{1}{\eps}\Vert F(u)\Vert_{\mathbb{L}^1}. 
\end{align}
The Cahn-Hilliard equation is a classical model for phase separation and
coarsening phenomena in binary alloys.
In the seminal paper \cite{abc94} it is shown that the sharp-interface limit (i.e., the limit for $\eps \rightarrow 0$) of the deterministic Cahn-Hilliard equation is the (deterministic) Hele-Shaw/Mullins-Sekerka problem.
The study of the sharp-interface limit of the stochastic Cahn-Hilliard equation is a relatively recent topic.
The sharp-interface limit of the stochastic Cahn-Hilliard equation with smooth noise was considered in \cite{abk18} where it is shown that for
sufficiently strong scaling of the noise the stochastic problem converges to the deterministic Hele-Shaw/Mullins-Sekerka problem.
Analogous results for the stochastic Cahn-Hilliard equation with singular noises (including the space-time white noise) have been obtained recently in \cite{BYZ22,BM24}.
Sharp-interface limit of numerical approximation of the stochastic Cahn-Hilliard equation with smooth noise
and uniform convergence to the deterministic Hele-Shaw/Mullins-Sekerka problem for $\eps \rightarrow 0$
has been shown in \cite{Banas19} in spatial dimension $d=2$. 
 We also mention the recent work \cite{sch_aposter} which derives robust a posteriori error estimates for the numerical approximation
of the stochastic Cahn-Hilliard equation with smooth noise, relaxing the assumption of asymptotically small noise.

In the present work we generalize the result of \cite{Banas19} to spatial dimension $d=3$
and consider the physically relevant case of space-time white noise.


Formally, the space-time white noise can be written as (see e.g., \cite{DaPratoZabczyk})
 \begin{align}
\label{decompoNoise}
W(t, x)=\sum_{k\in\mathbb{N}^d}e_k(x)\beta_k(t),\quad (t,x)\in \mathcal{D}_T,
\end{align}
 where $(e_k)_{k\in \mathbb{N}^d}$  is an orthonormal basis of $\mathbb{L}^2_0:=  \{ v\in \mathbb{L}^2;\;\; \int_{\mathcal{D}}v(x)\mathrm{d}x = 0\}$ consisting of eigenvectors of the Neumann Laplacian associated with positive eigenvalues
 and  $(\beta_k)_{k\in \mathbb{N}^d}$ are independent Brownian motions on a given probability space $(\Omega,\, \mathcal{F},\, \mathbb{P})$. 
 Note that the representation \eqref{decompoNoise} is formal since the series do not converge in $\mathbb{L}^2$, $\mathbb{P}$-a.s.
 The space-time white noise \eqref{decompoNoise} satisfies $\mathbb{P}$-a.s. $\int_{\mathcal{D}}W(t,x)dx=0$ for $t\in [0, T]$ which ensures the mass conservation property $\int_{\mathcal{D}} u(t, x) dx  = \int_{\mathcal{D}} u_0^\eps(x) dx$ for $t\in [0, T]$, $\mathbb{P}$-a.s..
To simplify the notation we assume throughout the paper without loss of generality that $\int_{\mathcal{D}}u^{\eps}_0dx=0$.

Numerical experiments in \cite{Banas19} indicate that the convergence to the sharp interface limit also holds in the case of the space-time white noise. Nevertheless,
due to the limited regularity in the white noise case, the analysis \cite{Banas19} is not applicable in the present setting.
 In particular, the regularity of the solution of (\ref{model1}) with space-time white noise is limited to the weakest $\mathbb{H}^{-1}$-regularity setting, 
cf. \cite{BYZ22}, which is not sufficient to deduce uniform convergence to the sharp-interface limit.

The essential ingredients in the present work to cover the case of the space-time white noise in the spatial dimension $d=3$ are the following.
 \begin{itemize}
\item
The proposed numerical approximation  \eqref{scheme1b} of (\ref{model1})
combines a structure preserving time discretization scheme with a practical discrete approximation of the space-time white noise \eqref{Noiseapprox1}.
The considered discrete noise approximation allows to control the singularity of the space-time white noise by a suitable choice of discretization parameters,
cf. Remark~\ref{remarkmomenta}.
\item  
We adopt the approach of \cite{Banas19} which makes use of a spectral estimate of the linearized deterministic Cahn-Hilliard equation.
Similarly to \cite{Banas19} we employ a discrete stopping time argument combined with solution dependent probability subsets.
To treat the problem in spatial dimension dimension $d=3$ we make use of suitable interpolation inequalities \cite[Lemma 4.5]{BM24}.
Hence, with suitable scaling of the discretization parameters in the approximation of the white noise, in \thmref{mainresult1}
we obtain error estimates for the numerical approximation which are robust w.r.t. the interfacial width parameter (i.e., they depend polynomially on $\eps^{-1}$).

\item


A major obstacle to show uniform convergence of the proposed numerical approximation to the sharp-interface limit is the low regularity of the considered noise approximation \eqref{Noiseapprox2}.
To overcome this issue, we split the numerical solution \eqref{scheme1b} as $X^j=\widehat{X}^j+\widetilde{X}^j$, where $\widehat{X}^j$ and $\widetilde{X}^j$
solve \eqref{linearscheme} and \eqref{randompdescheme}, respectively.
The respective numerical schemes \eqref{linearscheme} and \eqref{randompdescheme} can be interpreted as implicit Euler approximations of a corresponding linear stochastic partial differential equation (SPDE) and a corresponding random partial differential equation (PDE),
cf. \cite{BYZ22, BM24}.
For sufficiently large $\gamma$, 
it is possible to treat the solution $\widetilde{X}^j$ as a small perturbation in terms of $\eps$ which is estimated in \lemref{regularitynoise}.
Hence, we study the error $\widehat{Z}^j:=X^j-X^j_{\mathrm{CH}}-\widetilde{X}^j$, where $X^j_{\mathrm{CH}}$ is the numerical approximation of the deterministic Cahn-Hilliard equation (i.e. \eqref{model1} with $W\equiv 0$).
The estimate of $\widehat{Z}^j$ is complicated by the fact that one needs to handle the (cubic) nonlinearity
$f(\widehat{X}^j+\widetilde{X}^j)-f(X^j_{\mathrm{CH}})$. Our strategy to control  this term consists in introducing {the subset $\Omega_{\kappa, J}$ in \eqref{SetXa} along with the subset} $\Omega_{\widetilde{W}}$ in \eqref{SetW} and estimate
$\widehat{Z}^j$  on  $\Omega_{\widetilde{W}}\cap\Omega_{\kappa, J}$, see \lemref{maintheorem2a}.
This  $L^{\infty}(0, T; \mathbb{L}^{\infty})$ estimate for $\widehat{Z}^j$ along with the $\mathbb{L}^\infty$-estimate \eqref{LinfinityXtilde} of $\widetilde{X}^j$
allow us to conclude a $L^{\infty}(0, T; \mathbb{L}^{\infty})$ estimate for the error ${Z}^j$ in \thmref{LinfinityZ}
which is the key ingredient to show the convergence in probability of the numerical scheme to the deterministic Hele-Shaw/Mullins-Sekerka problem in \thmref{maintheorem2}.
\end{itemize}
\revd{We note that in contrast to \cite[Lemma 5.1]{Banas19}, thanks to the improved time regularity of $\widehat{X}^j$ along with the bound for $\widetilde{X}^j$, the splitting $X^j=\widehat{X}^j+\widetilde{X}^j$
enables us to derive a $\tau$-independent $\mathbb{L}^{\infty}$-estimate of the numerical solution $X^j$ on a subset of high probability, see \lemref{Linfinitya}. Hence, we show the convergence of the numerical solution $X^j$ to the Hele-Shaw/Mullins-Sekerka
problem with less restrictive assumptions than in \cite{Banas19}.}

The remainder of the paper is organized as follows. In \secref{preliminaries} we introduce notation an preliminary results on the analytical properties of (\ref{model1}). In \secref{timeapproximation} we propose the numerical approximation of (\ref{model1}) and analyze its stability properties. 
Error estimates for the numerical approximation are derived in \secref{convergenceanalysis}. The sharp-interface limit of the approximation
is studied in \secref{SharpLimit}  where it is  shown that the proposed numerical approximation converges uniformly
to the deterministic  Hele-Shaw/Mullins-Sekerka problem for $\eps\rightarrow 0$.

\section{Notation and preliminaries}
\label{preliminaries}
Throughout the paper by $C$, $C_1$, $C_2$, $\dots$ we denote generic positive constants that may depend on the data $T$, $\mathcal{D}$, but are independent of other parameters (the interfacial width parameter $\eps$, the time-step $\tau$, the mesh size $h$).
\subsection{Function spaces}
\label{Notations}
For $p\in [1, \infty]$, we denote by $(\mathbb{L}^p, \Vert .\Vert_{\mathbb{L}^p})$ the standard space of   functions on $\mathcal{D}$ that are  $p$-th order integrable. We denote by $(., .)$ the inner product on $\mathbb{L}^2$ and by $\Vert .\Vert=\Vert .\Vert_{\mathbb{L}^2}$ its corresponding norm. For $k\in\mathbb{N}$, $(\mathbb{H}^k, \Vert.\Vert_{\mathbb{H}^k})$ stands for the standard Sobolev space of functions which and their up to $k$-th weak derivatives belong to $\mathbb{L}^2$, and $\mathbb{H}^s:=H^s(\mathcal{D})$, $s>0$ is the standard fractional Sobolev space.  For $r\geq 0$,  $\mathbb{H}^{-r}:=(\mathbb{H}^r)^*$ is the dual space of $\mathbb{H}^r$, and
\begin{align*}
\mathbb{H}_0^{-r}:=\{v\in \mathbb{H}^{-r}:\; \langle v, 1\rangle_r=0\},
\end{align*}
where $\langle., .\rangle_r$ stands for the duality pairing between $\mathbb{H}^r$ and $\mathbb{H}^{-r}$.

For $v\in \mathbb{L}^2$, we denote by $m(v)$ the mean of $v$, i.e.,
\begin{align*}
m(v):=\frac{1}{\vert \mathcal{D}\vert}\int_{\mathcal{D}}v(x)dx,\quad v\in \mathbb{L}^2.
\end{align*}
and define the space $\mathbb{L}^2_0=\{\varphi\in \mathbb{L}^2:\; m(\varphi)=0\}$.
For  $v\in \mathbb{L}^2_0$, let $v_1:=(-\Delta)^{-1}v\in \mathbb{H}^2\cap\mathbb{L}^2_0$ be the unique variational solution to:
\begin{align*}
-\Delta v_1=v\quad \text{in}\; \mathcal{D},\quad \partial_{ \normal} v_1=0\quad \text{on}\; \partial\mathcal{D}. 
\end{align*}
In particular,  $(\nabla(-\Delta)^{-1}\overline{v}, \nabla\varphi)=(\overline{v}, \varphi)$ for all $\varphi\in \mathbb{H}^1$, $\overline{v}\in \mathbb{L}^2_0$. We define $\Delta^{-\frac{1}{2}}v$ as
\begin{align*}
\Delta^{-\frac{1}{2}}v:=\nabla v_1=\nabla(-\Delta)^{-1}v.
\end{align*}
Using  Cauchy-Schwarz's inequality and the embedding $\mathbb{H}^1\hookrightarrow\mathbb{L}^2$ yields
\begin{align}
\label{equiv1}
\Vert \Delta^{-1/2}\overline{v}\Vert=\Vert \overline{v}\Vert_{\mathbb{H}^{-1}}=\sup_{g\in \mathbb{H}^1}\frac{\vert (\overline{v}, g)\vert}{\Vert g\Vert_{\mathbb{H}^1}}\leq \sup_{g\in \mathbb{H}^1}\frac{\Vert \overline{v}\Vert\Vert g\Vert}{\Vert g\Vert_{\mathbb{H}^1}}\leq C\Vert \overline{v}\Vert\quad \forall \overline{v}\in \mathbb{L}^2_0.
\end{align}
Using Poincar\'{e}'s inequality, the definition of the inverse Laplace $\Delta^{-1}$ and  Cauchy-Schwarz's inequality we deduce
\revd{
\begin{align*}
\Vert (-\Delta)^{-1}\overline{v}\Vert^2&\leq C_P\Vert \nabla(-\Delta)^{-1}\overline{v}\Vert^2 =C_P\left(\nabla(-\Delta)^{-1}\overline{v}, \nabla(-\Delta)^{-1}\overline{v}\right)=C_P\left(\overline{v}, (-\Delta)^{-1}\overline{v}\right)\nonumber\\
&\leq C_P\Vert \overline{v}\Vert\Vert(-\Delta)^{-1}\overline{v}\Vert\quad \forall \overline{v}\in\mathbb{L}^2_0.
\end{align*}
}
It therefore follows from the preceding estimate that
\begin{align}
\label{equiv2}
\Vert (-\Delta)^{-1}\overline{v}\Vert\leq C_P\Vert \overline{v}\Vert \quad\forall \overline{v}\in \mathbb{L}^2_0. 
\end{align}
The inner product on $\mathbb{H}^{-1}$ is defined as
\begin{align*}
(u,v)_{-1}:=\langle u, (-\Delta)^{-1}v\rangle=\langle v, (-\Delta)^{-1}u\rangle=(\nabla (-\Delta)^{-1}u, \nabla(-\Delta)^{-1}v)\quad u, v\in \mathbb{H}^{-1}.
\end{align*}
Note that $\mathbb{L}^2\hookrightarrow\mathbb{H}^{-1}\hookrightarrow\mathbb{L}^2$ defines a Gelfand triple.

The operator $-\Delta$ with domain $ D(-\Delta)=\{v\in \mathbb{H}^2:\;\partial_{\normal} v=0\;\text{on}\; \partial\mathcal{D}\}$ is self-adjoint, positive and has compact resolvent.  It possesses a basis of eigenvectors $\{e_j\}$, with corresponding eigenvalues $\{\lambda_j\}$ such that
$
0=\lambda_0<\lambda_1\leq \lambda_2\leq \cdots \lambda_j\longrightarrow +\infty.
$
Note that for $k=(k_1,\cdots, k_d)\in \mathbb{Z}^d$, $\lambda_k$ satisfies $\lambda_k\simeq \vert k\vert^2$, where $\vert k\vert^2=\lambda_1^2+\cdots+\lambda_d^2$.   

 For $s\in \mathbb{R}$ the fractional power $(-\Delta)^s$ is defined as
\begin{align*}
(-\Delta)^su=\sum_{j\in \mathbb{N}^d}\lambda^{s}_j(u, e_j)e_j\quad u\in \mathbb{L}^2,
\end{align*}
see e.g., \cite[Section 1.2]{Debussche1}. 
For $s\in\mathbb{R}$, the domain of $D((-\Delta)^{\frac{s}{2}})$ is given by (see e.g.,   \cite[Section 1.2]{Debussche1})
\begin{align*}
D((-\Delta)^{\frac{s}{2}}):=\left\{u=\sum_{j\in\mathbb{N}^d}(u, e_j)e_j:\;\sum_{j\in\mathbb{N}^d}\lambda_j^s\vert(u, e_j)\vert^2<\infty\right\}.
\end{align*}
We endow $D((-\Delta)^{\frac{s}{2}})$ with the semi-norm and semi-scalar product
\begin{align*}
\vert v\vert_s=\Vert (-\Delta)^{\frac{s}{2}}v\Vert\quad\text{and}\quad (u, v)_s=\left( (-\Delta)^{\frac{s}{2}}u, (-\Delta)^{\frac{s}{2}}v\right),\quad u, v\in D((-\Delta)^{\frac{s}{2}}).
\end{align*} 
We equip $D((-\Delta)^{\frac{s}{2}})$  with the norm $\Vert v\Vert_s=\left(\vert v\vert^2_s+m^2(v)\right)^{\frac{1}{2}}$,  $v\in D((-\Delta)^{\frac{s}{2}})$. For $s\in[0, 2]$, $D((-\Delta)^{\frac{s}{2}})$ is a closed subspace of $ \mathbb{H}^s$ and on $D((-\Delta)^{\frac{s}{2}})$ the norm $\Vert.\Vert_s$ is equivalent to the usual norm $\Vert .\Vert_{\mathbb{H}^s}$ of $\mathbb{H}^s$, see e.g.,  \cite[Section 1.2]{Debussche1}.
For $s>0$, $(-\Delta)^{-s}$ is a bounded linear operator in $\mathbb{L}^{2}$. It therefore follows that for $s\in[0, 2]$, on $D((-\Delta)^{\frac{s}{2}})$,   $\Vert(-\Delta)^{\frac{s}{2}}.\Vert$ is equivalent  to the standard norm $\Vert .\Vert_{\mathbb{H}^s}$ of $\mathbb{H}^s$. In fact, for all $v\in D((-\Delta)^{\frac{s}{2}})$, on one hand it holds that
\begin{align*}
\Vert v\Vert^2_{\mathbb{H}^s}&\leq C(\Vert (-\Delta)^{\frac{s}{2}}v\Vert^2+m^2(v))\leq C\Vert(-\Delta)^{\frac{s}{2}}v\Vert^2+C\Vert v\Vert^2\nonumber\\
&\leq C \Vert(-\Delta)^{\frac{s}{2}}v\Vert^2+C\Vert(-\Delta)^{-\frac{s}{2}}\Vert^2_{\mathcal{L}(\mathbb{L}^2)}\Vert (-\Delta)^{\frac{s}{2}}v\Vert^2\nonumber\\
&\leq C \Vert (-\Delta)^{\frac{s}{2}}v\Vert^2+C\Vert(-\Delta)^{\frac{s}{2}}v\Vert^2\leq C\Vert (-\Delta)^{\frac{s}{2}}v\Vert^2.
\end{align*}
On the other hand, it obviously holds that
$
\Vert (-\Delta)^{\frac{s}{2}}v\Vert^2\leq \vert v\vert^2_s+m^2(v)\leq C\Vert v\Vert^2_{\mathbb{H}^s}.
$

  \subsection{Existence and regularity results} 
  In this subsection, we summarize  the existence and some regularity results of the unique  strong variational solution to \eqref{model1}. 
\begin{proposition} (\cite[Theorem 2.1]{Debussche1}\& \cite[Theorem 3.1]{BM24})
\label{PropExistence}
Let the initial value $u^{\eps}_0$ be $\mathcal{F}_0$-measurable and $u^{\eps}_0\in \mathbb{H}^{-1}$, then \eqref{model1} has a unique strong variational solution, i.e., there exists a unique stochastic process $u\in C([0, T], \mathbb{H}^{-1})$ $\mathbb{P}$-a.s., such that for  $t\in [0, T]$ it holds
\begin{align*}
(u(t), \varphi)=(u^{\eps}_0, \varphi)+\int_0^t\left(-\eps\Delta u+\frac{1}{\eps}f(u), \Delta \varphi\right)ds+\left(\int_0^t\mathrm{d}W(s), \varphi\right)\quad \forall\varphi\in\mathbb{H}^2\quad \mathbb{P}\text{-a.s.}
\end{align*}
In addition, the   solution $u\in L^2\left(\Omega, \{\mathcal{F}\}_t,\mathbb{P}; C([0, T]; \mathbb{H}^{-1})\right)\cap L^4\left(\Omega, \{\mathcal{F}_t\}_t,\mathbb{P}; L^4(0, T; \mathbb{L}^4)\right)$  is  mass preserving, that is, $m(u)=0$. Moreover,
  \begin{align*}
  \mathbb{E}\left[\Vert u\Vert^2_{L^{\infty}(0, T;\mathbb{H}^{-1})}+\frac{1}{\eps}\Vert u\Vert^4_{L^4(0, T; \mathbb{L}^4)}\right]\leq C\left(1+\eps^{-1}+\eps^{4\gamma-3}\right).
  \end{align*}
\end{proposition}
\revd{
 To establish convergence of the iterated numerical approximation $X^j$ in \eqref{scheme1b} to the strong variational solution $u$ (cf. \thmref{mainresult1}), we need  an estimate  of $u-u_{\mathrm{CH}}$, where $u_{\mathrm{CH}}$ is the unique weak solution to the deterministic Cahn-Hilliard equation, that is, the weak solution of \eqref{model1} with $ W\equiv 0$. An  estimate of $u-u_{\mathrm{CH}}$ was obtained in \cite[Corollary 4.1]{BM24}. Such  estimate will be used here. But to make the paper self-contained, we  briefly recall it.

A central ingredient in deriving an  estimate for $u-u_{\mathrm{CH}}$ is the use  of a stopping time argument to control the drift nonlinearity. The stopping time in \cite{BM24} is defined as
\begin{align}
\label{StopT}
T_{\eps}:= T\wedge \inf\left\{t>0: \; \int_0^t\Vert u(s)-u_{\mathsf{A}}(s)-Z^{\eps}(s)\Vert^3_{\mathbb{L}^3}ds>\eps^{\sigma_0}\right\},
\end{align}
for some constant $\sigma_0>0$, where  $Z^{\eps}$ is the stochastic convolution, given by $Z^{\eps}(t)=\eps^{\gamma}\int_0^te^{-(t-s)\eps^2\Delta^2}dW(s)$.
The function $u_{\mathsf{A}}$ is an approximation of $u_{\mathrm{CH}}$ constructed in \cite{abc94} which satisfies (cf. \cite[Theorem 2.1]{abc94})
\begin{align}
\label{Alikakos}
\Vert u_{\mathrm{CH}}- u_{\mathsf{A}}\Vert_{L^p(0, T; \mathbb{L}^p)}\leq C\eps^k\quad \text{for}\; p\in\left(2, \frac{2d+8}{d+2}\right],
\end{align}
for some Constant $C$, independent of $\eps$ and  for
\begin{align*}
k>(d+2)\frac{d^2+6d+10}{4d+16}. 
\end{align*}
Moreover, the approximation $u_{\mathsf{A}}$ satisfies a spectral  estimate
\begin{equation}
\label{Spectral1}
\inf_{0\leq t\leq T}\inf_{w=(-\Delta)^{-1}\psi}\frac{\eps\Vert\nabla\psi\Vert^2+\frac{1}{\eps}\left(f'(u_{\mathsf{A}})\psi, \psi\right)}{\Vert\nabla w\Vert^2}\geq -C_0,
\end{equation}
where the constant $C_0>0$ does not depend of $\eps>0$.
 
  The stopping time \eqref{StopT} enables the derivation of an estimate of $u-u_{\mathsf{A}}-Z^{\eps}$ up to $T_{\eps}$ on a large sample subset
\begin{align}
\label{HigherSet}
\Omega_{\delta_0, \eta_0, \eps}:=\left\{\omega\in\Omega:\; \Vert Z^{\eps}\Vert_{C(\mathcal{D}_T)}\leq \eps^{\gamma-\frac{1}{4}-2\delta_0-2\eta_0}\right\}
\end{align}
that satisfies $\mathbb{P}[\Omega_{\delta_0, \eta_0, \eps}]\rightarrow 1$ for $\eps\rightarrow 0$, for some $\delta_0, \eta_0>0$ and for $\gamma>0$ large enough. More precesily, it is proved in \cite[Lemma 4.4]{BM24} that for $t\in[0, T_{\eps}]$ and $\omega\in\Omega_{\delta_0, \eta_0, \eps}$, it holds 
\begin{align}
\label{LuboJd}
&\sup_{s\in[0,t]}\Vert u(s)-u_{\mathsf{A}}(s)-Z^{\eps}(s)\Vert^2_{\mathbb{H}^{-1}}+\frac{13}{18\eps}\int_0^t\Vert u(s)-u_{\mathsf{A}}(s)-Z^{\eps}(s)\Vert^4_{\mathbb{L}^4}ds\nonumber\\
&\quad \leq C\left(\eps^{\sigma_0-1}+\eps^{\frac{4}{3}\left(\gamma-\frac{1}{4}-2\delta_0-2\eta_0\right)-1}+\eps^{\frac{3K-5}{2}}\right).
\end{align}
Under \assref{assumption1} below, it can be shown that $T_{\eps}\equiv T$. Using \eqref{LuboJd} and \eqref{Alikakos}, one can derive  an  estime for $u- u_{\mathrm{CH}}$, see \cite[Corollary 4.1]{BM24}.
 The derivation of such  estimate requires  the parameters $\gamma$, $\sigma_0$, $\delta_0$ and $\eta_0$ to satisfy the following assumption, see \cite{BM24} for more details.
 }
\begin{Assumption}
\label{assumption1}
Let $\mathcal{E}(u^{\eps}_0)<C$. Assume that for fixed \revl{$0 < \alpha < 7$, $2<  \delta \leq \frac{8}{3}$} the parameters $(\eta_0, \delta_0, \sigma_0,  \gamma)$ satisfy
\begin{align*}
  \sigma_0>\frac{(7-\alpha) \delta +6\alpha-8}{ \delta -2},\quad \gamma>\frac{3}{4}\sigma_0+\frac{1}{4}+2\delta_0+2\eta_0.
\end{align*}
\end{Assumption}
The following lemma gives an error bound  for the difference $u-u_{\mathrm{CH}}$, which is a consequence of \cite[Corollary 4.1]{BM24}. 
\begin{lemma}
\label{analyticCahnHilliard}
Let \assref{assumption1} be fulfilled  and assume that $\delta_0+\eta_0\geq \frac{4}{3}\sigma_0+1$. Then the following error estimates hold
\begin{align*}
\mathbb{E}\left[\Vert u-u_{\mathrm{CH}}\Vert^2_{L^{\infty}(0, T; \mathbb{H}^{-1})}+\frac{1}{\eps}\Vert u-u_{\mathrm{CH}}\Vert^4_{L^4(0, T; \mathbb{L}^4)}\right]\leq C\eps^{\frac{2\sigma_0}{3}},
\end{align*}
where $u$ is the strong variational solution to the stochastic Cahn-Hilliard equation \eqref{model1}  and $u_{\mathrm{CH}}$ is the unique weak solution to the deterministic Cahn-Hilliard equation.   
\end{lemma}



In \secref{convergenceanalysis} we perform an analogous analysis as above on the discrete level by using a stopping index $J_{\eps}$ in \eqref{Stopindex}, and a set $\Omega_2$ in \eqref{SetOmega2}, which are  discrete counterparts of $T_{\eps}$ and $\Omega_{\delta_0, \eta_0, \eps}$ respectivly. 

We provide in the next section the numerical approximation  and its a priori estimates.


\section{Numerical approximation}
\label{timeapproximation}  

In this section we construct a semi-discrete numerical approximation scheme for \eqref{model1} and analyze its stability properties.

We construct a discrete approximation of the noise as follows. Let $\mathcal{T}_h$ be a quasi-uniform triangulation of $\mathcal{D}$ with mesh-size $h=\max_{T\in \mathcal{T}_h}\mathrm{diam} (T)$ and $\mathbb{V}_h\subset\mathbb{H}^1$ be the space of globally continuous functions that are piecewise linear on $\mathcal{T}_h$, i.e.,
\begin{align*}
\mathbb{V}_h:=\{v_h\in C(\overline{\mathcal{D}}):\; v_h|_K\in \mathcal{P}_1(K)\quad \forall K\in \mathcal{T}_h\}.
\end{align*}

 We consider the basis $\{\phi_l\}_{l=1}^L$ of $\mathbb{V}_h$ such that $\mathbb{V}_h=\mathrm{span}\{\phi_l,\; l=1, \cdots, L\}$.
Following \cite{sllg_book,BanasBrzezniakProhl2013} we the consider the following $\mathbb{V}_h$-valued approximation of the space-time white noise
\begin{align}
\label{Noiseapprox1}
\Delta_jW(x):=\sum_{l=1}^L\frac{\phi_l(x)}{\sqrt{(d+1)^{-1}\vert (\phi_l, 1)\vert}}\Delta_j\beta_l,\quad x\in \overline{\mathcal{D}}\subset\mathbb{R}^d,
\end{align}
where  $\{\beta_l\}_{l=1}^L$ are standard real-valued Brownian motions and $\Delta_j\beta_l:=\beta_l(t_j)-\beta_l(t_{j-1})$.
\begin{remark}
The discrete noise \eqref{Noiseapprox1} was considered in \cite{sllg_book,BanasBrzezniakProhl2013} as an approximation of the space-time white noise, cf. \cite[Remark A.1]{BanasBrzezniakProhl2013}.
Approximation \eqref{Noiseapprox1} can also be interpreted as the $\mathbb{L}^2$-projection on $\mathbb{V}_h$ of the higher-dimensional  analogue of the piecewise constant approximation of the space-time white noise from \cite{AllenNovoselZhang1998}.
\end{remark}
To preserve the zero mean value property of the space-time white noise on the discrete level we normalize the approximation \eqref{Noiseapprox1} as follows  
\begin{align}
\label{Noiseapprox2}
\Delta_j\overline{W}:=\Delta_j{W}-m(\Delta_jW).
\end{align}

We consider the following semi-discrete numerical approximation of \eqref{model1} which
combines the implicit Euler time-discretisation with the noise approximation
(\ref{Noiseapprox2}): given $J\in \mathbb{N}$, $\mathbb{V}_h$ set $\tau=T/J$, $X^0 = u^{\eps}_0$  and determine $\{X^j\}_{j=1}^J$ as
 \begin{align}
\label{scheme1b}
\begin{array}{rclll}
(X^j-X^{j-1}, \varphi)+\tau(\nabla w^j, \nabla\varphi) & = &\eps^{\gamma}(\Delta_j\overline{W}, \varphi)\quad & \varphi\in\mathbb{H}^1,\\
\eps(\nabla X^j, \nabla\psi)+\displaystyle\frac{1}{\eps}(f(X^j), \psi) & =& (w^j, \psi)\quad\quad\quad\quad\quad & \psi\in\mathbb{H}^1.
\end{array}
\end{align}
For $\tau\leq \frac{1}{2}\eps^3$ the solutions of the implicit scheme \eqref{scheme1b} exist and are $\mathbb{P}$-a.s. unique for $j=1,\dots, J$, and $X^j$ is $\mathcal{F}_{t_j}$-measurable,
see \thmref{existencescheme} below.

We recall in the following lemma some  basic properties of nodal basis functions $(\phi_l)_{l=1}^L$ of $\mathbb{V}_h$ for quasi-uniform triangles, easy to verify and useful in the rest of the paper. 
\begin{lemma}
\label{Lemmabasis}
The following hold  for all $\phi_l\in \mathbb{V}_h$, $l=1, \cdots, L$ uniformly in $h$:
\begin{itemize}
\item[(i)]   $\Vert \phi_l\Vert_{\mathbb{L}^{\infty}}=1$, $ C_1h^d\leq\vert(\phi_l, 1)\vert\leq C_2h^{d}$, $L=dim(\mathbb{V}_h)\leq Ch^{-d}$,
\item[(ii)] $\Vert \phi_l\Vert\leq Ch^{\frac{d}{2}}$ and $\Vert\nabla \phi_l\Vert\leq Ch^{-1}\Vert \phi_l\Vert$.
\end{itemize}
\end{lemma}

Let us close this subsection by noting that the  nonlinearity $f$ satisfies the following identity, which will be used throughout the paper
\begin{align}
\label{nonlinearident}
f(a)-f(b)&=(a-b)f'(a)+(a-b)^3-3(a-b)^2a\nonumber\\
&=3(a-b)a^2-(a-b)+(a-b)^3-3(a-b)^2a.
\end{align}

 To obtain energy estimates for the numerical approximation \eqref{scheme1b}, we need the following preparatory lemma. 
\begin{lemma}
\label{LemmaBruit0}
The following estimates hold
\begin{align*} 
\mathbb{E}[\vert m(\Delta_jW)\vert^2]\leq C\tau,\quad \mathbb{E}[\Vert \Delta_j\overline{W}\Vert^2]\leq Ch^{-d}\tau+C\tau,\quad \mathbb{E}[\vert m(\Delta_jW)\vert^4]\leq Ch^{-2d}\tau^2. 
\end{align*}
\end{lemma}
\begin{proof}
Using \lemref{Lemmabasis} and the fact that $\mathbb{E}[(\Delta_j\beta_k)(\Delta_j\beta_l)]=0$ for $k\neq l$,  $\mathbb{E}[(\Delta_k\beta_l)^2]=\tau$,
we estimate the mean of the noise increment as
\begin{align*}
\mathbb{E}[\vert m(\Delta_jW)\vert^2]\leq C\sum_{l=1}^L\frac{(\phi_l, 1)^2}{h^d}\mathbb{E}[\vert\Delta_j\beta_l\vert^2]\leq C\tau.
\end{align*}
Using the definition of $\Delta_jW$ \eqref{Noiseapprox1}, the fact that $\mathbb{E}[(\Delta_j\beta_k)(\Delta_j\beta_l)]=0$ for $k\neq l$, $\mathbb{E}[(\Delta_k\beta_l)^2]=\tau$  and \lemref{Lemmabasis}, it follows that
\begin{align*}
\mathbb{E}[\Vert\Delta_jW\Vert^2]&=\mathbb{E}\left[\int_{\mathcal{D}}\left(\sum_{l=1}^L\frac{\phi_l(x)}{\sqrt{(d+1)^{-1}\vert (\phi_l, 1)\vert}}\Delta_j\beta_l\right)^2{dx}\right]\nonumber\\
&=\tau\int_{\mathcal{D}}\sum_{l=1}^L\frac{\phi_l^2(x)}{(d+1)^{-1}\vert (\phi_l, 1)\vert}dx=\tau\sum_{l=1}^L\frac{\Vert\phi_l\Vert^2}{(d+1)^{-1}\vert (\phi_l, 1)\vert}\leq Ch^{-d}\tau. 
\end{align*}
Using triangle inequality and the preceding estimates, it follows that
\begin{align*}
 \mathbb{E}\Vert\Delta_j\overline{W}\Vert^2\leq 2\mathbb{E}[\Vert\Delta_jW\Vert^2]+2\mathbb{E}[\vert m(\Delta_j{W})\vert^2]\leq Ch^{-d}\tau+C\tau. 
\end{align*}
Along the same lines as above, we obtain
\begin{align*}
\mathbb{E}[\vert m(\Delta_jW)\vert^4]\leq Ch^{-2d}\tau^2. 
\end{align*}
\end{proof}

In Lemmas \ref{momenta}, \ref{moment} and \ref{MomentLemma} below we derive energy estimates of the numerical approximation \eqref{scheme1b}.
\begin{lemma}  
\label{momenta}
Let  $0<\eps_0<<1$, $\eps\in(0, \eps_0)$ and $\tau\leq \frac{1}{2}\eps^3$. Then  the  scheme \eqref{scheme1b} satisfies
\begin{align*}
&\max_{1\leq j\leq J}\mathbb{E}[\mathcal{E}(X^j)]+\frac{\tau}{2}\sum_{j=1}^J\mathbb{E}[\Vert \nabla w^j\Vert^2]\\
&\leq C\left[\mathcal{E}(u^{\eps}_0)+\eps^{2\gamma+1}h^{-2-2d}+\eps^{4\gamma-1}h^{-6d}\tau+\eps^{2\gamma-3}h^{-d}+\eps^{2\gamma-1}h^{-3d}\right]\exp\left(\eps^{2\gamma-2} h^{-3d}\right),
\end{align*}
where the constant $C$ depends on $T$, but is independent of $\varepsilon$, $\tau$ and $h$.
\end{lemma}
\begin{proof}
Taking $\varphi=w^j(\omega)$ and $\psi=(X^j-X^{j-1})(\omega)$ in  \eqref{scheme1b}, with $\omega\in \Omega$  fixed and summing the resulting equations  yields 
\begin{align}
\label{bound1}
&\frac{\eps}{2}\left(\Vert\nabla X^j\Vert^2-\Vert \nabla X^{j-1}\Vert^2+\Vert \nabla(X^j-X^{j-1})\Vert^2\right)+\tau\Vert\nabla w^j\Vert^2+\frac{1}{\eps}(f(X^j), X^j-X^{j-1})\nonumber\\
&=\eps^{\gamma}(\Delta_j\overline{W}, w^j)\quad \mathbb{P}\text{-a.s.}
\end{align}
Setting $\mathfrak{f}(u):=\vert u\vert^2-1$ (hence $f(X^j)=\mathfrak{f}(X^j)X^j$), we recall  from  \cite[Section 3.1]{FengProhl} that
\begin{align}
\label{bound2}
\frac{1}{\eps}\left(f(X^j), X^j-X^{j-1}\right)\geq& \frac{1}{4\eps}\Vert\mathfrak{f}(X^j)\Vert^2-\frac{1}{4\eps}\Vert\mathfrak{f}(X^{j-1})\Vert^2+\frac{1}{4\eps}\Vert \mathfrak{f}(X^j)-\mathfrak{f}(X^{j-1})\Vert^2\nonumber\\
&-\frac{1}{2\eps}\Vert X^j-X^{j-1}\Vert^2. 
\end{align}
To estimate the third term on the left hand side of \eqref{bound1}, we take  $\varphi=(-\Delta)^{-1}(X^j-X^{j-1})(\omega)$ in the first   equation of \eqref{scheme1b}, with $\omega\in\Omega$  fixed. This yields
\begin{align}
\label{bound2b}
&\Vert \Delta^{-1/2}(X^j-X^{j-1})\Vert^2+\tau\left(\nabla w^j,  \nabla(-\Delta)^{-1}(X^j-X^{j-1})\right)\nonumber\\
&=\eps^{\gamma}\left(\Delta_j\overline{W}, (-\Delta)^{-1}(X^j-X^{j-1})\right).
\end{align}
Using Cauchy-Schwarz's  and triangle inequalities, it follows from \eqref{bound2b} that
\begin{align*}
\Vert \Delta^{-1/2}(X^j-X^{j-1})\Vert^2\leq \left(\tau\Vert\nabla w ^j\Vert+\eps^{\gamma}\Vert\Delta^{-1/2}\Delta_j\overline{W}\Vert\right)\Vert \Delta^{-1/2}(X^j-X^{j-1})\Vert. 
\end{align*}
Using the fact that $\Delta^{-1/2}$ is a linear bounded operator on $\mathbb{L}^2_0$ (cf. \eqref{equiv1}), it follows that
\begin{align}
\label{bound2a}
\Vert \Delta^{-1/2}(X^j-X^{j-1})\Vert\leq \tau\Vert\nabla w^j\Vert+C\eps^{\gamma}\Vert\Delta_j\overline{W}\Vert. 
\end{align}
Squaring both sides of \eqref{bound2a} and using Young's inequality
yields
\begin{align*}
\Vert \Delta^{-1/2}(X^j-X^{j-1})\Vert^2\leq 2\tau^2\Vert\nabla w^j\Vert^2+2C\eps^{2\gamma}\Vert\Delta_j\overline{W}\Vert^2. 
\end{align*}
Using Cauchy-Schwarz's inequality and the preceding estimate leads to
\begin{align}
\label{bound3}
\frac{1}{2\eps}\Vert X^j-X^{j-1}\Vert^2&=\frac{1}{2\eps}\left(\nabla(-\Delta)^{-1}(X^j-X^{j-1}), \nabla(X^j-X^{j-1})\right)\nonumber\\
&\leq \frac{1}{4\eps^3}\Vert \Delta^{-1/2}(X^j-X^{j-1})\Vert^2+\frac{\eps}{4}\Vert \nabla(X^j-X^{j-1})\Vert^2\\
&\leq \frac{\tau^2}{2\eps^3}\Vert\nabla w^j\Vert^2+C\eps^{2\gamma-3}\Vert\Delta_j\overline{W}\Vert^2+\frac{\eps}{4}\Vert \nabla(X^j-X^{j-1})\Vert^2.\nonumber
\end{align}
Substituting \eqref{bound3} and \eqref{bound2} in \eqref{bound1} yields
\begin{align}
\label{bound4}
&\frac{\eps}{2}\left(\Vert\nabla X^j\Vert^2-\Vert\nabla X^{j-1}\Vert^2\right)+\frac{\eps}{4}\Vert\nabla(X^j-X^{j-1})\Vert^2+\left(\tau-\frac{\tau^2}{2\eps^3}\right)\Vert\nabla w^j\Vert^2\nonumber\\
&+\frac{1}{4\eps}\left(\Vert \mathfrak{f}(X^j)\Vert^2-\Vert \mathfrak{f}(X^{j-1})\Vert^2+\Vert \mathfrak{f}(X^j)-\mathfrak{f}(X^{j-1})\Vert^2\right)\\
&\leq C\eps^{2\gamma-3}\Vert\Delta_j\overline{W}\Vert^2+\eps^{\gamma}(\Delta_j\overline{W}, w^j). \nonumber
\end{align}
In order to keep  the term involving $\Vert\nabla w^j\Vert^2$ on the left hand side of \eqref{bound4} positive, we require $\tau< 2\eps^3$. To estimate the second term on the right hand side of \eqref{bound4}, we  note that
\begin{align}
\label{Warman1} 
\eps^{\gamma}(\Delta_j\overline{W}, w^j)=\eps^{\gamma}(\Delta_jW, w^j)-\eps^{\gamma}(w^j, 1)m(\Delta_jW). 
\end{align}
The second equation in \eqref{scheme1b} with $\psi=1$  yields
\begin{align}
\label{Warman2}
\eps^{\gamma}(w^j, 1)m(\Delta_jW)&=\eps^{\gamma-1}[(f(X^j)-f(X^{j-1}), 1)+(f(X^{j-1}), 1)]m(\Delta_jW)\nonumber\\
&=:A_1+A_2.
\end{align}
Note that $\mathbb{E}[A_2]=0$. Next, on recalling $f(X^j)=\mathfrak{f}(X^j)X^j$, we can rewritte $A_1$ as follows. 
\begin{align*}
A_1&=\eps^{\gamma-1}\left([\mathfrak{f}(X^j)-\mathfrak{f}(X^{j-1})], X^j\right)m(\Delta_jW)+\eps^{\gamma-1}\left(\mathfrak{f}(X^{j-1}),[X^j-X^{j-1}]\right)m(\Delta_jW)\nonumber\\
&=:A_{1,1}+A_{1,2}.
\end{align*}
Using the embedding $\mathbb{L}^s\hookrightarrow\mathbb{L}^r$ $(r\leq s)$,  Cauchy-Schwarz and Young's inequalities yields
\begin{align*}
A_{1,1}&\leq \frac{1}{16\eps}\Vert\mathfrak{f}(X^j)-\mathfrak{f}(X^{j-1})\Vert^2+C\eps^{2\gamma-1}\Vert \vert X^j\vert^2\Vert_{\mathbb{L}^1}\vert m(\Delta_jW)\vert^2\nonumber\\
&\leq \frac{1}{16\eps}\Vert\mathfrak{f}(X^j)-\mathfrak{f}(X^{j-1})\Vert^2+C\eps^{2\gamma-1}\left(\Vert\mathfrak{f}(X^j)-\mathfrak{f}(X^{j-1})\Vert_{\mathbb{L}^1}+\Vert X^{j-1}\Vert^2\right)\vert m(\Delta_jW)\vert^2\nonumber\\
&\leq \frac{1}{8\eps}\Vert\mathfrak{f}(X^j)-\mathfrak{f}(X^{j-1})\Vert^2+C\eps^{4\gamma-1}\vert m(\Delta_jW)\vert^4+C\eps^{2\gamma-1}\left(\Vert\mathfrak{f}(X^{j-1}\Vert^2+1\right)\vert m(\Delta_jW)\vert^2. 
\end{align*}
 We  estimate  $A_{1,2}$ by Cauchy-Schwarz,  Poincar\'{e} and Young's inequalities as
\begin{align*}
A_{1,2}&\leq \eps^{\gamma-1}\Vert \mathfrak{f}(X^{j-1})\Vert\Vert X^j-X^{j-1}\Vert\vert m(\Delta_jW)\vert\nonumber\\
&\leq C_{\mathcal{D}} \eps^{\gamma-1}\Vert \mathfrak{f}(X^{j-1})\Vert\Vert \nabla(X^j-X^{j-1})\Vert\vert m(\Delta_jW)\vert\nonumber\\
&\leq C\eps^{2\gamma-3}\Vert\mathfrak{f}(X^{j-1})\Vert^2\vert m(\Delta_jW)\vert^2+\frac{\eps}{16}\Vert\nabla(X^j-X^{j-1})\Vert^2. 
\end{align*}
We use the above estimates of $A_{1,1}$ and $A_{1,2}$ to obtain an estimate of $A_1$. Substituting the estimate of $A_1$ in \eqref{Warman2} yields
\begin{align}
\label{Warman3}
\eps^{\gamma}(w^j,1)m(\Delta_jW)\leq & \frac{1}{8\eps}\Vert \mathfrak{f}(X^j)-\mathfrak{f}(X^{j-1})\Vert^2+C\eps^{4\gamma-1}\vert m(\Delta_jW)\vert^4\nonumber\\
&+C\eps^{2\gamma-1}\left(\Vert\mathfrak{f}(X^{j-1})\Vert^2+1\right)\vert m(\Delta_jW)\vert^2+C\eps^{2\gamma-3}\Vert\mathfrak{f}(X^{j-1})\Vert^2\vert m(\Delta_jW)\vert^2\\
&+\frac{\eps}{16}\Vert \nabla(X^j-X^{j-1}\Vert^2+A_2.\nonumber
\end{align}
 Substituting \eqref{Warman3} in \eqref{Warman1} and substituting the resulting estimate in \eqref{bound4} yields
\begin{align}
\label{moment1}
&\frac{\eps}{2}\left(\Vert\nabla X^j\Vert^2-\Vert\nabla X^{j-1}\Vert^2\right)+\frac{3\eps}{16}\Vert\nabla(X^j-X^{j-1})\Vert^2+\frac{\tau}{2}\Vert\nabla w^j\Vert^2\nonumber\\
&\quad+\frac{1}{4\eps}\left(\Vert \mathfrak{f}(X^j)\Vert^2-\Vert \mathfrak{f}(X^{j-1})\Vert^2\right)+\frac{1}{8\eps}\Vert \mathfrak{f}(X^j)-\mathfrak{f}(X^{j-1})\Vert^2\nonumber\\
&\leq C\eps^{2\gamma-3}\Vert\Delta_j\overline{W}\Vert^2+C\eps^{4\gamma-1}\vert m(\Delta_jW)\vert^4+C\eps^{2\gamma-1}\left(\Vert\mathfrak{f}(X^{j-1}\Vert^2+1\right)\vert m(\Delta_jW)\vert^2\\
&\quad+C\eps^{2\gamma-3}\Vert\mathfrak{f}(X^{j-1})\Vert^2\vert m(\Delta_jW)\vert^2+\eps^{\gamma}(\Delta_jW, w^j)+A_2. \nonumber
\end{align}
Taking  $\psi=\phi_l$ in the second equation in \eqref{scheme1b} leads to
\begin{align*}
\eps^{\gamma}(w^j, \phi_l)\Delta_j\beta_l=&\eps^{\gamma+1}(\nabla X^j, \nabla\phi_l)\Delta_j\beta_l+\eps^{\gamma-1}(f(X^j), \phi_l)\Delta_j\beta_l\nonumber\\
=&\eps^{\gamma+1}(\nabla X^j, \nabla\phi_l)\Delta_j\beta_l+\eps^{\gamma-1}(f(X^j)-f(X^{j-1}), \phi_l)\Delta_j \beta_l\nonumber\\
&+\eps^{\gamma-1}(f(X^{j-1}), \phi_l)\Delta_j\beta_l\quad l=1,\cdots, L. 
\end{align*}
Taking into account the preceding identity, it follows from \eqref{Noiseapprox1} that
\begin{align*}
\eps^{\gamma}(\Delta_jW, w^j)
&=\frac{\eps^{\gamma}}{(d+1)^{-\frac{1}{2}}}\sum_{l=1}^L\frac{1}{\sqrt{\vert (\phi_l,1)\vert}}(w^j, \phi_l)\Delta_j\beta_l\nonumber\\
&=\eps^{\gamma+1}(\nabla X^j,\nabla\Delta_jW)+\revd{\eps^{\gamma-1}(f(X^j)-f(X^{j-1}), \Delta_jW)}+\eps^{\gamma-1}(f(X^{j-1}),\Delta_jW)\nonumber\\
&=: B_1+B_2+B_3,
\end{align*}
where we used the notation 
\begin{align*}
\nabla\Delta_jW:=\frac{1}{(d+1)^{-\frac{1}{2}}}\sum_{l=1}^L\frac{1}{\sqrt{\vert (\phi_l,1)\vert}}\nabla\phi_l\Delta_j\beta_l.
\end{align*}
Note that $\mathbb{E}[B_{3}]=0$. 
In order to estimate $B_1$, we split it  as follows
\begin{align}
\label{DecompB1}
B_1=\eps^{\gamma+1}\left(\nabla(X^j-X^{j-1}), \nabla\Delta_jW\right)+\eps^{\gamma+1}(\nabla X^{j-1}, \nabla \Delta_jW)=:B_{1,1}+B_{1,2}.
\end{align}
Note that $\mathbb{E}[B_{1,2}]=0$. Using  Cauchy-Schwarz's inequality and \lemref{Lemmabasis},    it follows that
\begin{align}
\label{EstimationB11}
B_{1,1}&\leq C\eps^{\gamma+1}\sum_{l=1}^L\frac{1}{\sqrt{\vert (\phi_l, 1)\vert}}\Vert \nabla(X^j-X^{j-1})\Vert\Vert \nabla\phi_l\Vert\vert\Delta_j\beta_l\vert\nonumber\\
&\leq C\eps^{\gamma+1} h^{-1}\Vert\nabla(X^j-X^{j-1})\Vert\sum_{l=1}^L\vert\Delta_j\beta_l\vert\nonumber\\
&\leq\frac{\eps}{16}\Vert\nabla(X^j-X^{j-1})\Vert^2+C\eps^{2\gamma+1}h^{-2}\left(\sum_{l=1}^L\vert\Delta_j\beta_l\vert\right)^2\\
&\leq \frac{\eps}{16}\Vert\nabla(X^j-X^{j-1})\Vert^2+C\eps^{2\gamma+1}h^{-2-d}\sum_{l=1}^L\vert\Delta_j\beta_l\vert^2.\nonumber
\end{align}
In order to estimate $B_2$, we use the identity $f(X^j)=\mathfrak{f}(X^j)X^j$ to split it as follows
\begin{align*}
B_2=\eps^{\gamma-1}\left((\mathfrak{f}(X^j)-\mathfrak{f}(X^{j-1}))X^j, \Delta_jW\right)+\eps^{\gamma-1}\left(\mathfrak{f}(X^{j-1})(X^j-X^{j-1}), \Delta_jW\right)=:B_{2,1}+B_{2,2}.
\end{align*}
Using Cauchy-Schwarz's inequality, the embedding $\mathbb{L}^s\hookrightarrow \mathbb{L}^r$ $(r\leq s)$ and \lemref{Lemmabasis} yields
\begin{align*}
B_{2,1}
&\leq C\eps^{\gamma-1}h^{-\frac{d}{2}}\sum_{l=1}^L\Vert \mathfrak{f}(X^j)-\mathfrak{f}(X^{j-1})\Vert\Vert X^j\Vert\Vert\phi_l\Vert_{\mathbb{L}^{\infty}}\vert\Delta_j\beta_l\vert\nonumber\\
&\leq \frac{1}{16\eps}\Vert \mathfrak{f}(X^j)-\mathfrak{f}(X^{j-1})\Vert^2+C\eps^{2\gamma-1}\Vert\vert X^j\vert^2\Vert_{\mathbb{L}^1}h^{-d}\left(\sum_{l=1}^L\vert\Delta_j\beta_l\vert\right)^2.
\end{align*}
Using Young's inequality and \lemref{Lemmabasis} yields
\begin{align}
\label{EstimationB21}
B_{2,1}\leq& \frac{1}{16\eps}\Vert \mathfrak{f}(X^j)-\mathfrak{f}(X^{j-1})\Vert^2\nonumber\\
&+C\eps^{2\gamma-1}h^{-2d}\left(\Vert\mathfrak{f}(X^j)-\mathfrak{f}(X^{j-1})\Vert_{\mathbb{L}^1}+\Vert X^{j-1}\Vert^2\right)\sum_{l=1}^L\vert\Delta_j\beta_l\vert^2\nonumber\\
\leq& \frac{1}{8\eps}\Vert \mathfrak{f}(X^j)-\mathfrak{f}(X^{j-1})\Vert^2+C\eps^{4\gamma-1}h^{-5d}\sum_{l=1}^L\vert\Delta_j\beta_l\vert^4\\
&+C\eps^{2\gamma-1}h^{-2d}\left(\Vert\mathfrak{f}(X^{j-1})\Vert^2+1\right)\sum_{l=1}^L\vert\Delta_j\beta_l\vert^2.\nonumber
\end{align}
 Next, we use Cauchy-Schwarz,  Young, Poincar\'{e}'s inequalities and \lemref{Lemmabasis} to obtain
\begin{align}
\label{EstimationB22}
B_{2, 2}&\leq C\eps^{\gamma-1}h^{-\frac{d}{2}}\sum_{l=1}^L\Vert\mathfrak{f}(X^{j-1})\Vert\Vert X^j-X^{j-1}\Vert\Vert\phi_l\Vert_{\mathbb{L}^{\infty}}\vert\Delta_j\beta_l\vert\nonumber\\
&\leq \frac{\eps}{16}\Vert \nabla(X^j-X^{j-1})\Vert^2+C\eps^{2\gamma-3} h^{-d}\Vert\mathfrak{f}(X^{j-1})\Vert^2\left(\sum_{l=1}^L\vert\Delta_j\beta_l\vert\right)^2\\
&\leq\frac{\eps}{16}\Vert\nabla(X^j-X^{j-1})\Vert^2+C\eps^{2\gamma-3}h^{-2d}\Vert\mathfrak{f}(X^{j-1})\Vert^2\sum_{l=1}^L\vert\Delta_j\beta_l\vert^2. \nonumber
\end{align}
Substituting \eqref{EstimationB22}, \eqref{EstimationB21},  \eqref{EstimationB11}, \eqref{DecompB1} in \eqref{moment1}, noting  $\Vert F(u)\Vert_{\mathbb{L}^1}=\frac{1}{4}\Vert\mathfrak{f}(u)\Vert^2$ and \eqref{energy0},  yields
\begin{align}
\label{energy1} 
&\mathcal{E}(X^j)-\mathcal{E}(X^{j-1})+\frac{\tau}{2}\Vert\nabla w^j\Vert^2+\frac{\eps}{16}\Vert\nabla(X^j-X^{j-1})\Vert^2\nonumber\\
&\leq C\eps^{2\gamma-3}\Vert\Delta_j\overline{W}\Vert^2+C\eps^{4\gamma-1}\vert m(\Delta_jW)\vert^4+C\eps^{2\gamma+1}h^{-2-d}\sum_{l=1}^L\vert\Delta_j\beta_l\vert^2\\
&\quad+C\left(\eps^{2\gamma-1}\left(\Vert\mathfrak{f}(X^{j-1}\Vert^2+1\right)+\eps^{2\gamma-3}\Vert\mathfrak{f}(X^{j-1})\Vert^2\right)\vert m(\Delta_jW)\vert^2\nonumber\\
&\quad+C\eps^{4\gamma-1}h^{-5d}\sum_{l=1}^L\vert\Delta_j\beta_l\vert^4+C\eps^{2\gamma-3}h^{-2d}\Vert\mathfrak{f}(X^{j-1})\Vert^2\sum_{l=1}^L\vert\Delta_j\beta_l\vert^2\nonumber\\
&\quad+C\eps^{2\gamma-1}h^{-2d}\left(\Vert\mathfrak{f}(X^{j-1})\Vert^2+1\right)\sum_{l=1}^L\vert\Delta_j\beta_l\vert^2+A_2+B_{1,2}+B_3.\nonumber
\end{align}
Summing \eqref{energy1} over $j$,  taking the expectation,  recalling that $\mathbb{E}[A_2]=\mathbb{E}[B_{1,2}]=\mathbb{E}[B_3]=0$,  using \lemref{Lemmabasis}, the independence of $X^{i-1}$ and $\Delta_i\beta_l$,   yields
\begin{align*}
&\mathbb{E}[\mathcal{E}(X^j)]+\frac{\tau}{2}\sum_{i=1}^j\mathbb{E}[\Vert\nabla w^i\Vert^2]\nonumber\\
&\leq \mathbb{E}[\mathcal{E}(u^{\eps}_0)]+ C\eps^{2\gamma+1}h^{-2-2d}+C\eps^{4\gamma-1}h^{-6d}\tau+C\eps^{2\gamma-3}\sum_{i=1}^j\mathbb{E}[\Vert\Delta_i\overline{W}\Vert^2]\nonumber\\
&\quad+C\eps^{4\gamma-1}\sum_{i=1}^j\mathbb{E}[\vert m(\Delta_iW)\vert^4]+C\eps^{2\gamma-1}\sum_{i=1}^j\left(\mathbb{E}[\Vert\mathfrak{f}(X^{i-1}\Vert^2]+1\right)\mathbb{E}[\vert m(\Delta_iW)\vert^2]\nonumber\\
&\quad+C\eps^{2\gamma-3}\sum_{i=1}^j\mathbb{E}[\Vert\mathfrak{f}(X^{i-1})\Vert^2]\mathbb{E}[\vert m(\Delta_iW)\vert^2]+C\eps^{2\gamma-3}h^{-3d}\sum_{i=1}^j\mathbb{E}[\Vert\mathfrak{f}(X^{i-1})\Vert^2]\nonumber\\
&\quad+C\eps^{2\gamma-1}h^{-3d}\tau\sum_{i=1}^j\left(\mathbb{E}[\Vert\mathfrak{f}(X^{i-1})\Vert^2]+1\right).
\end{align*}
Using \lemref{LemmaBruit0}, it follows from the preceding estimate that
\begin{align}
\label{energy1a}
&\mathbb{E}[\mathcal{E}(X^j)]+\frac{\tau}{2}\sum_{i=1}^j\mathbb{E}[\Vert\nabla w^i\Vert^2]\nonumber\\
&\leq \mathcal{E}(u^{\eps}_0)+ C\eps^{2\gamma+1}h^{-2-2d}+C\eps^{4\gamma-1}h^{-6d}\tau+C\eps^{2\gamma-3}h^{-d}+C\eps^{4\gamma-1}h^{-2d}\tau+C\eps^{2\gamma-1}h^{-3d}\\
&\quad+C\left(\eps^{2\gamma-1}+ \eps^{2\gamma-3}+\eps^{2\gamma-3} h^{-3d}+\eps^{2\gamma-1}h^{-3d}\right)\tau\sum_{i=1}^j\mathbb{E}[\Vert\mathfrak{f}(X^{j-1})\Vert^2].\nonumber
\end{align}
Recalling that $\mathcal{E}(u)=\displaystyle\frac{\eps}{2}\Vert \nabla u\Vert^2_{\mathbb{L}^2}+\frac{1}{\eps}\Vert F(u)\Vert_{\mathbb{L}^1}$, $\Vert F(u)\Vert_{\mathbb{L}^1}=\frac{1}{4}\Vert\mathfrak{f}(u)\Vert^2_{\mathbb{L}^2}$, applying the discrete Gronwall lemma to \eqref{energy1a}, using the fact that $\eps>0$ and $h>0$ yields the desired result. 
\end{proof}

\begin{remark}
\label{remarkmomenta}
To control the exponential term on the right-hand side of the estimate in \lemref{momenta}
one may choose $h=\eps^{\eta}$ with $2\gamma-2-3\eta d\geq 0$, i.e., $0<\eta\leq \frac{2}{3d}\gamma-\frac{2}{3d}$, which ensures that
 \begin{align*}
 \eps^{2\gamma+1}h^{-2-2d}+\eps^{4\gamma-1}h^{-6d}+ \eps^{2\gamma-1}h^{-3d}\leq C\eps^{\beta}\quad \text{ for some}\; \beta\geq 0.
 \end{align*}
 One can also  check that if $0<\gamma<\frac{5}{2}$ then  $\eps^{2\gamma-3}h^{-d}\leq \eps^{-\alpha}$ for some $\alpha>0$ and if  $\gamma\geq \frac{5}{2}$ then  $ \eps^{2\gamma-3}h^{-d}\leq \eps^{\delta}$ for some $\delta\geq 0$. 
  In fact, for $h=\eps^{\eta}$, $\eps^{2\gamma-3}h^{-d}=\eps^{2\gamma-3-\eta d}$ and if $0<\eta\leq \frac{2}{3d}\gamma-\frac{2}{3d}$, $\gamma\geq\frac{5}{2}$, then $2\gamma-3-\eta d\geq 0$. 
Furthermore, if $0<\eta\leq \frac{1}{d}\gamma-\frac{3}{2d}$ then $ \eps^{2\gamma+1}h^{-2-2d}+\eps^{4\gamma-1}h^{-6d}+\eps^{2\gamma-3}h^{-d}\leq C\eps^{\beta}$ for some $\beta\geq 0$. 
 
Hence, under the addition condition $h=\eps^{\eta}$ with $0<\eta\leq \frac{2}{3d}\gamma-\frac{3}{2d}$, we deduce from \lemref{moment} by the above arguments that
there exists $\alpha, \beta, \delta>0$ such that
 \begin{align*}
\max\limits_{1\leq j\leq J}\revd{\mathbb{E}}[\mathcal{E}(X^j)]+\displaystyle\frac{\tau}{2}\sum\limits_{j=1}^J\mathbb{E}[\Vert \nabla w^j\Vert^2]\leq \left\{\begin{array}{ll}
C(\mathcal{E}(u^{\eps}_0)+\eps^{\beta}+\eps^{-\alpha})\;&\text{if} \; \gamma<\frac{5}{2},\\
C(\mathcal{E}(u^{\eps}_0)+\eps^{\beta}+\eps^{\delta})\;&\text{if} \; \gamma\geq \frac{5}{2}.
\end{array}\right.
\end{align*}
Note, that under the above condition, the estimate in  \lemref{moment} may still depend on polynomially on $1/\eps$.
This is analogous to \cite[Lemma 3.2]{Banas19}, where the condition $\gamma>\frac{3}{2}$ is imposed to obtained an $\eps$-independent estimate. 
In the present case, to obtain an $\eps$ independent estimate requires slightly stronger condition $\gamma\geq \frac{5}{2}$. 
\end{remark}

\begin{lemma}
\label{moment}
Let the assumptions in \lemref{momenta} be fulfilled. Let $\gamma\geq\frac{5}{2}$ and let the mesh-size be such that $h=\eps^{\eta}$, with $0<\eta\leq \frac{2}{3d}\gamma-\frac{3}{2d}$, then there exists $\alpha, \beta, \delta>0$ such that
\begin{align*}
\mathbb{E}\left[\max\limits_{1\leq j\leq J}\mathcal{E}(X^j)\right]+\frac{\tau}{2}\sum_{j=1}^J\mathbb{E}[\Vert\nabla w^j\Vert^2]\leq 
C\left(\mathcal{E}(u^{\eps}_0)+\eps^{\beta}+\eps^{\delta}\right).
\end{align*}
\end{lemma}
\begin{proof}
 The proof goes along the same lines as that of \lemref{momenta} by summing \eqref{energy1} and taking the maximum before applying the expectation. Additional terms involving the noise  can be handled by using the discrete Burkholder-Davis-Gundy inequality \cite[Lemma 3.3]{Banas19}. 

\end{proof}

\begin{lemma}
\label{MomentLemma}
Let the assumptions of \lemref{momenta} be fulfilled. Then it holds that
\begin{align}
\label{Allenergy}
\max\limits_{1\leq j\leq J}\mathbb{E}[\mathcal{E}(X^{j})^2]\leq C\left((\mathcal{E}(u^{\eps}_0))^2+\mathcal{N}(\eps,\gamma,h,\tau,d)\right)\exp\left(CT\mathcal{M}(\eps,\gamma,h,\tau,d)\right), 
\end{align}
where 
\begin{align*}
\mathcal{N}(\eps,\gamma,h,\tau,d):=&\eps^{2\gamma-3}h^{-2-3d}+\eps^{4\gamma-1}h^{-2d}\tau+\eps^{2\gamma-2}+\eps^{2\gamma-1}h^{-2-2d}+\eps^{4\gamma-1}h^{-6d}\tau\nonumber\\
&+\eps^{4\gamma-6}h^{-4d}\tau+\eps^{8\gamma-1}h^{-4d}\tau^3+\eps^{4\gamma-2}h^{-2d}\tau+ \eps^{4\gamma+2}h^{-6d}\tau\nonumber\\
&+\eps^{2\gamma+2}h^{-2-2d}+\eps^{8\gamma-2}h^{-12d}\tau^3+\eps^{2\gamma-2}h^{-3d}+\eps^{4\gamma-2}h^{-6d}\tau,
\end{align*}
and 
\begin{align*}
\mathcal{M}(\eps,\gamma,h,\tau,d):=\eps^{4\gamma-2}h^{-2d}\tau+\eps^{2\gamma-2}+\eps^{2\gamma-3}h^{-3d}+\eps^{4\gamma-6}h^{-6d}\tau.
\end{align*}
If in addition  $\gamma\geq \frac{5}{2}$ and $h=\eps^{\eta}$ for
\begin{align}
\label{Requirementmeshsize}
0<\eta\leq\min\left\{\frac{2\gamma-3}{2+3d}, \frac{2\gamma-6}{3d}\right\}, 
\end{align} 
 then it holds that
 \begin{itemize}
 \item[i)]  $\max\limits_{1\leq j\leq J}\mathbb{E}[\mathcal{E}(X^{j})^2]\leq C((\mathcal{E}(u^{\eps}_0))^2+1)$, 
 \item[ii)] $\mathbb{E}[\max\limits_{1\leq j\leq J}\mathcal{E}(X^{j})^2]\leq C((\mathcal{E}(u^{\eps}_0))^2+1)$.
 \end{itemize}
\end{lemma}
\begin{proof}
\revl{We multiply \eqref{energy1} by $\mathcal{E}(X^{j})$ and obtain using the identity $(a-b)a=\frac{1}{2}[a^2-b^2+(a-b)^2]$ on the left-hand side of the resulting inequality that
\begin{align}\label{energy1_hm}
& \frac{1}{2}\left[\vert\mathcal{E}(X^{j})\vert^2-\vert\mathcal{E}(X^{j-1})\vert^2+\vert\mathcal{E}(X^j)-\mathcal{E}(X^{j-1})\vert^2\right]
\\ \nonumber
& \qquad \qquad  \leq
\widetilde{A}_0 +
\mathcal{E}(X^{j}) A_2+ \mathcal{E}(X^{j}) B_{1,2}+ \mathcal{E}(X^{j})B_3,
\end{align}
where
\begin{align*}
\widetilde{A}_0 := & \mathcal{E}(X^{j})\Bigg(C\eps^{2\gamma-3}\Vert\Delta_j\overline{W}\Vert^2+C\eps^{4\gamma-1}\vert m(\Delta_jW)\vert^4+C\eps^{2\gamma+1}h^{-2-d}\sum_{l=1}^L\vert\Delta_j\beta_l\vert^2\\
&\quad+C\left(\eps^{2\gamma-1}\left(\Vert\mathfrak{f}(X^{j-1}\Vert^2+1\right)+\eps^{2\gamma-3}\Vert\mathfrak{f}(X^{j-1})\Vert^2\right)\vert m(\Delta_jW)\vert^2\nonumber\\
&\quad+C\eps^{4\gamma-1}h^{-5d}\sum_{l=1}^L\vert\Delta_j\beta_l\vert^4+C\eps^{2\gamma-3}h^{-2d}\Vert\mathfrak{f}(X^{j-1})\Vert^2\sum_{l=1}^L\vert\Delta_j\beta_l\vert^2\nonumber\\
&\quad+C\eps^{2\gamma-1}h^{-2d}\left(\Vert\mathfrak{f}(X^{j-1})\Vert^2+1\right)\sum_{l=1}^L\vert\Delta_j\beta_l\vert^2\Bigg).
\end{align*}
}
\revd{
We estimate the four resulting terms on the right-hand side of \eqref{energy1_hm} separately. We start with the estimate of $\widetilde{A}_0$.
We can rewrite $\widetilde{A}_0$ as follows
\begin{align*}
\widetilde{A}_0=&C\left(\mathcal{E}(X^{j})-\mathcal{E}(X^{j-1})\right)\Bigg(\eps^{2\gamma-3}\Vert\Delta_j\overline{W}\Vert^2+\eps^{4\gamma-1}\vert m(\Delta_jW)\vert^4+\eps^{2\gamma+1}h^{-2-d}\sum_{l=1}^L\vert\Delta_j\beta_l\vert^2\\
&\quad+\left(\eps^{2\gamma-1}\left(\Vert\mathfrak{f}(X^{j-1}\Vert^2+1\right)+\eps^{2\gamma-3}\Vert\mathfrak{f}(X^{j-1})\Vert^2\right)\vert m(\Delta_jW)\vert^2\nonumber\\
&\quad+\eps^{4\gamma-1}h^{-5d}\sum_{l=1}^L\vert\Delta_j\beta_l\vert^4+\eps^{2\gamma-3}h^{-2d}\Vert\mathfrak{f}(X^{j-1})\Vert^2\sum_{l=1}^L\vert\Delta_j\beta_l\vert^2\nonumber\\
&\quad+\eps^{2\gamma-1}h^{-2d}\left(\Vert\mathfrak{f}(X^{j-1})\Vert^2+1\right)\sum_{l=1}^L\vert\Delta_j\beta_l\vert^2\Bigg)+\widetilde{A}_{0,1},
\end{align*}
where 
\begin{align*}
\widetilde{A}_{0,1}:=& C\mathcal{E}(X^{j-1})\Bigg(\eps^{2\gamma-3}\Vert\Delta_j\overline{W}\Vert^2+\eps^{4\gamma-1}\vert m(\Delta_jW)\vert^4+\eps^{2\gamma+1}h^{-2-d}\sum_{l=1}^L\vert\Delta_j\beta_l\vert^2\\
&\quad+\left(\eps^{2\gamma-1}\left(\Vert\mathfrak{f}(X^{j-1}\Vert^2+1\right)+\eps^{2\gamma-3}\Vert\mathfrak{f}(X^{j-1})\Vert^2\right)\vert m(\Delta_jW)\vert^2\nonumber\\
&\quad+\eps^{4\gamma-1}h^{-5d}\sum_{l=1}^L\vert\Delta_j\beta_l\vert^4+\eps^{2\gamma-3}h^{-2d}\Vert\mathfrak{f}(X^{j-1})\Vert^2\sum_{l=1}^L\vert\Delta_j\beta_l\vert^2\nonumber\\
&\quad+\eps^{2\gamma-1}h^{-2d}\left(\Vert\mathfrak{f}(X^{j-1})\Vert^2+1\right)\sum_{l=1}^L\vert\Delta_j\beta_l\vert^2\Bigg).
\end{align*}
Using Young's inequality, we estimate $\widetilde{A}_0$ as follows
\begin{align}
\label{EstimateA0a}
\widetilde{A}_0 \leq \frac{1}{32}\vert \mathcal{E}(X^{j})-\mathcal{E}(X^{j-1})\vert^2+\widetilde{A}_{0,1}+\widetilde{A}_{0,2},
\end{align}
where 
\begin{align*}
\widetilde{A}_{0,2}&=C\Bigg(\eps^{2\gamma-3}\Vert\Delta_j\overline{W}\Vert^2+\eps^{4\gamma-1}\vert m(\Delta_jW)\vert^4+\eps^{2\gamma+1}h^{-2-d}\sum_{l=1}^L\vert\Delta_j\beta_l\vert^2\\
&\quad+\left(\eps^{2\gamma-1}\left(\Vert\mathfrak{f}(X^{j-1}\Vert^2+1\right)+\eps^{2\gamma-3}\Vert\mathfrak{f}(X^{j-1})\Vert^2\right)\vert m(\Delta_jW)\vert^2\nonumber\\
&\quad+\eps^{4\gamma-1}h^{-5d}\sum_{l=1}^L\vert\Delta_j\beta_l\vert^4+\eps^{2\gamma-3}h^{-2d}\Vert\mathfrak{f}(X^{j-1})\Vert^2\sum_{l=1}^L\vert\Delta_j\beta_l\vert^2\nonumber\\
&\quad+\eps^{2\gamma-1}h^{-2d}\left(\Vert\mathfrak{f}(X^{j-1})\Vert^2+1\right)\sum_{l=1}^L\vert\Delta_j\beta_l\vert^2\Bigg)^2.
\end{align*}
Note that the following estimate holds
\begin{align}
\label{Lowerenergy}
\displaystyle\eps\Vert \nabla X^j\Vert^2+\frac{1}{\eps}\Vert \mathfrak{f}(X^j)\Vert^2\leq C\mathcal{E}(X^j)\quad  j=1,\cdots, J.
\end{align}
Using \eqref{Lowerenergy}, considering only the leading factors of $\eps^{-1}$, using Young's inequality and \lemref{Lemmabasis}, we estimate $\widetilde{A}_{0,2}$ as follows
\begin{align}
\label{EstiA02}
\widetilde{A}_{0,2}&\leq C\Bigg(\eps^{2\gamma-3}\Vert\Delta_j\overline{W}\Vert^2+\eps^{4\gamma-1}\vert m(\Delta_jW)\vert^4+\eps^{2\gamma+1}h^{-2-d}\sum_{l=1}^L\vert\Delta_j\beta_l\vert^2\nonumber\\
&\quad+\eps^{2\gamma-1}\vert m(\Delta_jW)\vert^2+\eps^{2\gamma-2}\mathcal{E}(X^{j-1})\vert m(\Delta_jW)\vert^2+\eps^{4\gamma-1}h^{-5d}\sum_{l=1}^L\vert\Delta_j\beta_l\vert^4\nonumber\\
&\quad+\eps^{2\gamma-1}h^{-2d}\sum_{l=1}^L\vert\Delta_j\beta_l\vert^2+\eps^{2\gamma-2}h^{-2d}\mathcal{E}(X^{j-1})\sum_{l=1}^L\vert\Delta_j\beta_l\vert^2\Bigg)^2\nonumber\\
&\leq C\Bigg(\eps^{4\gamma-6}\Vert\Delta_j\overline{W}\Vert^4+\eps^{8\gamma-2}\vert m(\Delta_jW)\vert^8+\eps^{4\gamma+2}h^{-4-3d}\sum_{l=1}^L\vert\Delta_j\beta_l\vert^4\\
&\quad+\eps^{4\gamma-2}\vert m(\Delta_jW)\vert^4+\eps^{4\gamma-4}\vert\mathcal{E}(X^{j-1})\vert^2\vert m(\Delta_jW)\vert^4+\eps^{8\gamma-2}h^{-11d}\sum_{l=1}^L\vert\Delta_j\beta_l\vert^8\nonumber\\
&\quad+\eps^{4\gamma-2}h^{-5d}\sum_{l=1}^L\vert\Delta_j\beta_l\vert^4+\eps^{4\gamma-4}h^{-5d}\vert\mathcal{E}(X^{j-1})\vert^2\sum_{l=1}^L\vert\Delta_j\beta_l\vert^4\Bigg).\nonumber
\end{align}
Using \eqref{Lowerenergy} and considering only the leading factors of $\eps^{-1}$ and $h^{-1}$, we estimate $\widetilde{A}_{0,1}$ as follows
\begin{align}
\label{EstiA01}
\widetilde{A}_{0,1} &\leq C\mathcal{E}(X^{j-1})\Bigg(\eps^{2\gamma-3}\Vert\Delta_j\overline{W}\Vert^2+\eps^{4\gamma-1}\vert m(\Delta_jW)\vert^4+\eps^{2\gamma+1}h^{-2-d}\sum_{l=1}^L\vert\Delta_j\beta_l\vert^2\nonumber\\
&\quad+\eps^{2\gamma-1}h^{-2d}\sum_{l=1}^L\vert\Delta_j\beta_l\vert^2+\eps^{2\gamma-1}\vert m(\Delta_jW)\vert^2+\eps^{2\gamma-2}\mathcal{E}(X^{j-1})\vert m(\Delta_jW)\vert^2\nonumber\\
&\quad+\eps^{4\gamma-1}h^{-5d}\sum_{l=1}^L\vert\Delta_j\beta_l\vert^4+\eps^{2\gamma-2}h^{-2d}\mathcal{E}(X^{j-1})\sum_{l=1}^L\vert\Delta_j\beta_l\vert^2\Bigg)\nonumber\\
&\leq  C\Bigg(\eps^{2\gamma-3}\mathcal{E}(X^{j-1})\Vert\Delta_j\overline{W}\Vert^2+\eps^{4\gamma-1}\mathcal{E}(X^{j-1})\vert m(\Delta_jW)\vert^4\\
&\quad+\eps^{2\gamma-1}h^{-2-d}\mathcal{E}(X^{j-1})\sum_{l=1}^L\vert\Delta_j\beta_l\vert^2+\eps^{2\gamma-1}\mathcal{E}(X^{j-1})\vert m(\Delta_jW)\vert^2\nonumber\\
&\quad+\eps^{2\gamma-2}\vert\mathcal{E}(X^{j-1})\vert^2\vert m(\Delta_jW)\vert^2\nonumber\\
&\quad+\eps^{4\gamma-1}h^{-5d}\mathcal{E}(X^{j-1})\sum_{l=1}^L\vert\Delta_j\beta_l\vert^4+\eps^{2\gamma-2}h^{-2d}\vert\mathcal{E}(X^{j-1})\vert^2\sum_{l=1}^L\vert\Delta_j\beta_l\vert^2\Bigg).\nonumber
\end{align}
Substituting \eqref{EstiA01} and \eqref{EstiA02} in \eqref{EstimateA0a}, we obtain
\begin{align}
\label{EstimateA0b}
\widetilde{A}_{0} &\leq \frac{1}{32}\vert \mathcal{E}(X^{j})-\mathcal{E}(X^{j-1})\vert^2+ C\Bigg(\eps^{2\gamma-3}\mathcal{E}(X^{j-1})\Vert\Delta_j\overline{W}\Vert^2+\eps^{4\gamma-1}\mathcal{E}(X^{j-1})\vert m(\Delta_jW)\vert^4\nonumber\\
&\quad+\eps^{2\gamma-1}h^{-2-d}\mathcal{E}(X^{j-1})\sum_{l=1}^L\vert\Delta_j\beta_l\vert^2+\eps^{2\gamma-1}\mathcal{E}(X^{j-1})\vert m(\Delta_jW)\vert^2\nonumber\\
&\quad+\eps^{2\gamma-2}\vert\mathcal{E}(X^{j-1})\vert^2\vert m(\Delta_jW)\vert^2\nonumber\\
&\quad+\eps^{4\gamma-1}h^{-5d}\mathcal{E}(X^{j-1})\sum_{l=1}^L\vert\Delta_j\beta_l\vert^4+\eps^{2\gamma-2}h^{-2d}\vert\mathcal{E}(X^{j-1})\vert^2\sum_{l=1}^L\vert\Delta_j\beta_l\vert^2\Bigg)\nonumber\\
&\quad +C\Bigg(\eps^{4\gamma-6}\Vert\Delta_j\overline{W}\Vert^4+\eps^{8\gamma-2}\vert m(\Delta_jW)\vert^8+\eps^{4\gamma+2}h^{-4-3d}\sum_{l=1}^L\vert\Delta_j\beta_l\vert^4\\
&\quad+\eps^{4\gamma-2}\vert m(\Delta_jW)\vert^4+\eps^{4\gamma-4}\vert\mathcal{E}(X^{j-1})\vert^2\vert m(\Delta_jW)\vert^4+\eps^{8\gamma-2}h^{-11d}\sum_{l=1}^L\vert\Delta_j\beta_l\vert^8\nonumber\\
&\quad+\eps^{4\gamma-2}h^{-5d}\sum_{l=1}^L\vert\Delta_j\beta_l\vert^4+\eps^{4\gamma-4}h^{-5d}\vert\mathcal{E}(X^{j-1})\vert^2\sum_{l=1}^L\vert\Delta_j\beta_l\vert^4\Bigg)\nonumber.
\end{align}

}
Now we  estimate  $\mathcal{E}(X^{j})B_{1,2}$.
Using  Young's inequality we get
\begin{align*}
\mathcal{E}(X^{j})B_{1,2}=&\frac{\eps^{\gamma+1}}{(d+1)^{-\frac{1}{2}}}\sum_{l=1}^L\frac{1}{\sqrt{\vert (\phi_l, 1)\vert}}\mathcal{E}(X^{j-1})(\nabla X^{j-1}, \nabla\phi_l)\Delta_j\beta_l\nonumber\\
&+\frac{\eps^{\gamma+1}\left(\mathcal{E}(X^j)-\mathcal{E}(X^{j-1})\right)}{(d+1)^{-\frac{1}{2}}}\sum_{l=1}^L\frac{1}{\sqrt{\vert (\phi_l, 1)\vert}}(\nabla X^{j-1}, \nabla\phi_l)\Delta_j\beta_l\nonumber\\
\leq&\frac{\eps^{\gamma+1}}{(d+1)^{-\frac{1}{2}}}\sum_{l=1}^L\frac{1}{\sqrt{\vert (\phi_l, 1)\vert}}\mathcal{E}(X^{j-1})(\nabla X^{j-1}, \nabla\phi_l)\Delta_j\beta_l\nonumber\\
&+ \frac{1}{32}\vert \mathcal{E}(X^j)-\mathcal{E}(X^{j-1})\vert^2+C\eps^{2\gamma+2}L\sum_{l=1}^L\frac{1}{\vert (\phi_l, 1)\vert}\vert (\nabla X^{j-1}, \nabla\phi_l)\vert^2\vert \Delta_j\beta_l\vert^2.
\end{align*}
By \lemref{Lemmabasis} and \eqref{Lowerenergy} we estimate
\begin{align}
\label{Ville1}
\mathcal{E}(X^{j})B_{1,2}\leq&\eps^{\gamma+1}\mathcal{E}(X^{j-1})(\nabla X^{j-1}, \nabla\Delta_jW)\nonumber\\
&+ \frac{1}{32}\vert \mathcal{E}(X^j)-\mathcal{E}(X^{j-1})\vert^2+C\eps^{2\gamma+2}h^{-2-d}\sum_{l=1}^L\Vert\nabla X^{j-1}\Vert^2\vert\Delta_j\beta_l\vert^2\nonumber\\
\leq&\eps^{\gamma+1}\mathcal{E}(X^{j-1})(\nabla X^{j-1}, \nabla\Delta_jW)\\
&+ \frac{1}{32}\vert \mathcal{E}(X^j)-\mathcal{E}(X^{j-1})\vert^2+C\eps^{2\gamma+2}h^{-2-d}\mathcal{E}(X^{j-1})\sum_{l=1}^L\vert\Delta_j\beta_l\vert^2.\nonumber
\end{align}
Similarly we get by  Young's inequality
\begin{align*}
\mathcal{E}(X^{j})B_{3}=&\frac{\eps^{\gamma-1}}{(d+1)^{-\frac{1}{2}}}\sum_{l=1}^L\frac{1}{\sqrt{\vert (\phi_l, 1)\vert}}\mathcal{E}(X^{j-1})(f( X^{j-1}), \phi_l)\Delta_j\beta_l\nonumber\\
&+\frac{\eps^{\gamma-1}\left(\mathcal{E}(X^j)-\mathcal{E}(X^{j-1})\right)}{(d+1)^{-\frac{1}{2}}}\sum_{l=1}^L\frac{1}{\sqrt{\vert (\phi_l, 1)\vert}}(f(X^{j-1}), \phi_l)\Delta_j\beta_l\nonumber\\
\leq&\frac{\eps^{\gamma-1}}{(d+1)^{-\frac{1}{2}}}\sum_{l=1}^L\frac{1}{\sqrt{\vert (\phi_l, 1)\vert}}\mathcal{E}(X^{j-1})(f(X^{j-1}), \phi_l)\Delta_j\beta_l\nonumber\\
&+ \frac{1}{32}\vert \mathcal{E}(X^j)-\mathcal{E}(X^{j-1})\vert^2+C\eps^{2\gamma-2}L\sum_{l=1}^L\frac{1}{\vert (\phi_l, 1)\vert}\vert (f(X^{j-1}), \phi_l)\vert^2\vert \Delta_j\beta_l\vert^2.
\end{align*}
Using \lemref{Lemmabasis}, Poincar\'{e}'s inequality, the fact that $f(u)=\mathfrak{f}(u)u$ and \eqref{Lowerenergy} we deduce
\begin{align}
\label{Ville2}
\mathcal{E}(X^{j})B_{3}\leq& \eps^{\gamma-1}\revl{\mathcal{E}(X^{j-1})}\left(f(X^{j-1}), \Delta_jW\right)+\frac{1}{32}\vert \mathcal{E}(X^j)-\mathcal{E}(X^{j-1})\vert^2\nonumber\\
&+C\eps^{2\gamma-2}h^{-2d}\vert\mathcal{E}(X^{j-1})\vert^2\sum_{l=1}^L\vert \Delta_j\beta_l\vert^2. 
\end{align}
Along the same lines as above one can show that
\begin{align}
\label{Ville3}
\mathcal{E}(X^j)A_2=&\eps^{\gamma-1}\mathcal{E}(X^j)\left(f(X^{j-1}), 1\right)m(\Delta_jW)\nonumber\\
\leq & \eps^{\gamma-1}\mathcal{E}(X^{j-1})\left(f(X^{j-1}), 1\right)m(\Delta_jW)+\frac{1}{32}\vert\mathcal{E}(X^j)-\mathcal{E}(X^{j-1})\vert^2\\
&+C\eps^{2\gamma-2}\vert\mathcal{E}(X^{j-1})\vert^2\vert m(\Delta_jW)\vert^2. \nonumber
\end{align}

Substituting \eqref{EstimateA0b}, \eqref{Ville1}, \eqref{Ville2} and \eqref{Ville3} in \eqref{energy1_hm} we obtain
{\small
\begin{align}
\label{bound10}
&\frac{1}{2}\left[\vert\mathcal{E}(X^{j})\vert^2-\vert\mathcal{E}(X^{j-1})\vert^2+\frac{3}{4}\vert\mathcal{E}(X^j)-\mathcal{E}(X^{j-1})\vert^2\right]\nonumber\\
&\leq C\eps^{2\gamma-3}\mathcal{E}(X^{j-1})\Vert\Delta_j\overline{W}\Vert^2+C\eps^{4\gamma-1}\mathcal{E}(X^{j-1})\vert m(\Delta_jW)\vert^4+C\eps^{2\gamma-1}\mathcal{E}(X^{j-1})\vert m(\Delta_jW)\vert^2\nonumber\\
&\quad+C\eps^{2\gamma-1}h^{-2-d}\mathcal{E}(X^{j-1})\sum_{l=1}^L\vert\Delta_j\beta_l\vert^2 +C\eps^{4\gamma-4}\vert\mathcal{E}(X^{j-1})\vert^2\vert\vert m(\Delta_jW)\vert^4\nonumber\\
&\quad+C\eps^{4\gamma-1}h^{-5d}\mathcal{E}(X^{j-1})\sum_{l=1}^L\vert\Delta_j\beta_l\vert^4+C\eps^{2\gamma-2}\vert\mathcal{E}(X^{j-1})\vert^2\vert m(\Delta_jW)\vert^2\nonumber\\
&\quad+C\eps^{2\gamma-3}h^{-2d}\vert\mathcal{E}(X^{j-1})\vert^2\sum_{l=1}^L\vert\Delta_j\beta_l\vert^2+C\eps^{4\gamma-6}h^{-5d}\vert\mathcal{E}(X^{j-1})\vert^2\sum_{l=1}^L\vert\Delta_j\beta_l\vert^4\\
&\quad+C\eps^{4\gamma-6}\Vert\Delta_j\overline{W}\Vert^4+C\eps^{8\gamma-2}\vert m(\Delta_jW)\vert^8+C\eps^{4\gamma-2}\vert m(\Delta_jW)\vert^4+C\eps^{4\gamma+2}h^{-5d}\sum_{l=1}^L\vert \Delta_j\beta_l\vert^4\nonumber\\
&\quad+C\eps^{2\gamma+2}h^{-2-d}\mathcal{E}(X^{j-1})\sum_{l=1}^L\vert\Delta_j\beta_l\vert^2+C\eps^{8\gamma-2}h^{-11d}\sum_{l=1}^L\vert\Delta_j\beta_l\vert^8\nonumber\\
&\quad+C\eps^{2\gamma-1}h^{-2d}\mathcal{E}(X^{j-1})\sum_{l=1}^L\vert\Delta_j\beta_l\vert^2+C\eps^{4\gamma-2}h^{-5d}\sum_{l=1}^L\vert\Delta_j\beta_l\vert^4+\frac{1}{8}\vert\mathcal{E}(X^j)-\mathcal{E}(X^{j-1})\vert^2\nonumber\\
&\quad+\eps^{\gamma+1}\mathcal{E}(X^{j-1})\left(\nabla X^{j-1}, \nabla \Delta_jW\right)+\eps^{\gamma-1}\mathcal{E}(X^{j-1})\left(f(X^{j-1}), \Delta_jW\right)\nonumber\\
&\quad+\eps^{\gamma-1}\mathcal{E}(X^{j-1})\left(f(X^{j-1}), 1\right)m(\Delta_jW).\nonumber
\end{align}
}
We estimate terms in \eqref{bound10} which are multiplied by $\mathcal{E}(X^{j-1})$ using Young's inequality. 
For instance we have
\revd{
\begin{align*}
\eps^{4\gamma-1}\mathcal{E}(X^{j-1})\vert m(\Delta_jW)\vert^4&\leq C \eps^{4\gamma-1}\left(\mathcal{E}(X^{j-1})^2+1\right)\vert m(\Delta_jW)\vert^4,\\
\eps^{2\gamma-1}\mathcal{E}(X^{j-1})\vert m(\Delta_jW)\vert^2&\leq C\eps^{2\gamma-1}\left(\mathcal{E}(X^{j-1})^2+1\right)\vert m(\Delta_jW)\vert^2,\\
\eps^{2\gamma-1}h^{-2-d}\mathcal{E}(X^{j-1})\sum_{l=1}^L\vert\Delta_j\beta_l\vert^2&\leq C\eps^{2\gamma-1}h^{-2-d}\left(\mathcal{E}(X^{j-1})^2+1\right)\sum_{l=1}^L\vert\Delta_j\beta_l\vert^2,
\end{align*}
}
and similarly for the remaining terms.
\revd{ The inequality \eqref{bound10} therefore becomes
{\small
\begin{align}
\label{bound10a}
&\frac{1}{2}\left[\vert\mathcal{E}(X^{j})\vert^2-\vert\mathcal{E}(X^{j-1})\vert^2+\frac{3}{4}\vert\mathcal{E}(X^j)-\mathcal{E}(X^{j-1})\vert^2\right]\nonumber\\
&\leq C\eps^{2\gamma-3}\left(\mathcal{E}(X^{j-1})^2+1\right)\Vert\Delta_j\overline{W}\Vert^2+C\eps^{4\gamma-1}\left(\mathcal{E}(X^{j-1})^2+1\right)\vert m(\Delta_jW)\vert^4\nonumber\\
&\quad+C\eps^{2\gamma-1}\left(\mathcal{E}(X^{j-1})^2+1\right)\vert m(\Delta_jW)\vert^2+C\eps^{2\gamma-1}h^{-2-d}\left(\mathcal{E}(X^{j-1})^2+1\right)\sum_{l=1}^L\vert\Delta_j\beta_l\vert^2\nonumber\\
&\quad +C\eps^{4\gamma-4}\vert\mathcal{E}(X^{j-1})\vert^2\vert\vert m(\Delta_jW)\vert^4+C\eps^{4\gamma-1}h^{-5d}\left(\mathcal{E}(X^{j-1})^2+1\right)\sum_{l=1}^L\vert\Delta_j\beta_l\vert^4\nonumber\\
&\quad+C\eps^{2\gamma-2}\vert\mathcal{E}(X^{j-1})\vert^2\vert m(\Delta_jW)\vert^2+C\eps^{2\gamma-3}h^{-2d}\vert\mathcal{E}(X^{j-1})\vert^2\sum_{l=1}^L\vert\Delta_j\beta_l\vert^2\nonumber\\
&\quad+C\eps^{4\gamma-6}h^{-5d}\vert\mathcal{E}(X^{j-1})\vert^2\sum_{l=1}^L\vert\Delta_j\beta_l\vert^4+C\eps^{4\gamma-6}\Vert\Delta_j\overline{W}\Vert^4+C\eps^{8\gamma-2}\vert m(\Delta_jW)\vert^8\\
&\quad+C\eps^{4\gamma-2}\vert m(\Delta_jW)\vert^4+C\eps^{4\gamma+2}h^{-5d}\sum_{l=1}^L\vert \Delta_j\beta_l\vert^4\nonumber\\
&\quad+C\eps^{2\gamma+2}h^{-2-d}\left(\mathcal{E}(X^{j-1})^2+1\right)\sum_{l=1}^L\vert\Delta_j\beta_l\vert^2+C\eps^{8\gamma-2}h^{-11d}\sum_{l=1}^L\vert\Delta_j\beta_l\vert^8\nonumber\\
&\quad+C\eps^{2\gamma-1}h^{-2d}\left(\mathcal{E}(X^{j-1})^2+1\right)\sum_{l=1}^L\vert\Delta_j\beta_l\vert^2+C\eps^{4\gamma-2}h^{-5d}\sum_{l=1}^L\vert\Delta_j\beta_l\vert^4\nonumber\\
&\quad+\eps^{\gamma+1}\mathcal{E}(X^{j-1})\left(\nabla X^{j-1}, \nabla \Delta_jW\right)+\eps^{\gamma-1}\mathcal{E}(X^{j-1})\left(f(X^{j-1}), \Delta_jW\right)\nonumber\\
&\quad+\eps^{\gamma-1}\mathcal{E}(X^{j-1})\left(f(X^{j-1}), 1\right)m(\Delta_jW).\nonumber
\end{align}
}
Along the same lines as those in the proof of \lemref{LemmaBruit0}, we have
\begin{align}
\label{NoiseEstimate1}
\mathbb{E}[\Vert\Delta_j\overline{W}\Vert^4]\leq Ch^{-4d}\tau^2,\quad \mathbb{E}[\vert m(\Delta_jW)\vert^8]\leq Ch^{-4d}\tau^4\quad \text{and}\quad \mathbb{E}[\Vert\Delta_j\overline{W}\Vert^8]\leq Ch^{-8d}\tau^4.
\end{align}
}

Summing  \eqref{bound10a} over $j$,  taking the expectation in both sides, using \eqref{NoiseEstimate1}, Lemmas \ref{Lemmabasis}, \ref{LemmaBruit0} we conclude that
\begin{align*}
&\frac{1}{2}\mathbb{E}[\mathcal{E}(X^{j})^2]+\frac{3}{8}\sum_{i=1}^{j}\mathbb{E}[\vert\mathcal{E}(X^i)-\mathcal{E}(X^{i-1})\vert^2]\nonumber\\
&\leq \mathcal{E}(u^{\eps}_0)^2+C\eps^{2\gamma-3}h^{-2-3d}+C\eps^{4\gamma-1}h^{-2d}\tau+C\eps^{2\gamma-2}+C\eps^{2\gamma-1}h^{-2-2d}+C\eps^{4\gamma-1}h^{-6d}\tau\nonumber\\
&\quad+C\eps^{4\gamma-6}h^{-4d}\tau+C\eps^{8\gamma-1}h^{-4d}\tau^3+C\eps^{4\gamma-2}h^{-2d}\tau+C\eps^{4\gamma+2}h^{-6d}\tau+C\eps^{2\gamma+2}h^{-2-2d}\nonumber\\
&\quad+ C\eps^{8\gamma-2}h^{-12d}\tau^3+C\eps^{2\gamma-1}h^{-3d}+C\eps^{4\gamma-2}h^{-6d}\tau\nonumber\\
&\quad+C\left[\eps^{4\gamma-2}h^{-2d}\tau+\eps^{2\gamma-2}+\eps^{2\gamma-3}h^{-3d}+\eps^{4\gamma-6}h^{-6d}\tau\right]\tau\sum_{i=0}^{j-1}\mathbb{E}[\mathcal{E}(X^{i})^2].\nonumber
\end{align*}
Applying the discrete Gronwall lemma to the preceding estimate yields the estimate \eqref{Allenergy}.

Then the estimate i) follows from \eqref{Allenergy} under the condition $h=\eps^\eta$.
 
The proof of the estimate ii)  follows analogously to i) by the  modified discrete Burkholder-Davis-Gundy inequality \cite[Lemma 3.3]{Banas19} and \lemref{moment}.
\end{proof}


\section{Error analysis}
\label{convergenceanalysis}
\revd{In this section we derive a robust estimate for the approximation error $X^j-u(t_j)$, where $X^j$ is the numerical approximation \eqref{scheme1b} of the strong variational solution $u$ of \eqref{model1}.
To show the error estimate we rewrite the error as
\begin{align*}
X^j-u(t_j)=\left(X^j-X^j_{\mathrm{CH}}\right)+\left(X^j_{\mathrm{CH}}-u_{\mathrm{CH}}(t_j)\right)+\left(u_{\mathrm{CH}}(t_j)-u(t_j)\right),
\end{align*}
and estimate the individual contributions on the right-hind side separately.
An estimate of $u_{\mathrm{CH}}(t_j)-u(t_j)$ is provided in \lemref{analyticCahnHilliard}. An estimate of $X^j_{\mathrm{CH}}-u_{\mathrm{CH}}(t_j)$ was shown in \cite[Corollary 1]{fp04} and is stated
in \lemref{LemmaLubo19} (iv) below.
Here we estimate the remaining term $Z^j:=X^j-X^j_{\mathrm{CH}}$ in \lemref{EstiZ} which allows us to conclude the desired error estimate in \thmref{mainresult1} by the triangle inequality.

In the lemma below we recall the properties of the numerical approximation $X^j_{\mathrm{CH}}$
of the deterministic problem (i.e., $X^j_{\mathrm{CH}}$ satisfies \eqref{scheme1b} with $\Delta_j\overline{W}\equiv 0$) from \cite[Lemma 3.1]{Banas19}.
\begin{lemma}
\label{LemmaLubo19}
Assume that $\mathcal{E}(u^{\epsilon}_0)\leq C$. Let 
$\{(X^j_{\mathrm{CH}}, w^j_{\mathrm{CH}})\}_{j=0}^J\subset[\mathbb{H}^1]^2$ be the solution of \eqref{scheme1b} with $\Delta_j\overline{W}\equiv 0$. For every $0<\beta<\frac{1}{2}$, $\varepsilon\in(0, \varepsilon_0)$, $\tau\leq \varepsilon^3$, and $\mathfrak{p}_{\mathrm{CH}}>0$, there exist $\mathfrak{m}_{\mathrm{CH}}, \mathfrak{n}_{\mathrm{CH}}, C>0$, and $\mathfrak{l}_{\mathrm{CH}}\geq 3$ such that
\begin{enumerate}
\item[(i)] $\max\limits_{1\leq j\leq J}\mathcal{E}(X^j_{\mathrm{CH}})\leq \mathcal{E}(u^{\varepsilon}_0)$. 
\end{enumerate}
 Assume moreover that $\Vert u^{\varepsilon}_0\Vert_{\mathbb{H}^2}\leq C\varepsilon^{-\mathfrak{p}_{\mathrm{CH}}} $. Then
 \begin{enumerate}
\item[(ii)] $
\max\limits_{1\leq j\leq J}\Vert X^j_{\mathrm{CH}}\Vert_{\mathbb{H}^2}\leq C\varepsilon^{-\mathfrak{n}_{\mathrm{CH}}},
$
\item[(iii)] $\max\limits_{1\leq j\leq J}\Vert X^j_{\mathrm{CH}}\Vert_{\mathbb{L}^{\infty}}\leq C$ for $\tau\leq C\varepsilon^{\mathfrak{l}_{\mathrm{CH}}}$.
\end{enumerate}
Assume in addition $\Vert u^{\varepsilon}_0\Vert_{\mathbb{H}^3}\leq C\varepsilon^{-\mathfrak{p}_{\mathrm{CH}}}$ and let $u_{\mathrm{CH}}$ be the unique solution of the deterministic Cahn-Hilliard equation. Then for $\tau\leq C\varepsilon^{\mathfrak{l}_{\mathrm{CH}}}$ and $C_0$ from \eqref{Spectral1} it holds
\begin{enumerate}
\item[(iv)] $\max\limits_{1\leq j\leq J}\Vert u_{\mathrm{CH}}(t_j)-X^j_{\mathrm{CH}}\Vert^2_{\mathbb{H}^{-1}}+\sum\limits_{j=1}^J\tau^{1+\beta}\Vert\nabla[u_{\mathrm{CH}}(t_j)-X^j_{\mathrm{CH}}]\Vert^2\leq C\frac{\tau^{2-\beta}}{\varepsilon^{\mathfrak{m}_{\mathrm{CH}}}}$,
 \item[(v)] $\inf\limits_{0\leq t\leq T}\inf\limits_{\psi\in\mathbb{H}^1, w=(-\Delta)^{-1}\psi}\frac{\varepsilon\Vert\nabla\psi\Vert^2+\frac{1-\varepsilon^3}{\varepsilon}\left(f'(X^j_{\mathrm{CH}})\psi, \psi\right)}{\Vert\nabla w\Vert^2}\geq -(1-\varepsilon^3)(C_0+1)$.
\end{enumerate}
\end{lemma}
}

 We start by deriving an $\mathbb{P}$-a.s. {a priori} error estimate for $Z^j=X^j-X^j_{\mathrm{CH}}$.
\begin{lemma}
\label{lemmaerror1}
The following estimate holds for all $l= 1,\cdots, J$
\begin{align*}
&\max_{1\leq j\leq l}\Vert\Delta^{-1/2}Z^j\Vert^2+\frac{\eps^4\tau}{2}\sum_{j=1}^l\Vert\nabla Z^j\Vert^2+\frac{\tau}{\eps}\sum_{j=1}^l\Vert Z^j\Vert^4_{\mathbb{L}^4}+\frac{1}{4}\sum_{j=1}^l\Vert\Delta^{-1/2}(Z^j-Z^{j-1})\Vert^2\nonumber\\
&\leq \frac{C\tau}{\eps}\sum_{j=1}^l\Vert Z^j\Vert^3_{\mathbb{L}^3}+\eps^{\gamma}\max_{1\leq j\leq l}\left\vert\sum_{i=1}^j((-\Delta)^{-1}Z^{i-1}, \Delta_i\overline{W})\right\vert+C\eps^{2\gamma}\sum_{j=1}^l\Vert\Delta_j\overline{W}\Vert^2.
\end{align*}
\end{lemma}

\begin{proof}
We take $\varphi=(-\Delta)^{-1}Z^j(\omega)$ and $\psi=Z^j(\omega)$ in  \eqref{scheme1b} for fixed $\omega\in\Omega$ and  obtain  $\mathbb{P}$-a.s.
\begin{align}
\label{essen1}
&\frac{1}{2}\left(\Vert \Delta^{-1/2}Z^j\Vert^2-\Vert\Delta^{-1/2}Z^{j-1}\Vert^2+\Vert\Delta^{-1/2}(Z^j-Z^{j-1})\Vert^2\right)+\eps\tau\Vert\nabla Z^j\Vert^2\nonumber\\
&+\frac{\tau}{\eps}\left(f(X^j)-f(X^j_{\mathrm{CH}}), Z^j\right)=\eps^{\gamma}((-\Delta)^{-1}Z^j, \Delta_j\overline{W}).
\end{align}
To handle the term $\left(f(X^j)-f(X^j_{\mathrm{CH}}), Z^j\right)$, we use the fact that $f'(u)=3u^2-1$ (which implies $(f'(u)v,v)\geq -\Vert v\Vert^2$) to obtain
\begin{align*}
\left(f'(X^j_{\mathrm{CH}})Z^j, Z^j\right)
&=(1-\eps^3)(f'(X^j_{\mathrm{CH}})Z^j, Z^j)+\eps^3(f'(X^j_{\mathrm{CH}})Z^j, Z^j)\nonumber\\
&\geq (1-\eps^3)(f'(X^j_{\mathrm{CH}})Z^j, Z^j)-\eps^3\Vert Z^j\Vert^2.
\end{align*}
Using \eqref{nonlinearident} and  the preceding estimate, it follows that
\begin{align}
\label{NonlinS1}
\left(f(X^j)-f(X^j_{\mathrm{CH}}), Z^j\right)&=\left(f(X^j_{\mathrm{CH}})-f(X^j), X^j_{\mathrm{CH}}-X^j\right)\nonumber\\
&=(f'(X^j_{\mathrm{CH}})Z^j, Z^j)+((Z^j)^3, Z^j)+3((Z^j)^3, X^j_{\mathrm{CH}})\nonumber\\
&\geq (1-\eps^3)(f'(X^j_{\mathrm{CH}})Z^j, Z^j)-\eps^3\Vert Z^j\Vert^2+3((Z^j)^3, X^j_{\mathrm{CH}})+\Vert Z^j\Vert^4_{\mathbb{L}^4}.
\end{align}
 Using \lemref{LemmaLubo19} (v) yields
\begin{align}
\label{NonlinS2}
&\eps\Vert \nabla Z^j\Vert^2+\frac{(1-\eps^3)}{\eps}(f'(X^j_{\mathrm{CH}})Z^j, Z^j)\nonumber\\
&=(1-\eps^3)\left(\eps\Vert \nabla Z^j\Vert^2+\frac{(1-\eps^3)}{\eps}(f'(X^j_{\mathrm{CH}})Z^j, Z^j)\right)\\
&\quad+\eps^3\left(\eps\Vert \nabla Z^j\Vert^2+\frac{(1-\eps^3)}{\eps}(f'(X^j_{\mathrm{CH}})Z^j, Z^j)\right)\nonumber\\
&\geq -(C_0+1)\Vert\Delta^{-1/2} Z^j\Vert^2+\eps^3\left(\eps\Vert \nabla Z^j\Vert^2+\frac{(1-\eps^3)}{\eps}(\revd{f'(X^j_{\mathrm{CH}})Z^j, Z^j)}\right)\nonumber\\
&\revd{=}-(C_0+1)\Vert \Delta^{-1/2}Z^j\Vert^2+\eps^4\Vert\nabla Z^j\Vert^2+\eps^2(1-\eps^3)(f'(X^j_{\mathrm{CH}})Z^j, Z^j), \nonumber
\end{align}
where we have used the fact that $\eps\in (0,1)$. 

Substituting   \eqref{NonlinS2} into \eqref{NonlinS1}  and substituting the resulting estimate into \eqref{essen1} yields
\begin{align}
\label{relec1}
&\frac{1}{2}\left(\Vert \Delta^{-1/2}Z^j\Vert^2-\Vert\Delta^{-1/2} Z^{j-1}\Vert^2+\Vert\Delta^{-1/2}(Z^j-Z^{j-1})\Vert^2\right)+\eps^4\tau\Vert\nabla Z^j\Vert^2+\frac{\tau}{\eps}\Vert Z^j\Vert^4_{\mathbb{L}^4}\nonumber\\
&\leq 2\eps^2\tau\Vert Z^j\Vert^2+C\tau\Vert\Delta^{-1/2}Z^j\Vert^2+\frac{3\tau}{\eps}\vert((Z^j)^3, X^j_{\mathrm{CH}})\vert+\eps^{\gamma}((-\Delta)^{-1}Z^j, \Delta_j\overline{W}),
\end{align}
where we have used the fact that $(f'(u)v, v)\geq -\Vert v\Vert^2$, see e.g. \cite[(2.5)]{fp05}.

Using the uniformly boundedness of $X^j_{\mathrm{CH}}$ (cf. \lemref{LemmaLubo19} (iii)), it holds that
\begin{align*}
\frac{3\tau}{\eps}\vert((Z^j)^3, X^j_{\mathrm{CH}})\vert\leq \frac{C\tau}{\eps}\Vert Z^j\Vert^3_{\mathbb{L}^3}.
\end{align*}
Next, using the interpolating inequality $\Vert .\Vert^2\leq \Vert .\Vert_{\mathbb{H}^{-1}}\Vert \nabla .\Vert$ and Young's inequality leads to
\begin{align*}
\eps^2\Vert Z^j\Vert^2\leq \eps^2\Vert \Delta^{-1/2}Z^j\Vert\Vert \nabla Z^j\Vert
\leq  C\Vert \Delta^{-1/2}Z^j\Vert^2+\frac{\eps^4}{2}\Vert \nabla Z^j\Vert^2. 
\end{align*}
Using Cauchy-Schwarz's inequality and \eqref{equiv2}, we obtain
\begin{align*}
\eps^{\gamma}((-\Delta)^{-1}Z^j, \Delta_{j}\overline{W})&\leq\eps^{\gamma}\Vert \Delta^{-1}(Z^j-Z^{j-1})\Vert\Vert \Delta_j\overline{W}\Vert+\eps^{\gamma}((-\Delta)^{-1}Z^{j-1}, \Delta_j\overline{W})\nonumber\\
&\leq\eps^{\gamma}\Vert \Delta^{-1/2}(Z^j-Z^{j-1})\Vert\Vert \Delta_j\overline{W}\Vert+\eps^{\gamma}((-\Delta)^{-1}Z^{j-1}, \Delta_j\overline{W})\nonumber\\
&\leq \revd{\frac{1}{4}}\Vert \Delta^{-1/2}(Z^j-Z^{j-1})\Vert^2+C\eps^{2\gamma}\Vert\Delta_j\overline{W}\Vert^2+\eps^{\gamma}((-\Delta)^{-1}Z^{j-1}, \Delta_j\overline{W}).
\end{align*}
 Substituting  the two preceding estimates into \eqref{relec1} leads to
\begin{align*}
&\frac{1}{2}\left(\Vert \Delta^{-1/2}Z^j\Vert^2-\Vert\Delta^{-1/2} Z^{j-1}\Vert^2+\frac{1}{2}\Vert\Delta^{-1/2}(Z^j-Z^{j-1})\Vert^2\right)+\frac{\eps^4\tau}{2}\Vert\nabla Z^j\Vert^2+\frac{\tau}{\eps}\Vert Z^j\Vert^4_{\mathbb{L}^4}\nonumber\\
&\leq\frac{C\tau}{\eps}\Vert Z^j\Vert^3_{\mathbb{L}^3}+C\tau\Vert\Delta^{-1/2}Z^j\Vert^2+\eps^{\gamma}((-\Delta)^{-1}Z^{j-1}, \Delta_j\overline{W})+C\eps^{2\gamma}\Vert\Delta_j\overline{W}\Vert^2.
\end{align*}
Summing the preceding estimate over $1\leq j\leq l$ and  taking the maximum yields 
\begin{align}
\label{Val1}
&\frac{1}{2}\max_{1\leq j\leq l}\Vert\Delta^{-1/2}Z^j\Vert^2+\frac{\eps^4\tau}{2}\sum_{j=1}^l\Vert\nabla Z^j\Vert^2+\frac{\tau}{\eps}\sum_{j=1}^l\Vert Z^j\Vert^4_{\mathbb{L}^4}+\frac{1}{4}\sum_{j=1}^l\Vert\Delta^{-1/2}(Z^j-Z^{j-1})\Vert^2\nonumber\\
&\leq \frac{C\tau}{\eps}\sum_{j=1}^l\Vert Z^j\Vert^3_{\mathbb{L}^3}+C\tau\sum_{j=1}^l\max_{1\leq i\leq j}\Vert \Delta^{-1/2}Z^i\Vert^2+\eps^{\gamma}\max_{1\leq j\leq l}\left\vert\sum_{i=1}^j((-\Delta)^{-1}Z^{i-1}, \Delta_i\overline{W})\right\vert\\
&\quad+C\eps^{2\gamma}\sum_{j=1}^l\Vert\Delta_j\overline{W}\Vert^2,\nonumber
\end{align}
where we have used the fact that $Z^0=0$. For  $1\leq l\leq J$, we set
\begin{align}
\label{Al}
\mathcal{A}_l&:=\frac{1}{2}\max_{1\leq j\leq l}\Vert\Delta^{-1/2}Z^j\Vert^2+\frac{\eps^4\tau}{2}\sum_{j=1}^l\Vert\nabla Z^j\Vert^2+\frac{\tau}{\eps}\sum_{j=1}^l\Vert Z^j\Vert^4_{\mathbb{L}^4}\nonumber\\
&\quad+\revd{\frac{1}{4}\sum_{j=1}^l\Vert\Delta^{-1/2}(Z^j-Z^{j-1})\Vert^2},\\
\label{Rl}
\mathcal{R}_l&:=\frac{\tau}{\eps}\sum_{j=1}^l\Vert Z^j\Vert^3_{\mathbb{L}^3}+\eps^{\gamma}\max_{1\leq j\leq l}\left\vert\sum_{i=1}^j((-\Delta)^{-1}Z^{i-1}, \Delta_i\overline{W})\right\vert+C\eps^{2\gamma}\sum_{j=1}^l\Vert\Delta_j\overline{W}\Vert^2.
\end{align}
It therefore follows from \eqref{Val1} that
\revd{
\begin{align}
\label{Val2}
\mathcal{A}_l\leq C\mathcal{R}_l+C\tau\sum_{j=1}^l\mathcal{A}_j\quad \mathbb{P}\text{-a.s.}\quad \forall 1\leq l\leq J.
\end{align} 
}
 Applying the implicit discrete Gronwall lemma to \eqref{Val2} yields the desired result, for $\tau$ small enough.
\end{proof}

\begin{remark}
\label{Difficulty}
One of the difficulties in estimating the error $Z^j$ directly from \lemref{lemmaerror1} is  the presence of the cubic term on the right hand side.  To handle this issue, we introduce a discrete stopping time (or stopping index) $1\leq J_{\eps}\leq J$:
\begin{align}
\label{Stopindex}
J_{\eps}:=\inf\left\{1\leq j\leq J:\; \frac{\tau}{\eps}\sum_{i=1}^j\Vert Z^i\Vert^3_{\mathbb{L}^3}>\eps^{\sigma_0}\right\},
\end{align}
where $\sigma_0>0$ is a constant which will be specified later. The purpose of the stopping index $J_{\eps}$ is to identify those $\omega\in \Omega$ for which the cubic term is small enough. We estimate the right-hand side of the inequality in \lemref{lemmaerror1} for $l=J_{\eps}$ on a probability subset $\Omega_2$ (defined in  \eqref{SetOmega2}) on which the cubic term is small enough.
Then we conclude that $J_{\eps}=J$ on $\Omega_2$ and that $\lim_{\eps\rightarrow 0}\mathbb{P}[\Omega_2]=1$.

The term $\frac{\tau}{\eps}\sum_{j=1}^{J_{\eps}-1}\Vert Z^j\Vert^3_{\mathbb{L}^3}$ of $\mathcal{R}_{J_{\eps}}$ in \eqref{Rl} is bounded above by $\eps^{\sigma_0}$.  We denote the remaining part by $\widetilde{\mathcal{R}}_{J_{\eps}}:=\mathcal{R}_{J_{\eps}}-\frac{\tau}{\eps}\sum_{j=1}^{J_{\eps}-1}\Vert Z^j\Vert^3_{\mathbb{L}^3}$, that is, 
\begin{align}
\label{Restetilde}
\widetilde{\mathcal{R}}_{J_{\eps}}=\frac{\tau}{\eps}\Vert Z^{J_{\eps}}\Vert^3_{\mathbb{L}^3}+\eps^{\gamma}\max_{1\leq j\leq J_{\eps}}\left\vert\sum_{i=1}^j((-\Delta)^{-1}Z^{i-1}, \Delta_i\overline{W})\right\vert+C\eps^{2\gamma}\sum_{j=1}^{J_{\eps}}\Vert\Delta_j\overline{W}\Vert^2.
\end{align}
For some $0<\kappa_0<\sigma_0$, we introduce the following subset of $\Omega$:
\begin{align}
\label{SetOmega2}
\Omega_2:=\{\omega\in\Omega:\: \widetilde{\mathcal{R}}_{J_{\eps}}(\omega)\leq \eps^{\kappa_0}\}. 
\end{align}
The set $\Omega_2\subseteq\Omega$ contains those $\omega\in \Omega$ for which the remainder $\widetilde{\mathcal{R}}_{J_{\eps}}$  does not exceed the  threshold $\eps^{\kappa_0}$.  We will  show that for an appropriate $\kappa_0$, the  subset $\Omega_2$ has high  probability as $\eps\rightarrow 0$, that is, $\lim\limits_{\eps\rightarrow 0}\mathbb{P}[\Omega_2]=1$. To sum up, our strategy is the following:
\begin{itemize}
\item[(i)] we estimate $\mathbb{P}[\Omega_2]$ and  the left hand side of \lemref{lemmaerror1} on $\Omega_2$ up to $J_{\eps}$, see \lemref{Omega2lemma}, 
\item[(ii)] we prove that on $\Omega_2$, it holds $J_{\eps}=J$, see \lemref{Mepsilonlemma}, 
\item[(iii)] we use the identity $\mathbb{E}[\mathcal{A}_J]=\mathbb{E}[\displaystyle{1\!\!1_{\Omega_2}}\mathcal{A}_J]+\mathbb{E}[\displaystyle{1\!\!1_{\Omega_2^c}}\mathcal{A}_J]$, (i) and (ii) to obtain  error estimate for $Z^j$, see \lemref{EstiZ}. 
\end{itemize} 
\end{remark}
We show the (i) in  \lemref{Omega2lemma} below under the following additional assumption.
\begin{Assumption} 
\label{assumption2}
Let $\gamma>\frac{5}{2}$, $0<\eps_0\ll 1$, $\eps\in (0, \eps_0)$,  $\tau\leq \frac{1}{2}\eps^3$ and $h=\eps^{\eta}$, with
\begin{align*}
\revd{
0<\eta\leq\min\left\{\frac{2\gamma-3}{2+3d}, \frac{2\gamma-6}{3d} \right\}}.
\end{align*}
\end{Assumption}
\begin{lemma}
\label{Omega2lemma}
Let  \assref{assumption2}  and the assumptions in \lemref{LemmaLubo19} be fulfilled, let $0<\kappa_0<\sigma_0$.
  Then  it hold that
\begin{itemize}
\item[(i)] $\displaystyle\max\limits_{1\leq i\leq J_{\eps}}\Vert \Delta^{-1/2}Z^i\Vert^2+\frac{\eps^4\tau}{2}\sum_{i=1}^{J_{\eps}}\Vert\nabla Z^i\Vert^2+\frac{\tau}{\eps}\sum_{i=1}^{J_{\eps}}\Vert Z^j\Vert^4_{\mathbb{L}^4}\leq C\eps^{\kappa_0}$ on $\Omega_2$,
\item[(ii)] 
 $\mathbb{P}[\Omega_2]\geq 1-\frac{C}{\eps^{\kappa_0}}\max\left(\eps^{\sigma_0}, \eps^{2\gamma-2d\eta}\tau,\eps^{\gamma-d\eta+\frac{\sigma_0+1}{3}},\frac{\tau^2}{\eps^{4}}\right)$. 
\end{itemize} 
\end{lemma}

\begin{proof}
 The proof of  (i) follows from the definition of the subset $\Omega_2$ \eqref{SetOmega2} and \lemref{lemmaerror1}. It remains to prove (ii). 
We recall that $\mathbb{P}[\Omega_2]=1-\mathbb{P}[\Omega_2^c]$. By Markov's inequality we have $\mathbb{P}[\Omega_2^c]\leq \frac{1}{\eps^{\kappa_0}}\mathbb{E}[\widetilde{\mathcal{R}}_{J_{\eps}}]$. It therefore remains to estimate $\mathbb{E}[\widetilde{\mathcal{R}}_{J_{\eps}}]$.
Using Young's inequality, it follows that $\vert v\vert^3=\vert v\vert^2\vert v\vert\leq \frac{1}{4}\vert v\vert^4+16\vert v\vert^2$ for all $v\in \mathbb{R}$.  This leads to
\begin{align}
\label{Interpol0}
\Vert v\Vert^3_{\mathbb{L}^3}\leq \frac{1}{4}\Vert v\Vert^4_{\mathbb{L}^4}+16\Vert  v\Vert^2,\quad v\in \mathbb{L}^4.
\end{align}
To handle the cubic term  in $\widetilde{\mathcal{R}}_{J_{\eps}}$ \eqref{Restetilde}, we employ \eqref{Interpol0},  the interpolation inequality $\Vert u\Vert^2_{\mathbb{L}^2}\leq \Vert u\Vert_{\mathbb{H}^{-1}}\Vert\nabla u\Vert_{\mathbb{L}^2}$ and  Young's inequality.  This leads to 
\begin{align}
\label{cubic1}
\frac{\tau}{\eps}\Vert Z^{J_{\eps}}\Vert^3_{\mathbb{L}^3}\leq \frac{\tau}{4\eps}\Vert Z^{J_{\eps}}\Vert^4_{\mathbb{L}^4}+\frac{1}{8}\Vert\Delta^{-1/2}Z^{J_{\eps}}\Vert^2+\frac{C\tau^2}{\eps^2}\Vert\nabla Z^{J_{\eps}}\Vert^2. 
\end{align}
From \lemref{lemmaerror1}, splitting the sum involving the cubic terms in  $\mathcal{R}_{J_{\eps}}$ \eqref{Restetilde}, employing  \eqref{cubic1} and using the definition   of $J_{\eps}$ \eqref{Stopindex} yields the following estimate  of $A_{J_{\eps}}$ \eqref{Al}
\begin{align*}
A_{J_{\eps}}
&\leq \frac{\tau}{\eps}\sum_{i=1}^{J_{\eps}-1}\Vert Z^i\Vert^3_{\mathbb{L}^3}+\frac{\tau}{\eps}\Vert Z^{J_{\eps}}\Vert^3_{\mathbb{L}^3}+\eps^{\gamma}\max_{1\leq j\leq J_{\eps}}\left\vert\sum_{i=1}^j((-\Delta)^{-1}Z^{i-1}, \Delta_i\overline{W})\right\vert+C\eps^{2\gamma}\sum_{j=1}^{J_{\eps}}\Vert\Delta_j\overline{W}\Vert^2\nonumber\\
&\leq C\eps^{\sigma_0}+\frac{1}{8}\Vert \Delta^{-1/2}Z^{J_{\eps}}\Vert^2+\frac{\tau}{4\eps}\Vert Z^{J_{\eps}}\Vert^4_{\mathbb{L}^4}+\frac{C\tau^2}{\eps^2}\Vert\nabla Z^{J_{\eps}}\Vert^2\nonumber\\
&\quad+\eps^{\gamma}\max_{1\leq j\leq J_{\eps}}\left\vert\sum_{i=1}^j((-\Delta)^{-1}Z^{i-1}, \Delta_i\overline{W})\right\vert+C\eps^{2\gamma}\sum_{i=1}^{J_{\eps}}\Vert\Delta_i\overline{W}\Vert^2.
\end{align*}
Absorbing  $\frac{1}{8}\Vert \Delta^{-1/2}Z^{J_{\eps}}\Vert^2$ and $\frac{\tau}{4\eps}\Vert Z^{J_{\eps}}\Vert^4_{\mathbb{L}^4}$ in the left hand side of the above estimate, taking the expectation in the resulting inequality and using \lemref{MomentLemma} ii)  yields
\begin{align}
\label{EstiAM1}
\mathbb{E}[\frac{1}{2}A_{J_{\eps}}]
\leq C\eps^{\sigma_0}+C\eps^{2\gamma-d\eta}+\frac{C\tau^2}{\eps^4}+\eps^{\gamma}\mathbb{E}\left[\max_{1\leq j\leq J_{\eps}}\left\vert\sum_{i=1}^j((-\Delta)^{-1}Z^{i-1}, \Delta_i\overline{W})\right\vert\right].
\end{align}
To estimate the last term in \eqref{EstiAM1}, we first use triangle inequality to split it as
\begin{align*}
\mathbb{E}\left[\max_{1\leq j\leq J_{\eps}}\vert\sum_{i=1}^j((-\Delta)^{-1}Z^{i-1}, \Delta_i\overline{W})\vert\right]&\leq \mathbb{E}\left[\max_{1\leq j\leq J_{\eps}}\left\vert\sum_{i=1}^j((-\Delta)^{-1}Z^{i-1}, \Delta_iW)\right\vert\right]\nonumber\\
&\quad+\mathbb{E}\left[\max_{1\leq j\leq J_{\eps}}\left\vert\sum_{i=1}^j\left((-\Delta)^{-1}Z^{i-1}, m(\Delta_iW)\right)\right\vert\right]. 
\end{align*} 
Using the expression of $\Delta_iW$ \eqref{Noiseapprox1},  \lemref{Lemmabasis} and \assref{assumption2} yields
\begin{align}
\label{maxW1}
&\eps^{\gamma}\mathbb{E}\left[\max_{1\leq j\leq J_{\eps}}\left\vert\sum_{i=1}^j((-\Delta)^{-1}Z^{i-1}, \Delta_iW)\right\vert\right]\nonumber\\
&\leq C\eps^{\gamma-\frac{d\eta}{2}}\sum_{l=1}^L\mathbb{E}\left[\max_{1\leq j\leq J_{\eps}}\left\vert\sum_{i=1}^j((-\Delta)^{-1}Z^{i-1}, \phi_l)\Delta_i\beta_l\right\vert\right].
\end{align}
Using the discrete Burkholder-Davis-Gundy inequality \cite[Lemma 3.3]{Banas19}, \eqref{equiv1} and  \eqref{equiv2} yields
\begin{align*}
&\mathbb{E}\left[\max_{1\leq j\leq J_{\eps}}\left\vert\sum_{i=1}^j((-\Delta)^{-1}Z^{i-1}, \phi_l)\Delta_i\beta_l\right\vert\right]\nonumber\\
&\leq C\mathbb{E}\left[\sum_{i=1}^{J_{\eps}+1}\tau((-\Delta)^{-1}Z^{i-1}, \phi_l)^2\right]^{\frac{1}{2}}\leq C\mathbb{E}\left[\sum_{i=1}^{J_{\eps}+1}\tau\Vert(-\Delta)^{-1}Z^{i-1}\Vert^2 \Vert\phi_l\Vert^2\right]^{\frac{1}{2}}\nonumber\\
&\leq Ch^{\frac{d}{2}}\mathbb{E}\left[\sum_{i=1}^{J_{\eps}+1}\tau\Vert\Delta^{-1/2}Z^{i-1}\Vert^2\right]^{\frac{1}{2}} \revd{ \leq} Ch^{\frac{d}{2}}\mathbb{E}\left[\sum_{i=1}^{J_{\eps}}\tau\Vert Z^{i-1}\Vert^2\right]^{\frac{1}{2}}+Ch^{\frac{d}{2}}\mathbb{E}\left[\tau\Vert\Delta^{-1/2}Z^{J_{\eps}}\Vert^2\right]^{\frac{1}{2}}.
\end{align*}
Using  the embedding $\mathbb{L}^3\hookrightarrow\mathbb{L}^2$,  H\"{o}lder's inequality  and the definition of $J_{\eps}$ \eqref{Stopindex} we obtain
\begin{align}
\label{Dadis1}
&\mathbb{E}\left[\max_{1\leq j\leq J_{\eps}}\left\vert\sum_{i=1}^j((-\Delta)^{-1}Z^{i-1}, \phi_l)\Delta_i\beta_l\right\vert\right]\nonumber\\
&\leq Ch^{\frac{d}{2}}\mathbb{E}\left[\tau\left(\sum_{i=1}^{J_{\eps}-1}\Vert Z^i\Vert^3_{\mathbb{L}^3}\right)^{\frac{2}{3}}\left(\sum_{i=1}^{J_{\eps}-1}1^3\right)^{\frac{1}{3}}\right]^{\frac{1}{2}}+Ch^{\frac{d}{2}}\tau^{\frac{1}{2}}\left(\mathbb{E}\Vert \Delta^{-1/2}Z^{J_{\eps}}\Vert^2\right)^{\frac{1}{2}}\nonumber\\
&\leq Ch^{\frac{d}{2}}\eps^{\frac{\sigma_0+1}{3}}+Ch^{\frac{d}{2}}\tau^{\frac{1}{2}}\left(\mathbb{E}\Vert \Delta^{-1/2}Z^{J_{\eps}}\Vert^2\right)^{\frac{1}{2}}.
\end{align}
Substituting \eqref{Dadis1} in \eqref{maxW1}, using \lemref{Lemmabasis} and the fact that $h=\eps^{\eta}$ yields
\begin{align}
\label{Yan1}
\eps^{\gamma}\mathbb{E}\left[\max_{1\leq j\leq J_{\eps}}\left\vert\sum_{i=1}^j((-\Delta)^{-1}Z^{i-1}, \Delta_iW)\right\vert\right]&\leq C\eps^{\gamma-d\eta+\frac{\sigma_0+1}{3}}+ C\eps^{\gamma-d\eta}\tau^{\frac{1}{2}}\left(\mathbb{E}\Vert \Delta^{-1/2}Z^{J_{\eps}}\Vert^2\right)^{\frac{1}{2}}\nonumber\\
&\leq C\eps^{\gamma-d\eta+\frac{\sigma_0+1}{3}}+C\eps^{2\gamma-2d\eta}\tau+\frac{1}{8}\mathbb{E}\Vert\Delta^{-1/2}Z^{J_{\eps}}\Vert^2.
\end{align}
Along the same lines as in the preceding estimate, we obtain 
\begin{align}
\label{Yan2}
&\eps^{\gamma}\mathbb{E}\left[\max\limits_{1\leq j\leq J_{\eps}}\left\vert\sum\limits_{i=1}^j\left((-\Delta)^{-1}Z^{i-1}, m(\Delta_iW)\right)\right\vert\right]\nonumber\\
&\leq  C\eps^{\gamma-d\eta+\frac{\sigma_0+1}{3}}+C\eps^{2\gamma-2d\eta}\tau+\frac{1}{8}\mathbb{E}\Vert\Delta^{-1/2}Z^{J_{\eps}}\Vert^2.
\end{align}  
From the expression of $\mathcal{A}_{J_{\eps}}$ \eqref{Al},  using  \lemref{lemmaerror1}, \eqref{Yan1}, \eqref{Yan2} and  \lemref{MomentLemma} yields
\begin{align}
\label{Yan3}
&\mathbb{E}[\Vert \Delta^{-1/2}Z^{J_{\eps}}\Vert^2]+\frac{\eps^4\tau}{2}\sum_{i=i}^{J_{\eps}}\mathbb{E}[\Vert\nabla Z^i\Vert^2]+\frac{3\tau}{4\eps}\sum_{i=1}^{J_{\eps}}\mathbb{E}[\Vert Z^i\Vert^4_{\mathbb{L}^4}]\nonumber\\
&\leq C\max\left(\eps^{\sigma_0}, \eps^{2\gamma-2d\eta}\tau,\eps^{\gamma-d\eta+\frac{\sigma_0+1}{3}},\frac{\tau^2}{\eps^{4}}\right).
\end{align} 
Substituting \eqref{Yan3} in \eqref{cubic1} and using \lemref{MomentLemma} yields
\begin{align}
\label{Yan4}
\frac{\tau}{\eps}\mathbb{E}\left[\Vert Z^{J_{\eps}}\Vert^3_{\mathbb{L}^3}\right]\leq  C\max\left(\eps^{\sigma_0}, \eps^{2\gamma-2d\eta}\tau,\eps^{\gamma-d\eta+\frac{\sigma_0+1}{3}},\frac{\tau^2}{\eps^{4}}\right).
\end{align}
Substituting \eqref{Yan2} in the expression of $\mathcal{\widetilde{R}}_{J_{\eps}}$ \eqref{Restetilde}, using \eqref{Yan3}, \eqref{Yan4} and \lemref{LemmaBruit0} leads to
\begin{align*}
\mathbb{E}[\mathcal{\widetilde{R}}_{J_{\eps}}]\leq C\max\left(\eps^{\sigma_0}, \eps^{2\gamma-2d\eta}\tau,\eps^{\gamma-d\eta+\frac{\sigma_0+1}{3}},\frac{\tau^2}{\eps^{4}}\right). 
\end{align*}
This completes the proof of the lemma. 
\end{proof}

We prove in \lemref{Mepsilonlemma} below that $J_{\eps}=J$ on $\Omega_2$ and that  $\mathbb{P}[\Omega_2]$ goes to $1$ as $\eps\rightarrow 0$.

\begin{lemma}
\label{Mepsilonlemma}
Let \assref{assumption2} be fulfilled. Assume that for fixed \revl{$0 < \alpha < 7$, $2<  \delta \leq \frac{8}{3}$ the parameters $(\sigma_0, \kappa_0)$ satisfy} 
\begin{align*}
\revd{\sigma_0>\frac{4\delta-7}{\delta-1}+\frac{\alpha(3-\delta)}{\delta-1}}\quad\text{and}\quad \revd{\sigma_0>} \kappa_0>\left(\frac{4-\delta}{3}\right)\sigma_0+\frac{4\delta-7}{3}+\frac{\alpha(3-\delta)}{3}.
\end{align*}
 Then there exists $\eps_0\equiv\eps_0(\sigma_0, \kappa_0)$, such that for every $\eps\in(0, \eps_0)$
 \begin{align*}
 \revd{J_{\eps}(\omega)=J \quad \forall \omega\in \Omega_2.}
\end{align*}  
 
\revd{Moreover, $\lim\limits_{\eps\rightarrow 0}\mathbb{P}[\Omega_2]=1$ if 
\begin{align*}
\gamma>\max\left\{\frac{\kappa_0}{2}+d\eta, \kappa_0+d\eta-\frac{\sigma_0+1}{3},\gamma>\frac{8\delta-14+2\alpha(3-\delta)}{3(\delta-1)}+d\eta-\frac{1}{3}\right\},\; \tau^2\leq \eps^{4+\kappa_0+\beta},
\end{align*}
}
where $\beta>0$ may be arbitrarily small and $\eta$ is as in \assref{assumption2}.
\end{lemma}

\begin{proof}
$\mathbf{1}.$
We proceed by contradiction. We assume that $J_{\eps}<J$ on $\Omega_2$ and show that
\begin{align*}
\frac{\tau}{\eps}\sum_{i=1}^{J_{\eps}}\Vert Z^{i}\Vert^3_{\mathbb{L}^3}\leq \eps^{\sigma_0}\quad \text{on}\quad \Omega_2,
\end{align*}
which contradicts  the definition of $J_{\eps}$. 

Using  \cite[Lemma 4.5]{BM24} and \lemref{Omega2lemma} (i), it follows that, on $\Omega_2$ we have
\begin{align*}
\frac{\tau}{\eps}\sum_{i=1}^{J_{\eps}}\Vert Z^i\Vert^3_{\mathbb{L}^3}&\leq C\eps^{\sigma_0+\alpha-\kappa_0-1}\tau\sum_{i=1}^{J_{\eps}}\Vert Z^i\Vert^4_{\mathbb{L}^4}\nonumber\\
&\qquad+C\eps^{(\kappa_0+1-\sigma_0-\alpha)(3-\delta)}\tau\sum_{i=1}^{J_{\eps}}\Vert\Delta^{-1/2} Z^i\Vert^{\frac{4-\delta}{2}}\Vert \nabla Z^i\Vert^{\frac{3\delta-4}{2}}\nonumber\\
&\leq C\eps^{\sigma_0+\alpha}+C\eps^{(\kappa_0+1-\sigma_0-\alpha)(3-\delta)}\max_{1\leq i\leq J_{\eps}}\Vert \Delta^{-1/2}Z^i\Vert^{\frac{4-\delta}{2}}\tau\sum_{i=1}^{J_{\eps}}\Vert \nabla Z^i\Vert^{\frac{3\delta-4}{2}}\nonumber\\
&\leq C\eps^{\sigma_0+\alpha}+C\eps^{(\kappa_0+1-\sigma_0-\alpha)(3-\delta)+\left(\frac{4-\delta}{4}\right)\kappa_0}\tau\sum_{i=1}^{J_{\eps}}(\Vert \nabla Z^i\Vert^2)^{\frac{3\delta-4}{4}}.
\end{align*}
Since for $2<\delta<\frac{8}{3}$ we have $\frac{4}{3\delta-4}>1$ and $\frac{4}{8-3\delta}>1$, using H\"{o}lder's inequality  with exponents $\frac{4}{8-3\delta}>1$ and $\frac{4}{8-3\delta} $; and \lemref{Omega2lemma} (i) leads to
\begin{align*}
\frac{\tau}{\eps}\sum_{i=1}^{J_{\eps}}\Vert Z^i\Vert^3_{\mathbb{L}^3}&\leq C\eps^{\sigma_0+\alpha}+C\eps^{(\kappa_0+1-\sigma_0-\alpha)(3-\delta)+\left(\frac{4-\delta}{4}\right)\kappa_0}\tau\left(\sum_{i=1}^{J_{\eps}}\Vert\nabla Z^i\Vert^2\right)^{\frac{3\delta-4}{4}}\left(\sum_{i=1}^{J_{\eps}}1^{\frac{4}{8-3\delta}}\right)^{\frac{8-3\delta}{4}}\nonumber\\
&\leq C\eps^{\sigma_0+\alpha}+C\eps^{(\kappa_0+1-\sigma_0-\alpha)(3-\delta)+\left(\frac{4-\delta}{4}\right)\kappa_0}(\eps^{-4})^{\frac{3\delta-4}{4}}\eps^{\left(\frac{3\delta-4}{4}\right)\kappa_0}\nonumber\\
&\leq C\eps^{\sigma_0+\alpha}+C\eps^{3\kappa_0-\sigma_0(3-\delta)+7-4\delta-\alpha(3-\delta)}.
\end{align*}
The right hand side of the above inequality is bounded above by $\eps^{\sigma_0}$ for $\eps$ small enough if $3\kappa_0-\sigma_0(3-\delta)+7-4\delta-\alpha(3-\delta)> \sigma_0$, i.e., if $\kappa_0>\left(\frac{4-\delta}{3}\right)\sigma_0+\frac{4\delta-7}{3}+\frac{\alpha(3-\delta)}{3}$. This proves the first statement of the lemma. 

$\mathbf{2}.$
We now prove the second statement. 
 Let us recall that from \lemref{Omega2lemma} (ii) we have
\begin{align*}
\mathbb{P}[\Omega_2]\geq 1-C\eps^{-\kappa_0}\max\left(\eps^{\sigma_0}, \eps^{2\gamma-2d\eta},\eps^{\gamma-d\eta+\frac{\sigma_0+1}{3}}, \frac{\tau^2}{\eps^{4}}\right).
\end{align*}
Hence, to ensure $\lim\limits_{\eps\rightarrow 0}\mathbb{P}[\Omega_2]=1$ we require $\sigma_0>\kappa_0$, $2\gamma-2d\eta-\kappa_0>0$, $\gamma-d\eta+\frac{\sigma_0+1}{3}-\kappa_0>0$ and $\tau^2\leq \eps^{4+\kappa_0+\beta}$ for an arbitrarily small $\beta$. Taking in account the requirement in Step $\mathbf{1}$ about $\kappa_0$, to get $\sigma_0>\kappa_0$ it is enough to require $\sigma_0>\frac{4\delta-7}{\delta-1}+\frac{\alpha(3-\delta)}{\delta-1}$. To get $2\gamma-2d\eta-\kappa_0>0$ and $\gamma-d\eta+\frac{\sigma_0+1}{3}-\kappa_0>0$, it is enough to require $\gamma>\max\left\{\frac{\kappa_0}{2}+d\eta, \kappa_0+d\eta-\frac{\sigma_0+1}{3}\right\}$.  \revd{In addition, by $\mathbf{1}.$, $\kappa_0>\left(\frac{4-\delta}{3}\right)\sigma_0+\frac{4\delta-7}{3}+\frac{\alpha(3-\delta)}{3}$, $\sigma_0>\frac{4\delta-7}{\delta-1}+\frac{\alpha(3-\delta)}{\delta-1}$, which along with $\gamma-d\eta+\frac{\sigma_0+1}{3}-\kappa_0>0$, $\sigma_0>\kappa_0$ implies $\gamma>\frac{8\delta-14+2\alpha(3-\delta)}{3(\delta-1)}+d\eta-\frac{1}{3}$.}
\end{proof}

\revd{
We collect the requirements on  parameters useful to derive an  estimate for $Z^j=X^j-X^j_{\mathrm{CH}}$ in the assumption below.
\begin{Assumption}
\label{assumption2a}
Let $u^{\eps}_0\in\mathbb{H}^3$, $\mathcal{E}(u^{\eps}_0)<C$.    Assume that for fixed $0 < \alpha < 7$, $2<  \delta \leq \frac{8}{3}$ the parameters $(\sigma_0, \kappa_0, \gamma)$ satisfy
\begin{align*}
\sigma_0&>\frac{4\delta-7}{\delta-1}+\frac{\alpha(3-\delta)}{\delta-1},\quad \sigma_0>\kappa_0>\left(\frac{4-\delta}{3}\right)\sigma_0+\frac{4\delta-7}{3}+\frac{\alpha(3-\delta)}{3}, \\
\gamma&>\max\left\{\frac{\kappa_0}{2}+d\eta, \kappa_0+d\eta-\frac{\sigma_0+1}{3},\frac{8\delta-14+2\alpha(3-\delta)}{3(\delta-1)}+d\eta-\frac{1}{3}, \frac{5}{2}\right\}.
\end{align*}
For sufficiently small $\eps_0\equiv (\sigma_0,\kappa_0)>0$ and $\mathfrak{l}_{\mathrm{CH}}\geq 3$ from \lemref{LemmaLubo19}, and arbitrary $0<\beta<\frac{1}{2}$, the time-step  $\tau$ and  the mesh-size  $h$ of the approximation of the noise satisfy
\begin{align*}
\tau\leq C\min\left\{\eps^{\mathfrak{l}_{\mathrm{CH}}}, \eps^{2+\frac{\kappa_0}{2}+\beta}\right\},\quad h=\eps^{\eta}\quad \forall \eps\in(0, \eps_0),
\end{align*}
where $\eta$ is such that
\begin{align*}
0<\eta<\min\left\{\frac{2\gamma-3}{2+3d}, \frac{2\gamma-6}{3d}\right\}.
\end{align*}
\end{Assumption}
}

\begin{lemma}
\label{EstiZ}
Let \assref{assumption2a}  be fulfilled. Then 
\begin{align*}
&\mathbb{E}\left[\max_{1\leq j\leq J}\Vert Z^j\Vert^2_{\mathbb{H}^{-1}}+\eps^4\tau\sum_{j=1}^J\Vert\nabla Z^j\Vert^2+\frac{\tau}{\eps}\sum_{j=1}^J\Vert Z^j\Vert^4_{\mathbb{L}^4}\right]\nonumber\\
&\quad\leq \left(\frac{C}{\eps^{\kappa_0}}\max\left\{\eps^{\sigma_0}, \eps^{2\gamma-2d\eta},\eps^{\gamma-d\eta+\frac{\sigma_0+1}{3}}, \frac{\tau^2}{\eps^{4}}\right\}\right)^{\frac{1}{2}}.
\end{align*}
\end{lemma}

\begin{proof} 
First of all, note that $\mathbb{E}[\mathcal{A}_J]=\mathbb{E}[\displaystyle{1\!\!1_{\Omega_2}}\mathcal{A}_J]+\mathbb{E}[\displaystyle{1\!\!1_{\Omega_2^c}}\mathcal{A}_J]$.  Since from \lemref{Mepsilonlemma} $J_{\eps}=J$ on $\Omega_2$, it follows from \eqref{EstiAM1} that
\begin{align*}
\mathbb{E}[\displaystyle{1\!\!1_{\Omega_2}}\mathcal{A}_J]=\mathbb{E}[\displaystyle{1\!\!1_{\Omega_2}}\mathcal{A}_{J_{\eps}}]=\mathbb{E}[\mathcal{A}_{J_{\eps}}]\leq C\max\left(\eps^{\sigma_0}, \eps^{2\gamma-2d\eta},\eps^{\gamma-d\eta+\frac{\sigma_0+1}{3}}, \frac{\tau^2}{\eps^{4}}\right).
\end{align*}
To bound $\mathbb{E}[\displaystyle{1\!\!1_{\Omega_2^c}}\mathcal{A}_J]$, we use the embeddings $\mathbb{H}^1\hookrightarrow\mathbb{L}^4\hookrightarrow \mathbb{H}^{-1}$, Poincar\'{e}'s inequality, which together with the higher moment estimate, namely \lemref{MomentLemma} implies
\begin{align}
\label{Yan5}
\mathbb{E}[\mathcal{A}_J^2]\leq C\mathbb{E}[\mathcal{E}(X^J)^2]\leq C\left(\mathcal{E}(u^{\eps}_0)^2+1\right).
\end{align}
Next, note that from \lemref{Omega2lemma} (ii) we have
\begin{align}
\label{Yan6}
\mathbb{P}[\Omega_2^c]=1-\mathbb{P}[\Omega_2]\leq \frac{C}{\eps^{\kappa_0}}\max\left(\eps^{\sigma_0}, \eps^{2\gamma-2d\eta},\eps^{\gamma-d\eta+\frac{\sigma_0+1}{3}}, \frac{\tau^2}{\eps^{4}}\right).
\end{align}
Finally using Cauchy-Schwarz's inequality, \eqref{Yan5} and \eqref{Yan6} yields
\begin{align*}
\mathbb{E}[\displaystyle{1\!\!1_{\Omega_2^c}}\mathcal{A}_J]\leq \left(\mathbb{P}[\Omega_2^c]\right)^{\frac{1}{2}}\left(\mathbb{E}[\mathcal{A}_J^2]\right)^{\frac{1}{2}}\leq \left(\frac{C}{\eps^{\kappa_0}}\max\left(\eps^{\sigma_0}, \eps^{2\gamma-2d\eta},\eps^{\gamma-d\eta+\frac{\sigma_0+1}{3}}, \frac{\tau^2}{\eps^{4}}\right)\right)^{\frac{1}{2}}.
\end{align*}
This completes the proof of the lemma. 
\end{proof}

The next theorem provides an error estimate for the numerical approximation \eqref{scheme1b} and is the main result of this section.
\revd{
We collect the conditions on  parameters required for an estimate of $u(t_j)-X^j$ in the assumption below. These conditions also include Assumptions \ref{assumption1} and \ref{assumption2}.
\begin{Assumption}
\label{assumption2b}
Let the assumptions of \lemref{LemmaLubo19} hold and in addition let 
 $\mathcal{E}(u^{\eps}_0)<C$.     Let $\delta_0>0$ and $\eta_0>0$ from \eqref{HigherSet}. Assume that for fixed $0 < \alpha < 7$, $2<  \delta \leq \frac{8}{3}$ the parameters $(\sigma_0, \kappa_0, \gamma)$ satisfy
\begin{align*}
\sigma_0&>\max\left\{\frac{4\delta-7}{\delta-1}+\frac{\alpha(3-\delta)}{\delta-1}, \frac{(7-\alpha)\delta+6\alpha-8}{\delta-2}\right\},\nonumber\\ \sigma_0&>\kappa_0>\max\left\{\left(\frac{4-\delta}{3}\right)\sigma_0+\frac{4\delta-7}{3}+\frac{\alpha(3-\delta)}{3}, \frac{3}{4}\sigma_0+\frac{1}{4}+2\delta_0+2\eta_0\right\},\\
\gamma&>\max\left\{\frac{\kappa_0}{2}+d\eta, \kappa_0+d\eta-\frac{\sigma_0+1}{3},\frac{(14-2\alpha)\delta+12\alpha-16}{3(\delta-2)}+d\eta-\frac{1}{3}, \frac{5}{2}\right\}.
\end{align*}
For sufficiently small $\eps_0\equiv (\sigma_0,\kappa_0)>0$ and $\mathfrak{l}_{\mathrm{CH}}\geq 3$ from \lemref{LemmaLubo19}, and arbitrary $0<\beta<\frac{1}{2}$, the time-step  $\tau$ and  the mesh-size  $h$ in the approximation of the noise (cf.  \eqref{Noiseapprox1}) respectively satisfy
\begin{align*}
\tau\leq C\min\left\{\eps^{\mathfrak{l}_{\mathrm{CH}}}, \eps^{2+\frac{\kappa_0}{2}+\beta}\right\},\quad h=\eps^{\eta}\quad \forall \eps\in(0, \eps_0),
\end{align*}
where $\eta$ is such that
\begin{align*}
0<\eta<\min\left\{\frac{2\gamma-3}{2+3d}, \frac{2\gamma-6}{3d}\right\}.
\end{align*}
\end{Assumption}
}

\begin{theorem}
\label{mainresult1}
 Let \assref{assumption2b} be fulfilled.  Let $X^j$ be the numerical approximation \eqref{scheme1b} and  $u$ the variational solution to \eqref{model1}. Then  for all $0<\beta<\frac{1}{2}$ the following holds 
\begin{align*}
& \mathbb{E}\left[\max_{1\leq j\leq J}\Vert u(t_j)-X^j\Vert^2_{\mathbb{H}^{-1}}\right]
\\
& \qquad\leq C\max\left\{\left(\frac{1}{\eps^{\kappa_0}}\max\left\{\eps^{\sigma_0}, \eps^{2\gamma-2d\eta},\eps^{\gamma-d\eta+\frac{\sigma_0+1}{3}}, \frac{\tau^2}{\eps^{4}}\right\}\right)^{\frac{1}{2}}, \eps^{\frac{2}{3}\sigma_0}, \frac{\tau^{2-2\beta}}{\eps^{\mathfrak{m}_{\mathrm{CH}}}}, \eps^{4\gamma-4\eta-2}\right\},
\end{align*}
where
the constant $C>0$ is independent of $\tau$, $h$ and $\eps$.
\end{theorem}

\begin{proof} 
We split the error  as follows
\begin{align*}
u(t_j)-X^j= u(t_j)-u_{\mathrm{CH}}(t_j)+u_{\mathrm{CH}}(t_j)-X^j_{\mathrm{CH}}+ X^j_{\mathrm{CH}}-X^j. 
\end{align*}
From \lemref{LemmaLubo19} (iv), we have
\begin{align}
\label{First}
\max_{1\leq j\leq J}\Vert u_{\mathrm{CH}}(t_j)-X^j_{\mathrm{CH}}\Vert^2_{\mathbb{H}^{-1}}\leq \frac{C\tau^{2-2\beta}}{\eps^{\mathfrak{m}_{\mathrm{CH}}+1}}.
\end{align}
By \lemref{EstiZ} we  estimate  $\mathbb{E}\left[\max_{1\leq j\leq J}\Vert X^j-X^j_{\mathrm{CH}}\Vert^2_{\mathbb{H}^{-1}}\right]$ and by
 \lemref{analyticCahnHilliard} we estimate $\mathbb{E}\left[\max_{1\leq j\leq J}\Vert u(t_j)-u_{\mathrm{CH}}(t_j)\Vert^2_{\mathbb{H}^{-1}}\right]$.
\end{proof}

\section{Sharp-interface limit}
\label{SharpLimit}
In this section we show uniform convergence of the numerical approximation \eqref{scheme1b} to its sharp-interface limit
which is the (deterministc) Hele-Shaw/Mullins-Sekerka problem.
The Hele-Shaw problem is defined as follows: Find $v_{\mathrm{MS}}: [0, T]\times\mathcal{D}\rightarrow\mathbb{R}$ and the interface $\{\Gamma_t^{\mathrm{MS}};\; 0\leq t\leq T\}$ such that for all $0<t\leq T$
\begin{align}
\label{HeleShaw1}
\begin{array}{rclllll}
&-\Delta v_{\mathrm{MS}} & = & 0&&\text{in}\quad \mathcal{D}\setminus\Gamma_t^{\mathrm{MS}},\\
&-2\mathcal{V} & = &\left[\partial_{ \normal_{\Gamma}} v_{\mathrm{MS}}\right]_{\Gamma_t^{\mathrm{MS}}}&& \text{on}\quad \Gamma_t^{\mathrm{MS}},\\
&v_{\mathrm{MS}} & = &\alpha \kappa &&\text{on}\quad \Gamma_t^{\mathrm{MS}},\\
&\partial_{\normal} v_{\mathrm{MS}}& = &0&&\text{on}\quad \partial\mathcal{D},\\
&\Gamma^{\mathrm{MS}}_0 & = & \Gamma_{00},
\end{array}
\end{align}
where  $\kappa$ is the mean curvature of the evolving interface $\Gamma_t^{\mathrm{MS}}$, and $\mathcal{V}$ is the velocity in the direction of its normal $\normal_{\Gamma}$, as well as $\left[\partial_{\normal_{\Gamma}} v_{\mathrm{MS}}\right]_{\Gamma_t^{\mathrm{MS}}}(z)=\left(\partial_{\normal} v^{+}_{\mathrm{MS}}-\partial_{\normal} v^{-}_{\mathrm{MS}}\right)(z)$ for all $z\in \Gamma_t^{\mathrm{MS}}$,  where $v^{+}_{\mathrm{MS}}$ and $v^{-}_{\mathrm{MS}}$ are  the restriction of $v_{\mathrm{MS}}$ on $\mathcal{D}_t^{\pm}$ (the exterior/interior of $\Gamma_t^{\mathrm{MS}}$ in $\mathcal{D}$). The constant $\alpha$ in \eqref{HeleShaw1} is chosen as $\alpha=\frac{1}{2}c_F$, where $c_F=\int_{-1}^1\sqrt{2F(s)}ds=\frac{1}{3}2^{\frac{3}{2}}$, and $F$ is the double-well potential.

To overcome the difficulties caused by the low regularity of the considered noise
we write $X^j = \widetilde{X}^j + \widehat{X}^j$, $j=1,\cdots, J$,
where $\widetilde{X}^j$ is the solution of
\begin{align}
\label{linearscheme}
\begin{array}{rclllll}
(\widetilde{X}^j-\widetilde{X}^{j-1}, \varphi)+\tau(\nabla \widetilde{w}^j, \nabla \varphi)& = &\eps^{\gamma}(\Delta_j\overline{W}, \varphi)&& \forall\varphi\in \mathbb{H}^1,\\
(\widetilde{w}^j, \psi) &= &\eps(\nabla\widetilde{X}^j, \nabla\psi)&& \forall\psi\in\mathbb{H}^1,\\
\widetilde{X}^0 & = &0,
\end{array}
\end{align}  
with $\partial_{\normal}\widetilde{X}^j=\partial_{\normal}\widetilde{w}^j=0$ on $\partial\mathcal{D}$,

and $\widehat{X}^j$ satisfies
\begin{align}
\label{randompdescheme}
\begin{array}{rlllll}
(\widehat{X}^j-\widehat{X}^{j-1}, \varphi)+\tau(\nabla \widehat{w}^j, \nabla\varphi)& = & 0&& \forall\varphi\in\mathbb{H}^1,\\
\eps(\nabla \widehat{X}^j, \nabla\psi)+\displaystyle\frac{1}{\eps}(f(\revl{\widehat{X}^j+\widetilde{X}^j}), \psi) & = &(\widehat{w}^j, \psi)&& \forall\psi\in\mathbb{H}^1,\\
\widehat{X}^0& = & u^{\eps}_0,
\end{array}
\end{align} 
with $\partial_{\normal}\widehat{X}^j=\partial_{\normal}\widehat{w}^j=0$ on $\partial\mathcal{D}$.

We note that in the subsequent derivation of the stronger stability estimates for the solutions of \eqref{linearscheme} and \eqref{randompdescheme} we implicitly assume that their respective (analytically) strong formulation
are well-defined.
The existence of the corresponding strong formulations can be justified rigorously by the regularity of the Neumann Laplace operator, cf. \cite[Section 5]{Banas19}.
The unique solvability and measurability of \eqref{linearscheme}, \eqref{randompdescheme} (which will be shown below) ensures that the approximation $X^j = \widetilde{X}^j + \widehat{X}^j$ satisfies the original numerical scheme \eqref{scheme1b}.

 \revl{
 To study the sharp-interface limit of the numerical solution $\{X^j\}_{j=0}^J$ in \eqref{scheme1b} we rewrite$ X^j\pm 1 = (X^j-X^j_{\mathrm{CH}})+(X^j_{\mathrm{CH}}\pm 1)$ and denote $ Z^j:=X^j-X^j_{\mathrm{CH}}$.
Thanks to the well-known result on the sharp-interface limit of the numerical solution $X^j_{\mathrm{CH}}$ of the deterministic Cahn-Hilliard equation (cf. \cite[Theorem 4.2]{fp04}),
it suffices to show that $\lim_{\eps\rightarrow 0} \|Z^j\|_{\mathbb{L}^\infty} = 0$ on a subset of high probability.
Owing to the low regularity of the (discrete) noise it is not possible to show an estimate for $Z^j$ directly.
Instead we rewrite $Z^j$ as
 \begin{align}
 \label{ErrorZ}
 Z^j=X^j-X^j_{\mathrm{CH}}=(X^j-X^j_{\mathrm{CH}}-\widetilde{X}^j)+\widetilde{X}^j,
\end{align}
and consider the translated difference
 \begin{align}
 \label{translateddifference}
 \widehat{Z}^j:=Z^j-\widetilde{X}^j= X^j-X^j_{\mathrm{CH}}-\widetilde{X}^j=\widehat{X}^j-X^j_{\mathrm{CH}}.
 \end{align}
 The  translated difference $\widehat{Z}^j$ is the discrete counterpart of the continuous one used in \cite{BYZ22,BM24} when dealing with space-time white noise.

Hence, we proceed as follows.
 \begin{enumerate}
 \item[(a)] We provide higher regularity estimates of the discrete stochastic convolution $\widetilde{X}^j$ (cf. \lemref{regularitynoise} in \secref{AprioriEstimates}) and prove that the $\mathbb{L}^{\infty}$-norm of $\widetilde{X}^j$ vanishes
for $\varepsilon\rightarrow 0$, see \eqref{LinfinityXtilde}. In addition, we show a $\tau$-independent $\mathbb{L}^\infty$ bound for the solution $X^j$ on a subset of high probability (cf. \lemref{Linfinitya}).
 \item[(b)] We provide  $\mathbb{L}^{\infty}$-estimate of the translated difference $\widehat{Z}^j$, see \lemref{maintheorem2a}.
 \item[(c)] We use triangle inequality, the $\mathbb{L}^{\infty}$-estimate of $\widehat{Z}^j$ (cf. \lemref{maintheorem2a}) and $\widetilde{X}^j$ (cf. \eqref{LinfinityXtilde}) to prove that on a subset of high probability,
  $\Vert Z^j\Vert_{\mathbb{L}^{\infty}}\rightarrow 0$
  for $\varepsilon\rightarrow 0$ (for suitable scaling of the considered parameters), see \thmref{LinfinityZ}. Finally, \thmref{LinfinityZ} is  used to conclude the sharp-interface limit of the numerical solution $X^j$ in \thmref{maintheorem2}.
\end{enumerate}   
}

\subsection{Well-posedness of the numerical schemes \eqref{randompdescheme} and \eqref{linearscheme}}
\label{WellposedSection}
In this section we show that there exist unique, $\mathcal{F}_{t_j}$-measurable solutions $\widehat{X}^j$ to \eqref{randompdescheme} and $\widetilde{X}^j$ to \eqref{linearscheme} for $j=1,\cdots,J$. Let us start with the solvability of \eqref{randompdescheme}.
We assume that for all $j=1,\cdots,J$, $\widetilde{X}^j\in L^2(\Omega, \mathbb{H}^1)$ and  is an $\mathcal{F}_{t_{j}}$-measurable random variable. We also assume that $\widehat{X}^{j-1}\in L^2(\Omega; \mathbb{H}^1)$ is an $\mathcal{F}_{t_{j-1}}$-measurable random variable.

Taking $\varphi=(-\Delta)^{-1}\psi$ in the first equation of \eqref{randompdescheme} yields
\begin{align*}
(\widehat{X}^j-\widehat{X}^{j-1}, \psi)_{-1}+\tau(\widehat{w}^j, \psi)=0. 
\end{align*}
Noting the second equation of \eqref{randompdescheme} we obtain that
\begin{align*}
(\widehat{X}^j-\widehat{X}^{j-1}, \psi)_{-1}+\eps\tau(\nabla \widehat{X}^j, \nabla\psi)+\frac{\tau}{\eps}\left(f(\widehat{X}^j+\widetilde{X}^j), \psi\right)=0\quad \psi\in \mathbb{H}^1,.
\end{align*}
Using that $f(u)=u^3-u$ we deduce
\begin{align*}
(\widehat{X}^j, \psi)_{-1}+\eps\tau(\nabla \widehat{X}^j, \nabla\psi)+\frac{\tau}{\eps}\left((\widehat{X}^j+\widetilde{X}^j)^3, \psi\right)-\frac{\tau}{\eps}\left((\widehat{X}^{j}+\widetilde{X}^j), \psi\right)-(\widehat{X}^{j-1}, \psi)_{-1}=0
\end{align*}
for all $\psi\in \mathbb{H}^1$. 
This motivates the introduction of the following functional
\begin{align}
\label{Convexfunction1}
G(v):=\frac{1}{2}\Vert v-\widehat{X}^{j-1}\Vert^2_{-1}+\frac{\tau}{4\eps}\Vert v+\widetilde{X}^j\Vert^4_{\mathbb{L}^4}+\frac{\eps\tau}{2}\Vert \nabla v\Vert^2-\frac{\tau}{2\eps}\Vert v+\widetilde{X}^j\Vert^2\quad v\in\mathbb{H}^1.
\end{align}
\begin{lemma}
\label{convexlemma}
For $\tau\leq \frac{1}{2}\eps^3$ the mapping $G:\mathbb{H}^1\rightarrow \mathbb{R}$ is coercive and strictly convex.
\end{lemma}
\begin{proof}
The first variation of the first term in $G$ is:
\begin{align*}
\frac{d}{ds}\left.\left[\frac{1}{2}\Vert v+s\psi-\widehat{X}^{j-1}\Vert^2_{-1}\right]\right\vert_{s=0}=\left.\left[(v+s\psi-\widehat{X}^{j-1}, \psi)_{-1}\right]\right\vert_{s=0}=(v-\widehat{X}^{j-1}, \psi)_{-1}. 
\end{align*}
The second variation of the first term in $G$ is:
\begin{align*}
\frac{d^2}{ds^2}\left.\left[\frac{1}{2}\Vert v+s\psi-\widehat{X}^{j-1}\Vert^2_{-1}\right]\right\vert_{s=0}=(\psi, \psi)_{-1}>0\quad \psi\neq 0. 
\end{align*}
Analogously, we compute the variations of remaining terms in $G$ and get
\begin{align*}
\frac{d}{ds}G(v+s\psi)\vert_{s=0}=(v-\widehat{X}^{j-1}, \psi)_{-1}+\frac{\tau}{\eps}((v+\widetilde{X}^j)^3, \psi)+\revl{\eps\tau (\nabla v,\nabla\psi)}-\frac{\tau}{\eps}(v+\widetilde{X}^j, \psi).
\end{align*}
The second variation of $G$ is easily computed and one obtains
\begin{align*}
\frac{d^2}{ds^2}G(v+s\psi)\vert_{s=0}=(\psi, \psi)_{-1}+\frac{3\tau}{\eps}\left((v+\widetilde{X}^j)^2, \psi^2\right)+\eps\tau\left(\nabla\psi, \nabla\psi\right)-\frac{\tau}{\eps}(\psi, \psi). 
\end{align*}
Using the interpolation inequality $\Vert.\Vert^2\leq\Vert.\Vert_{\mathbb{H}^{-1}}\Vert\nabla .\Vert$ and Young's inequality  yields
\begin{align}
\label{inter1}
\frac{\tau}{\eps}\Vert \psi\Vert^2\leq \frac{\tau}{\eps}\Vert \psi\Vert_{-1}\Vert \nabla\psi\Vert\leq \frac{1}{2}\Vert \psi\Vert^2_{-1}+\frac{\tau^2}{2\eps^2}\Vert\nabla\psi\Vert^2. 
\end{align}
Therefore, it holds that
\begin{align*}
\frac{d^2}{ds^2}G(v+s\psi)\vert_{s=0}\geq\frac{1}{2}\Vert\psi\Vert^2_{-1}+\tau\left(\eps-\frac{\tau}{2\eps^2}\right)\Vert\nabla\psi\Vert^2.
\end{align*}
Hence, for $\tau\leq \frac{1}{2}\eps^3$ the mapping $G$ is strictly convex.

Next, using triangle and Young's inequalities and \eqref{inter1} (with $\psi=v-\widehat{X}^{j-1}$) yields
\begin{align*}
\frac{\tau}{2\eps}\Vert v+\widetilde{X}^j\Vert^2&\leq \frac{\tau}{\eps}\Vert v-\widehat{X}^{j-1}\Vert^2+\frac{\tau}{\eps}\Vert \widehat{X}^{j-1}+\widetilde{X}^j\Vert^2\nonumber\\
&\leq \frac{1}{2}\Vert v-\widehat{X}^{j-1}\Vert^2_{-1}+\frac{\tau^2}{2\eps^2}\Vert \nabla[v-\widehat{X}^{j-1}]\Vert^2+\frac{\tau}{\eps}\Vert \widehat{X}^{j-1}+\widetilde{X}^j\Vert^2\nonumber\\
&\leq \frac{1}{2}\Vert v-\widehat{X}^{j-1}\Vert^2_{-1}+\frac{\tau^2}{\eps^2}\Vert \nabla v\Vert^2+\frac{\tau^2}{\eps^2}\Vert \nabla\widehat{X}^{j-1}\Vert^2+\frac{\tau}{\eps}\Vert \widehat{X}^{j-1}+\widetilde{X}^j\Vert^2. 
\end{align*}
From the above we deduce by Poincar\'{e}'s inequality that
\begin{align*}
G(v)\geq \tau\left(\frac{\eps}{2}-\frac{\tau}{\eps^2}\right)\Vert \nabla v\Vert^2-\frac{\tau^2}{\eps^2}\Vert\revd{\nabla\widehat{X}^{j-1}}\Vert^2-\frac{\tau}{\eps}\Vert \widehat{X}^{j-1}+\widetilde{X}^j\Vert^2\geq C_1(\eps)\Vert v\Vert^2_{\mathbb{H}^1}-C_2(\eps),
\end{align*}
for all $v\in\mathbb{H}^1$,
where $0<C_1(\eps)<\infty$ (since  $\tau\leq \frac{\eps^3}{2}$) is a constant which does not dependent on $v$ and $C_2(\eps)=\frac{\tau^2}{\eps^2}\Vert\revl{\nabla\widehat{X}^{j-1}}\Vert^2+\frac{\tau}{\eps}\Vert \widehat{X}^{j-1}+\widetilde{X}^j\Vert^2<\infty$. Therefore $G$ is coercive.
\end{proof}

To show the $\mathcal{F}_{t_j}$-measurability, we make use of the following lemma, which is a straightforward generalization of \cite[Lemma 3.2]{EmmrichSiska2} (or \cite[Lemma 3.8]{Gyongy})
to the infinite dimensional case.

\begin{lemma}
\label{measurabilitylemma}
Let $(S, \Sigma)$ be a measurable space and $V$ a Banach space. Let $\mathbf{f}: S\times V\longrightarrow V$   be a function that is $\Sigma$-measurable in its first argument for every fixed $v\in V$, that is continuous in its second argument for every fixed $s\in S$ and in addition such that for every $s\in S$ the equation $\mathbf{f}(s,v)=0$ has a unique solution $v=g(s)$. Then $g: S\longrightarrow V$ is $\Sigma$-measurable. 
\end{lemma}

\begin{lemma}
\label{existencescheme}
Let $\tau\leq \frac{1}{2}\eps^3$ and $\widehat{X}^0, \widetilde{X}^j\in L^2(\Omega, \mathbb{H}^1)$. Then there exists a unique $\mathcal{F}_{t_j}$-measurable solution
$(\widehat{X}^j, \widehat{w}^j) \in L^2(\Omega, \mathbb{H}^1)\times L^2(\Omega, \mathbb{H}^1)$ of \eqref{randompdescheme} for $j=1, \cdots, J$.
\end{lemma}
\begin{proof}
We proceed by induction and assume that given $\widehat{X}^0=u^{\epsilon}_0\in L^2(\Omega, \mathbb{H}^1)$ there exist unique $\mathcal{F}_{t_k}$-measurable solutions $\widehat{X}^k$, $\widehat{w}^k$ for all $k=1, \dots, j-1$.
Since $G$ is coercive and strictly convex (cf. \lemref{convexlemma}), by the standard theory of convex optimization \cite[Chapter 7]{ciarlet_1989},
$G$ has a unique (bounded) minimizer $\widehat{X}^j \equiv \widehat{X}^j(\omega)$ in $\mathbb{H}^1$. Moreover, from  \cite[Theorem 7.4-4]{ciarlet_1989}, $\widehat{X}^j$ is the unique minimizer of $G$ if and only if it satisfies $\mathbb{P}$-a.s. the Euler equation: $(\mathcal{A}(\widehat{X}^j), \psi)=0$ for all $\psi\in\mathbb{H}^1$, where
\begin{align*}
(\mathcal{A}(v), \psi)&:=\frac{d}{ds}G(v+s\psi)\vert_{s=0}\nonumber\\
&=(v-\widehat{X}^{j-1},\psi)_{-1}+\frac{\tau}{\eps}((v+\widetilde{X}^j)^3, \psi)+\eps\tau(\nabla v,\nabla\psi)-\frac{\tau}{\eps}(v+\widetilde{X}^j, \psi)\nonumber\\
&=(v-\widehat{X}^{j-1}, \psi)_{-1}+\frac{\tau}{\eps}(f(v+\widetilde{X}^j), \psi)+\eps\tau(\nabla v, \nabla\psi).
\end{align*}
Therefore $\widehat{X}^j$ is the unique solution to the variational problem: find $v\in \mathbb{H}^1$ such that 
\begin{align*}
(v-\widehat{X}^{j-1}, \psi)_{-1}+\frac{\tau}{\eps}(f(v+\widetilde{X}^j), \psi)+\eps\tau(\nabla v, \nabla\psi)=0\quad \forall \psi\in \mathbb{H}^1\quad \mathbb{P}\text{-a.s.}
\end{align*}
We consider the following variational problem: find $v\in\mathbb{H}^1$ such that
\begin{align}
\label{newscheme2}
(\nabla v,\nabla\varphi)=-\frac{1}{\tau}\left(\widehat{X}^j-\widehat{X}^{j-1}, \varphi\right)\quad 
\forall\varphi\in\mathbb{H}^1\quad \mathbb{P}\text{-a.s.}.
\end{align}
Note that by the Lax-Milgram theorem, the variational problem  \eqref{newscheme2} has a unique solution,
that is, there exists a unique process $\widehat{w}^j$ satisfying $\mathbb{P}$-a.s.
\begin{align}
\label{neweq1}
(\widehat{X}^j-\widehat{X}^{j-1}, \psi)=-\tau(\nabla\widehat{w}^j, \nabla\psi)\quad \psi\in\mathbb{H}^1.
\end{align}
Using the definition of the inner product $(., .)_{-1}$ and the identity \eqref{neweq1}, it holds $\mathbb{P}$-a.s.
\begin{align*}
(\widehat{X}^j-\widehat{X}^{j-1}, \psi)_{-1}&=(\widehat{X}^j-\widehat{X}^{j-1}, (-\Delta)^{-1}\psi)\nonumber\\
&=-\tau(\nabla\widehat{w}^j, \nabla(-\Delta)^{-1}\psi)=-\tau(\widehat{w}^j, \psi)\quad \forall \psi\in\mathbb{H}^1.
\end{align*}
Using the preceding identity it follows that the unique minimizer $\widehat{X}^j$ of the convex function $G$ in \eqref{Convexfunction1} is the unique process satisfying $\mathbb{P}$-a.s.
\begin{align*}
\eps(\nabla\widehat{X}^j, \nabla\psi)+\frac{1}{\eps}\left(f(\widehat{X}^j+\widetilde{X}^j), \psi\right)=(\widehat{w}^j, \psi)\quad \psi\in\mathbb{H}^1, 
\end{align*}
where $\widehat{w}^j$ is the unique stochastic process satisfying $\mathbb{P}$-a.s. 
\begin{align*}
(\widehat{X}^j-\widehat{X}^{j-1}, \psi)+\tau(\nabla\widehat{w}^j, \nabla\psi)=0\quad \psi\in\mathbb{H}^1.
\end{align*}
Hence \eqref{randompdescheme} has a unique solution $(\widehat{X}^j, \widehat{w}^j)$. 

Applying \lemref{measurabilitylemma} with  $(S, \Sigma)=(\Omega, \mathcal{F}_{t_j})$ and $\mathbf{f}: \Omega\times \mathbb{H}^1\longrightarrow\mathbb{H}^1$, given by
\begin{align*}
(\mathbf{f}({\omega}, u), \psi)=\frac{1}{\tau}(u-\widehat{X}^{j-1}(\omega), \psi)_{-1}+\eps(\nabla u, \nabla\psi)+\frac{1}{\eps}\left(f(u+\widetilde{X}^j(\omega)), \psi\right)\quad \forall\psi\in\mathbb{H}^1
\end{align*}
yields the $\mathcal{F}_{t_j}$-measurability of $\widehat{X}^j$. 
The $\mathcal{F}_{t_j}$-measurability of $\widehat{w}^j$ then follows directly from \eqref{neweq1}.
The proof of the fact that $\widehat{X}^j, \widehat{w}^j \in L^2(\Omega, \mathbb{H}^1)$ is analogous to the proof of \lemref{momenta}.
\end{proof}

\begin{remark}
The time step restriction $\tau\leq \frac{1}{2}\eps^3$ for the solvability of the numerical scheme \eqref{randompdescheme} in \thmref{existencescheme} is consistent with the condition for the solvability of the corresponding numerical scheme in the deterministic setting, see, e.g., \cite[Theorem 3.3]{FengLiYukunXing2016} or \cite[Theorem 3.3]{Aristotelous2013}.
\end{remark}

\revd{
\begin{lemma}
For $j=1,\cdots,J$, there exists a unique $\mathcal{F}_{t_j}$-measurable stochastic process $(\widetilde{X}^j, \widetilde{w}^j)$ satisfying $\mathbb{P}$-a.s. \eqref{linearscheme}. Moreover, $\widetilde{X}^j\in L^2(\Omega, \mathbb{H}^1)$, $j=1,\cdots,J$. 
\end{lemma}

\begin{proof}
The proof goes along the same lines as the proof of \thmref{existencescheme}, hence we only sketch it.
We proceed by induction and assume that there exist unique $\mathcal{F}_{t_k}$-measurable solutions $\widetilde{X}^k$, $\widetilde{w}^k$  for  $k=1, \dots, j-1$. We introduce the following functional
\begin{align}
\label{Convexfunction2}
G(v)=\frac{1}{2}\Vert v-\widetilde{X}^{j-1}\Vert^2_{-1}+\frac{\eps\tau}{2}\Vert \nabla v\Vert^2-\frac{\eps^{\gamma}}{2}\Vert \Delta_j\overline{W}+v\Vert^2_{-1}+\frac{\eps^{\gamma}}{2}\Vert v\Vert^2_{-1}\quad v\in\mathbb{H}^1.
\end{align}
We have 
\begin{align}
\label{Gradian1}
\frac{dG}{ds}(v+s\psi)=(v+s\psi-\widetilde{X}^{j-1}, \psi)_{-1}+\eps\tau\left(\nabla(v+s\psi), \nabla\psi\right)-\eps^{\gamma}(\Delta_j\overline{W}, \psi)_{-1}\quad \forall \psi\in\mathbb{H}^1.
\end{align}
The second variation of the functional $G$ is: 
\begin{align*}
\frac{d^2G}{ds^2}(v+s\psi)=\Vert\psi\Vert^2_{-1}+\epsilon\tau\Vert \nabla\psi\Vert^2>0.
\end{align*}
It follows therefore that $G$ is a strictly convex function. Using triangle and Young's inequalities, we have 
\begin{align*}
\eps^{\gamma}\Vert \Delta_j\overline{W}+v\Vert^2_{-1}&\leq 2\eps^{\gamma}\Vert\Delta_j\overline{W}\Vert^2_{-1}+2\eps^{\gamma}\Vert v\Vert^2_{-1}\nonumber\\
&\leq 2\eps^{\gamma}\Vert\Delta_j\overline{W}\Vert^2_{-1}+4\eps^{\gamma}\Vert \widetilde{X}^{j-1}\Vert^2_{-1}+4\eps^{\gamma}\Vert v-\widetilde{X}^{j-1}\Vert^2_{-1}.
\end{align*}
Using the preceding estimate, it follows that
\begin{align*}
G(v)\geq \left(\frac{1}{2}-2\eps^{\gamma}\right)\Vert v-\widetilde{X}^{j-1}\Vert^2_{-1}+\frac{\eps\tau}{2}\Vert\nabla v\Vert^2-\eps^{\gamma}\Vert\Delta_j\overline{W}\Vert^2_{-1}-2\eps^{\gamma}\Vert \widetilde{X}^{j-1}\Vert^2_{-1}+\frac{\eps^{\gamma}}{2}\Vert v\Vert^2_{-1}.
\end{align*}
Since $0<\eps<1$, choosing $\gamma$ large enough so that $\frac{1}{2}-2\eps^{\gamma}\geq 0$ and using Poincar\'{e}'s inequality, it follows that
\begin{align*}
G(v)&\geq \frac{\eps\tau}{2}\Vert\nabla v\Vert^2-\eps^{\gamma}\Vert\Delta_j\overline{W}\Vert^2_{-1}-2\eps^{\gamma}\Vert \widetilde{X}^{j-1}\Vert^2_{-1}\nonumber\\
&\geq C\Vert v\Vert^2_{\mathbb{H}^1}-\eps^{\gamma}\Vert\Delta_j\overline{W}\Vert^2_{-1}-2\eps^{\gamma}\Vert \widetilde{X}^{j-1}\Vert^2_{-1}.
\end{align*}
Therefore $G$ is coercive. 
By the standard theory of convex optimization (cf. \cite[Chapter 7]{ciarlet_1989}),
the functional $G$ in \eqref{Convexfunction2} has a unique (bounded) minimizer $\widetilde{X}^j \revl{\equiv \widetilde{X}^j(\omega)}$ in $\mathbb{H}^1$. Moreover, from  \cite[Theorem 7.4-4]{ciarlet_1989}, $\widetilde{X}^j$ is the unique minimizer of $G$ if and only if it satisfies $\mathbb{P}$-a.s. the Euler equation: $(\mathcal{A}(\widetilde{X}^j), \psi)=0$ for all $\psi\in\mathbb{H}^1$, where
\begin{align*}
(\mathcal{A}(v), \psi)&:=\frac{d}{ds}G(v+s\psi)\vert_{s=0}.
\end{align*}
Using \eqref{Gradian1}, it follows that the unique minimizer $\widetilde{X}^j$ of the functional in \eqref{Convexfunction2} is the unique stochastic process satisfying $\mathbb{P}$-a.s.
\begin{align}
\label{Newscheme3}
(\widetilde{X}^{j}-\widetilde{X}^{j-1}, \psi)_{-1}+\eps\tau(\nabla \widetilde{X}^{j}, \nabla\psi)-\eps^{\gamma}(\Delta_j\overline{W}, \psi)_{-1}=0\quad \forall \psi\in\mathbb{H}^1. 
\end{align}
Let us consider the following variational problem: find $v\in \mathbb{H}^1$, such that
\begin{align}
\label{newvarprob1} 
\tau(\nabla v, \nabla\psi)=\eps^{\gamma}(\Delta_j\overline{W}, \psi)-(\widetilde{X}^j-\widetilde{X}^{j-1}, \psi)\quad \forall \psi\in\mathbb{H}^1.
\end{align}
Using the Lax-Milgram theorem, it follows that \eqref{newvarprob1} has a unique solution, that is, there exists a unique stochastic process $\widetilde{w}^j$ satisfying $\mathbb{P}$-a.s.
\begin{align*}
(\widetilde{X}^j-\widetilde{X}^{j-1}, \psi)+\tau(\nabla \widetilde{w}^j, \nabla\psi)=\eps^{\gamma}(\Delta_j\overline{W}, \psi)\quad \forall \psi\in\mathbb{H}^1.
\end{align*}
Using the definition of the $(.,.)_{-1}$, it follows from the preceding identity that
\begin{align}
\label{Newscheme4}
(\widetilde{X}^j-\widetilde{X}^{j-1}, \psi)_{-1}+\tau( \widetilde{w}^j, \psi)=\eps^{\gamma}(\Delta_j\overline{W}, \psi)_{-1}\quad \forall \psi\in\mathbb{H}^1.
\end{align}
Combining \eqref{Newscheme3} and \eqref{Newscheme4}, it follows that $(\widetilde{X}^j, \widetilde{w}^j)$ satisfies
\begin{align*}
(\widetilde{w}^j, \psi)=\eps(\nabla\widetilde{X}^j, \nabla \psi)\quad \forall\psi\in\mathbb{H}^1. 
\end{align*}
Since the variational problem: find $v\in\mathbb{H}^1$ such that 
\begin{align}
\label{Newvariational1}
(v, \psi)=\eps(\nabla\widetilde{X}^j, \nabla \psi)\quad \forall\psi\in\mathbb{H}^1
\end{align}
 has a unique solution,  it follows that  $(\widetilde{X}^j, \widetilde{w}^j)$ is the unique solution of \eqref{linearscheme}. Applying \lemref{measurabilitylemma} with  $(S, \Sigma)=(\Omega, \mathcal{F}_{t_j})$ and $\mathbf{f}: \Omega\times \mathbb{H}^1\longrightarrow\mathbb{H}^1$, with
\begin{align*}
(\mathbf{f}(u), \psi)=(u-\widetilde{X}^{j-1}, \psi)_{-1}+\eps\tau(\nabla u, \nabla\psi)-\eps^{\gamma}(\Delta_j\overline{W}, \psi)_{-1}\quad \psi\in\mathbb{H}^1
\end{align*}
implies the $\mathcal{F}_{t_j}$-measurability of $\widetilde{X}^j$. The $\mathcal{F}_{t_j}$-measurability of $\widetilde{w}^j$ follows from the fact that $\widetilde{w}^j$ solves \eqref{Newvariational1}.
The proof of the fact that $\widetilde{X}^j\in L^2(\Omega, \mathbb{H}^1)$ is analogous to the proof of \lemref{momenta}.
\end{proof}
}

 
\subsection{$\mathbb{L}^\infty$-estimates for the solution of \eqref{linearscheme} and the solution of \eqref{scheme1b}}
\label{AprioriEstimates}
We start by deriving an alternative representation of $\widetilde{X}^j$ which is  more convenient for the subsequent analysis.
We consider a  discrete  process $\{\widetilde{Y}^j\}_{j=0}^J$ such that $\widetilde{Y}^0=0$ and $\{\widetilde{Y}^j\}_{j=1}^J$ satisfies
\begin{align}
\label{Discreteprocess1a}
\widetilde{Y}^j  =(\mathbf{I}+\eps\tau\Delta^2)^{-1}\widetilde{Y}^{j-1}+\eps^{\gamma}(\mathbf{I}+\eps\tau\Delta^2)^{-1}\Delta_j\overline{W}\, \text{ for } j=1,\cdots,J,
\end{align}
along with the boundary condition $\partial_\normal\widetilde{Y}^j = \partial_\normal\Delta \widetilde{Y}^{j}=0$ on $\partial\mathcal{D}$.

Obviously $\widetilde{Y}^j$ is  $\mathcal{F}_{t_{j}}$-measurable. Applying $(\mathbf{I}+\eps\tau\Delta^2)$ in both sides of \eqref{Discreteprocess1a} yields 
\begin{align}
\label{Discreteprocess1}
\widetilde{Y}^j-\widetilde{Y}^{j-1}=-\eps\tau\Delta^2\widetilde{Y}^j+\eps^{\gamma}\Delta_j\overline{W}\quad j=1,\cdots, J.
\end{align}
Setting $\widetilde{v}^j=-\eps\Delta \widetilde{Y}^j$, $j=0,\cdots,J$,
it follows from \eqref{Discreteprocess1} that $(\widetilde{Y}^j, \widetilde{v}^j)$ solves \eqref{linearscheme}. From the uniqueness of solution to \eqref{linearscheme}, it follows that $\widetilde{X}^j=\widetilde{Y}^j$, that is, $\widetilde{X}^0=0$ and
\begin{align}
\label{linearscheme1}
\widetilde{X}^j=(\mathbf{I}+\eps\tau\Delta^2)^{-1}\widetilde{X}^{j-1}+\eps^{\gamma}(\mathbf{I}+\eps\tau\Delta^2)^{-1}\Delta_j\overline{W}\quad \quad j=1,\cdots, J.
\end{align}
Using \eqref{linearscheme1} recursively and noting that $\widetilde{X}^0=0$ we obtain that
\begin{align}
\label{discreteconv}
\widetilde{X}^j=\eps^{\gamma} \sum_{i=0}^{j-1}(\mathbf{I}+\eps\tau\Delta^2)^{-(j-i)}\Delta_{i+1}\overline{W}\quad j=1,\cdots,J.
\end{align}
The above equivalent reformulation of \eqref{linearscheme} has been also used in the literature, see e.g., \cite{Larsson1,Larsson2}.
and can be viewed as the discrete counterpart of the stochastic convolution  $ \eps^{\gamma} \int_0^{t_j}e^{-\eps\Delta^2(t_j-s)}\mathrm{d}W(s)$, cf. \cite[(1.16)]{Debussche1}.

\begin{lemma}
 \label{regularitynoise}
Let $\alpha\in[0, 2)$.   Then for any $p\geq 1$ there exists a constant $C>0$ such that
\begin{itemize}
\item[(i)] $\max\limits_{1\leq j\leq J}\left(\mathbb{E}\left[\Vert \widetilde{X}^j\Vert_{\mathbb{H}^{\alpha}}^{2p}\right]\right)^{\frac{1}{2p}}\leq C\eps^{ \gamma-\frac{\alpha}{4}}h^{-\frac{d}{4}}$.
\item[(ii)]  $\left(\mathbb{E}\left[\max\limits_{1\leq j\leq J}\Vert\widetilde{X}^j\Vert^{2p}_{\mathbb{H}^{\alpha}}\right]\right)^{\frac{1}{2p}} \leq C\eps^{ \gamma-\frac{\alpha}{4}}h^{-\frac{d}{4}}$.
\end{itemize} 
\end{lemma}

\begin{proof}
 We denote $\mathcal{S}^j:=(\mathbf{I}+\eps\tau\Delta^2)^{-j}$ and $\mathcal{S}(t):=\mathcal{S}^j$ for $t\in[t_{j-1}, t_j)$.  Then we can write $\widetilde{X}^j$ as
\begin{align*}
\widetilde{X}^j&=\eps^{\gamma}\sum_{l=1}^{L}\int_0^T1\!\!1_{[0, t_j)}(s)\mathcal{S}(t_j-s)\phi_l d\beta_{l}(s)-\frac{\eps^{\gamma}}{\vert \mathcal{D}\vert}\sum_{l=1}^{L}\int_0^T1\!\!1_{[0, t_j)}(s)\mathcal{S}(t_j-s)(\phi_l,1) d\beta_{l}(s)\nonumber\\
&=:\widetilde{X}^j_{1}+\widetilde{X}^j_{2}\quad j=1,\cdots,J.
\end{align*}
 For a Banach space $E$, we denote by $\mathcal{L}(E)$  the space of bounded linear operators in $E$ and $\Vert .\Vert_{\mathcal{L}(E)}$  the operator norm in $\mathcal{L}(E)$.
From \cite[(2.10)]{Larsson1} we have   
\begin{align}
\label{Larssondiscrete}
\Vert (-\Delta)^{\frac{\alpha}{2}}\mathcal{S}(t)\Vert_{\mathcal{L}(\mathbb{L}^2)}&=\Vert (-\Delta)^{\frac{\alpha}{2}}(\mathbf{I}+\eps\tau\Delta^2)^{-j}\Vert_{\mathcal{L}(\mathbb{L}^2)}\nonumber\\
&\leq C\eps^{-\frac{\alpha}{4}}t_j^{-\frac{\alpha}{4}}\quad t\in[t_{j-1}, t_j),\quad j=1,\cdots, J,
\end{align}
where the constant $C$ is independent of $t$, $j$, $\tau$ and $\eps$.
Using the equivalence of norms $\Vert u\Vert_{\mathbb{H}^{\alpha}}\approx\Vert(-\Delta)^{\frac{\alpha}{2}}u\Vert_{\mathbb{L}^2}$, $\alpha\in[0, 2)$ (see e.g., \cite[Section 1.2]{Debussche1}), triangle inequality,  the Burkholder-Davis-Gundy inequality \cite [Theorem 4.36]{DaPratoZabczyk}, \eqref{Larssondiscrete} and \lemref{Lemmabasis} yields
\begin{align*}
\Vert \widetilde{X}_{1}^j\Vert^2_{L^{2p}(\Omega, \mathbb{H}^{\alpha})}&\leq C\Vert (-\Delta)^{\frac{\alpha}{2}} \revl{\widetilde{X}_{1}^j}\Vert^2_{L^{2p}(\Omega, \mathbb{L}^2)}\nonumber\\
&\leq \eps^{2\gamma}L\sum_{l=1}^{L}\left\Vert\int_0^T1\!\!1_{[0, t_j)}(s)(-\Delta)^{\frac{\alpha}{2}}\mathcal{S}(t_j-s)\phi_ld\beta_{l}(s)\right\Vert^2_{L^{2p}(\Omega, \mathbb{L}^2)}\nonumber\\
&\leq CL\eps^{2\gamma}\sum_{l=1}^{L}\left(\int_0^T\Vert1\!\!1_{[0, t_j)}(s)(-\Delta)^{\frac{\alpha}{2}}\mathcal{S}(t_j-s)\phi_l\Vert^2ds\right)\nonumber\\
&\leq C\eps^{2\gamma}L\sum_{l=1}^{L}\left(\sum_{i=0}^{J-1}\int_{t_i}^{t_{i+1}}\Vert 1\!\!1_{[0, t_j)}(s)(-\Delta)^{\frac{\alpha}{2}}\mathcal{S}(t_j-s)\Vert^2_{\mathcal{L}(\mathbb{L}^2)}\Vert\phi_l\Vert^2ds\right)\nonumber\\
&\leq C\eps^{2\gamma}h^{d}L^2\left(\tau\sum_{i=0}^{ j-1}\eps^{-\frac{\alpha}{2}}t_{j-i}^{-\frac{\alpha}{2}}\right)\leq C\eps^{2\gamma-\frac{\alpha}{2}}L^2h^d\leq C\eps^{2\gamma-\frac{\alpha}{2}}h^{-d}.
\end{align*}
Along the same lines as above, one obtains 
\begin{align*}
\Vert \widetilde{X}_{2}^j\Vert^2_{L^{2p}(\Omega, \mathbb{H}^{\alpha})}\leq  C\eps^{2\gamma-\frac{\alpha}{2}}h^{-d}.
\end{align*}
Summing the two preceding estimates  completes the proof of (i).

The proof of (ii) follows from (i) by the Doob martingale inequality \cite[Theorem 3.9]{DaPratoZabczyk}.
\end{proof}

We consider the following  subset of $\Omega$:
\begin{align}
\label{SetW}
\Omega_{\widetilde{W}}:=\left\{\omega\in\Omega:\; \max_{1\leq j\leq J}\Vert \widetilde{X}^j(\omega)\Vert_{\mathbb{L}^{\infty}}\leq C\eps^{\gamma-\eta-1} \right\},
\end{align}
where $\eta$ is defined in \assref{assumption2}.
Using \lemref{regularitynoise},  Markov's inequality and the embedding $\mathbb{H}^{\alpha}\hookrightarrow \mathbb{L}^{\infty}$ for $\alpha>\frac{d}{2}$, it follows that   $\lim\limits_{\eps\rightarrow 0}\mathbb{P}[\Omega_{\widetilde{W}}]=1$ if $\gamma>\eta+1$.
\revl{In addition, 
\begin{align}
\label{LinfinityXtilde}
\mathbb{E}\left[\max_{1\leq j\leq J}\Vert \widetilde{X}^j\Vert^r_{\mathbb{L}^{\infty}}\right]\leq  C\varepsilon^{(\gamma-\eta-1)r}\rightarrow 0\quad (\text{as}\; \varepsilon\rightarrow 0)\quad \forall r>0.
\end{align}
}

\revd{
Below we derive a $\mathbb{L}^{\infty}$-estimate for the numerical approximation $X^j$ \eqref{scheme1b}  on a smaller probability space $\Omega_{\mathcal{E}}$, where
\begin{align}
\label{SetOmegainfty}
\Omega_{\mathcal{E}}:=\left\{\omega\in\Omega:\; \max_{0\leq j\leq J} \mathcal{E}(X^j)\leq C\eps^{-\theta}\right\} \quad \text{for some}\; \theta> 0.
\end{align}
Using Chebyshev's inequality (see \cite[Theorem 3.14]{WalshBook2012}) and noting \lemref{moment}, we obtain
\begin{align*}
\mathbb{P}[\Omega_{\mathcal{E}}]=1-\mathbb{P}[\Omega_{\mathcal{E}}^c]\geq 1-\frac{\mathbb{E}\left[\max\limits_{1\leq j\leq J}\mathcal{E}(X^j)\right]}{\eps^{-\theta}}\geq 1-C\eps^{\theta}\rightarrow 1\quad \text{as}\; \eps\rightarrow 0. 
\end{align*}
In the next lemma we state the energy estimate of the numerical solution $\widetilde{X}^j$ \eqref{linearscheme}.
Its proof is a simpler variant of the proof of \lemref{moment}. 
\begin{lemma}
\label{momentatilde}
Let the assumptions of \lemref{moment} be fulfilled. Then   
\begin{align*}
\mathbb{E}\left[\max_{1\leq j\leq J}\mathcal{E}(\widetilde{X}^j)\right]+\frac{\tau}{2}\sum_{j-1}^j\mathbb{E}\Vert\nabla\widetilde{w}^j\Vert^2]\leq C, 
\end{align*}
where $\widetilde{X}^j$ is the numerical solution  in \eqref{linearscheme}.
\end{lemma}
Numerical scheme \eqref{randompdescheme} can be written in the following equivalent form
\begin{align}
\label{randompdescheme1}
\begin{array}{rlllll}
(d_t\widehat{X}^{j+1}, \varphi)+(\nabla \widehat{w}^{j+1}, \nabla\varphi)& = & 0&& \forall\varphi\in\mathbb{H}^1,\\
\eps(\nabla \widehat{X}^{j+1}, \nabla\psi)+\displaystyle\frac{1}{\eps}(f(X^{j+1}), \psi) & = &(\widehat{w}^{j+1}, \psi)&& \forall\psi\in\mathbb{H}^1,
\end{array}
\end{align} 
where  $d_t\widehat{X}^{j+1}:=(\widehat{X}^{j+1}-\widehat{X}^j)/\tau$ for $j=0,\cdots, J-1$.

Below we estimate the discrete time derivative $d_t\widehat{X}^{j+1}$.
\begin{lemma}
Let the assumptions of \lemref{moment} be fulfilled. Then 
\begin{align*}
\sum_{j=1}^J\tau\mathbb{E}[\Vert d_t\widehat{X}^{j}\Vert^2_{\mathbb{H}^{-1}}]\leq C.
\end{align*}
\end{lemma}
\begin{proof}
Using the first equation of \eqref{randompdescheme1}, it follows that
\begin{align}
\label{randomscheme2b}
\Vert d_t\widehat{X}^{j+1}\Vert_{\mathbb{H}^{-1}}=\sup_{0\not\equiv\varphi\in \mathbb{H}^1}\frac{(d_t\widehat{X}^{j+1}, \varphi)}{\Vert\varphi\Vert_{\mathbb{H}^1}}=\sup_{0\not\equiv\varphi\in \mathbb{H}^1}\frac{(\nabla\widehat{w}^{j+1}, \nabla\varphi)}{\Vert\varphi\Vert_{\mathbb{H}^1}}\leq C\Vert \nabla\widehat{w}^{j+1}\Vert.
\end{align}
Using \eqref{randomscheme2b}, noting that $\widehat{w}^j=w^j-\widetilde{w}^j$ and using Lemmas  \ref{momentatilde} and \ref{moment}, we obtain
\begin{align*}
\sum_{j=1}^J\tau\mathbb{E}[\Vert d_t\widehat{X}^{j}\Vert^2_{\mathbb{H}^{-1}}]\leq C\tau\sum_{j=1}^J\mathbb{E}[\Vert\nabla\widehat{w}^j\Vert^2]\leq C\tau\sum_{j=1}^J\mathbb{E}[\Vert\nabla w^j\Vert^2]+C\tau\sum_{j=1}^J\mathbb{E}[\Vert\nabla\widetilde{w}^j\Vert^2]\leq C.
\end{align*}
\end{proof}
In the following lemma we provide an estimate of $\Vert \Delta\widehat{X}^j\Vert$ on the probability space $\Omega_{\mathcal{E}}$. To reduce the number of parameters we assume without loss of generality
that the initial condition satisfies $\|u_0^{\eps}\|_{\mathbb{H}^2} \leq C \eps^{-\mathfrak{p}_{\mathrm{CH}}}$ with $2\mathfrak{p}_{\mathrm{CH}} < 5$  (cf. Lemma~\ref{LemmaLubo19}) in the remainder of the paper.
\begin{lemma}
\label{EstiXhat}
Let the assumptions of \lemref{moment} and \lemref{LemmaLubo19} be fulfilled.
Then the following estimates hold
\begin{itemize}
\item[i)] $\mathbb{E}\left[\max\limits_{1\leq j\leq J}1\!\!1_{\Omega_{\mathcal{E}}}\Vert\Delta^{-1}d_t\widehat{X}^j\Vert^2 \right]+\eps\tau\sum\limits_{j=1}^J\mathbb{E}[1\!\!1_{\Omega_{\mathcal{E}}}\Vert d_t\widehat{X}^j\Vert^2]\leq C\eps^{-2\theta-5}$,
\item[ii)] $\mathbb{E}\left[\max\limits_{0\leq j\leq J}1\!\!1_{\Omega_{\mathcal{E}}}\Vert\Delta\widehat{X}^j\Vert^2\right]\leq C\eps^{-2\theta-7}$.
\end{itemize}
\end{lemma}
\begin{proof}
i) Applying the difference operator  $d_t$ to \eqref{randompdescheme1}, yields for $j=0,\cdots, J-1$
\begin{align}
\label{randompdescheme2}
\begin{array}{rlllll}
(d^2_t\widehat{X}^{j+1}, \varphi)+(\nabla d_t\widehat{w}^{j+1}, \nabla\varphi)& = & 0&& \forall\varphi\in\mathbb{H}^1,\\
\eps(\nabla d_t\widehat{X}^{j+1}, \nabla\psi)+\displaystyle\frac{1}{\eps}\left(d_tf(X^{j+1}), \psi\right) & = &(d_t\widehat{w}^{j+1}, \psi)&& \forall\psi\in\mathbb{H}^1,
\end{array}
\end{align} 
where for $j=0$ we introduce $\widehat{X}^{-1}\in\mathbb{H}^1$, such that $\int_{\mathcal{D}}\widehat{X}^{-1}dx=0$, as the solution of
\begin{align*}
\left(\Delta^{-1}d_t\widehat{X}^0, \varphi\right)= (w^0, \varphi) =\left(-\eps\Delta\widehat{X}^0+\frac{1}{\eps}f(\widehat{X}^0), \varphi\right),
\end{align*}
for all $\varphi\in\{\chi\in\mathbb{H}^1:\; (\chi,1)=0\}$.

Taking $\varphi=\Delta^{-2}d_t\widehat{X}^{j+1}$ and $\psi=-\Delta^{-1}d_t\widehat{X}^{j+1}$ in \eqref{randompdescheme2} and summing the resulting equations, we obtain
\begin{align*}
\frac{1}{2}d_t\Vert \Delta^{-1}d_t\widehat{X}^{j+1}\Vert^2+\frac{\tau}{2}\Vert\Delta^{-1}d_t^2\widehat{X}^{j+1}\Vert^2+\eps\Vert d_t \widehat{X}^{j+1}\Vert^2=\frac{1}{\eps}\left(d_tf(X^{j+1}), \Delta^{-1}d_t\widehat{X}^{j+1}\right).
\end{align*}
 Using the mean value theorem for $d_tf(X^{j+1})$, yields
 \begin{align}
 \label{randomscheme2aaa}
 &\frac{1}{2}d_t\Vert \Delta^{-1}d_t\widehat{X}^{j+1}\Vert^2+\frac{\tau}{2}\Vert\Delta^{-1}d_t^2\widehat{X}^{j+1}\Vert^2+\eps\Vert d_t\widehat{X}^{j+1}\Vert^2\nonumber\\
 &\quad=\frac{1}{\eps}\left(R_f(X^{j+1})d_tX^{j+1}, \Delta^{-1}d_t\widehat{X}^{j+1}\right),
 \end{align}
where 
\begin{align*}
R_f(X^{j+1})=\int_0^1f'\left(sX^j+(1-s)X^{j+1}\right)ds.
\end{align*}

 Using Young's inequality, it follows from \eqref{randomscheme2aaa} that
\begin{align}
\label{randomscheme2a}
\frac{1}{2}d_t\Vert \Delta^{-1}d_t\widehat{X}^{j+1}\Vert^2+\eps\Vert d_t\widehat{X}^{j+1}\Vert^2&\leq \frac{\eps}{2}\Vert d_tX^{j+1}\Vert^2+\frac{C}{\eps^3}\Vert R_f(X^{j+1}) \Delta^{-1}d_t\widehat{X}^{j+1}\Vert^2\nonumber\\
&\leq \frac{\eps}{2}\Vert d_tX^{j+1}\Vert^2+\frac{C}{\eps^3}\Vert R_f(X^{j+1})\Vert^2_{\mathbb{L}^{3}}\Vert \Delta^{-1} d_t\widehat{X}^{j+1}\Vert^2_{\mathbb{L}^6}.
\end{align}
Noting that $f'(u)=u^2-1$ and using the Sobolev embedding $\mathbb{H}^1\hookrightarrow\mathbb{L}^6$, we obtain
\begin{align*}
\Vert R_f(X^{j+1})\Vert^2_{\mathbb{L}^{3}}&\leq C\int_0^1\left(s\Vert X^j\Vert^4_{\mathbb{L}^6}+(1-s)\Vert X^{j+1}\Vert^4_{\mathbb{L}^6}+1\right)ds\nonumber\\
&\leq C\int_0^1\left(s\Vert\nabla X^j\Vert^4+(1-s)\Vert\nabla X^{j+1}\Vert^4+1\right)ds\nonumber\\
&\leq C\eps^{-2}\left(\mathcal{E}(X^j)^2+\mathcal{E}(X^{j+1})^2+1\right).
\end{align*}
Substituting the preceding estimate into \eqref{randomscheme2a}, using the Sobolev embedding $\mathbb{H}^1\hookrightarrow\mathbb{L}^6$ and Poincar\'{e} inequality, yields
\begin{align*}
\frac{1}{2}d_t\Vert \Delta^{-1}d_t\widehat{X}^{j+1}\Vert^2+\frac{\eps}{2}\Vert d_t\widehat{X}^{j+1}\Vert^2\leq C\eps^{-5}\left(\mathcal{E}(X^j)^2+\mathcal{E}(X^{j+1})^2+1\right)\Vert d_t\widehat{X}^{j+1}\Vert^2_{\mathbb{H}^{-1}}
\end{align*}
Noting \eqref{SetOmegainfty}, it follows from the preceding estimate that
\begin{align}
\label{randomscheme2aa}
\frac{1}{2}1\!\!1_{\Omega_{\mathcal{E}}} d_t\Vert \Delta^{-1}d_t\widehat{X}^{j+1}\Vert^2+\frac{\eps}{2}1\!\!1_{\Omega_{\mathcal{E}}}\Vert d_t\widehat{X}^{j+1}\Vert^2\leq C\eps^{-2\theta-5}1\!\!1_{\Omega_{\mathcal{E}}}\Vert d_t\widehat{X}^{j+1}\Vert^2_{\mathbb{H}^{-1}}
\end{align}

Substituting \eqref{randomscheme2b} into \eqref{randomscheme2aa}, summing the resulting estimate over $j=0, \cdots,k$ and multiplying by $\tau$, we obtain
\begin{align*}
\frac{1}{2}1\!\!1_{\Omega_{\mathcal{E}}}\Vert\Delta^{-1}d_t\widehat{X}^k\Vert^2+\frac{\eps\tau}{2}\sum_{j=1}^k1\!\!1_{\Omega_{\mathcal{E}}}\Vert d_t\widehat{X}^j\Vert^2\leq C\eps^{-2\theta-5}\tau\sum_{j=1}^k1\!\!1_{\Omega_{\mathcal{E}}}\Vert\nabla\widehat{w}^{j}\Vert^2+\frac{1}{2}1\!\!1_{\Omega_{\mathcal{E}}}\Vert \Delta^{-1}d_t\widehat{X}^0\Vert^2. 
\end{align*}
Taking the maximum over $1\leq k\leq J$, taking the expectation, noting that $\widehat{w}^j=w^j-\widetilde{w}^j$, using Lemmas \ref{moment}  and \ref{momentatilde}, yields
\begin{align*}
\mathbb{E}\left[\max_{1\leq j\leq J}1\!\!1_{\Omega_{\mathcal{E}}}\Vert\Delta^{-1}d_t\widehat{X}^j\Vert^2 \right]+\eps\tau\sum_{j=1}^J\mathbb{E}[1\!\!1_{\Omega_{\mathcal{E}}}\Vert d_t\widehat{X}^j\Vert^2]&\leq C\eps^{-2\theta-5}\tau\sum_{j=1}^J\mathbb{E}\left[\Vert\nabla\widehat{w}^{j}\Vert^2\right]+C\eps^{-\mathfrak{p}_{\mathrm{CH}}}\nonumber\\
&\leq C\eps^{-2\theta-5}.
\end{align*}

ii)
Taking $\varphi=\widehat{X}^{j+1}$ and $\psi=-\Delta\widehat{X}^{j+1}$ in \eqref{randompdescheme1} and summing the resulting equations, we obtain
\begin{align}
\label{randomscheme2c}
\eps\Vert\Delta\widehat{X}^{j+1}\Vert^2+(d_t\widehat{X}^{j+1}, \widehat{X}^{j+1})+\frac{1}{\eps}\left(f(X^{j+1})\nabla\widehat{X}^{j+1}, \nabla\widehat{X}^{j+1}\right)=0.
\end{align}
Using Young's inequality, we obtain
\begin{align}
\label{randomscheme2d}
(d_t\widehat{X}^{j+1}, \widehat{X}^{j+1})&=(d_t\widehat{X}^{j+1}, \Delta^{-1}\Delta\widehat{X}^{j+1})=(\Delta^{-1}d_t\widehat{X}^{j+1}, \Delta\widehat{X}^{j+1})\nonumber\\
&\leq \frac{\eps}{2}\Vert\Delta\widehat{X}^{j+1}\Vert^2+\frac{1}{2\eps}\Vert\Delta^{-1}d_t\widehat{X}^{j+1}\Vert^2.
\end{align}
Substituting \eqref{randomscheme2d} into \eqref{randomscheme2c} and using the fact that $-(f'(u)v,v)\leq \Vert v\Vert^2$, yields
\begin{align*}
\eps\Vert\Delta\widehat{X}^{j+1}\Vert^2&\leq \frac{1}{\eps}\Vert\Delta^{-1}d_t\widehat{X}^{j+1}\Vert^2-\frac{2}{\eps}\left(f(X^{j+1})\nabla\widehat{X}^{j+1}, \nabla\widehat{X}^{j+1}\right)\nonumber\\
&\leq \frac{1}{\eps}\Vert\Delta^{-1}d_t\widehat{X}^{j+1}\Vert^2+\frac{2}{\eps}\Vert\nabla \widehat{X}^{j+1}\Vert^2.
\end{align*}
Taking the maximum over $j=0, \cdots, J-1$, taking the expectation, using part i), noting that $\widehat{X}^j=X^j-\widetilde{X}^j$, using  Lemmas \ref{moment}  and \ref{momentatilde}, we obtain 
\begin{align*}
\eps\mathbb{E}\left[\max_{0\leq j\leq J}1\!\!1_{\Omega_{\mathcal{E}}}\Vert\Delta\widehat{X}^j\Vert^2\right]\leq \frac{1}{\eps}\mathbb{E}\left[\max_{0\leq j\leq J}1\!\!1_{\Omega_{\mathcal{E}}}\Vert\Delta^{-1}d_t\widehat{X}^{j}\Vert^2\right]+\frac{2}{\eps}\mathbb{E}\left[\max_{1\leq j\leq J}\Vert\nabla\widehat{X}^j\Vert^2\right]\leq C\eps^{-2\theta-6}.
\end{align*}
\end{proof}
\begin{lemma}
\label{Linfinitya}
Let the assumptions of \lemref{moment} be fulfilled. Let $X^j$ be the solution to \eqref{scheme1b}. Then it holds
\begin{align*}
\mathbb{E}\left[\max_{1\leq j\leq J}1\!\!1_{\Omega_{\mathcal{E}}}\Vert X^j\Vert^2_{\mathbb{L}^{\infty}}\right]\leq C\eps^{-2\theta-7}.
\end{align*}
\end{lemma}
\begin{proof}
Using triangle inequality, \eqref{LinfinityXtilde}, the Sobolev embedding $\mathbb{H}^2\hookrightarrow\mathbb{L}^{\infty}$, the elliptic regularity of the Laplace operator and \lemref{EstiXhat} ii), we obtain
\begin{align*}
\mathbb{E}\left[\max_{1\leq j\leq J}1\!\!1_{\Omega_{\mathcal{E}}}\Vert X^j\Vert^2_{\mathbb{L}^{\infty}}\right]&\leq \mathbb{E}\left[\max_{1\leq j\leq J}1\!\!1_{\Omega_{\mathcal{E}}}\Vert \widehat{X}^j\Vert^2_{\mathbb{L}^{\infty}}\right]+\mathbb{E}\left[\max_{1\leq j\leq J}1\!\!1_{\Omega_{\mathcal{E}}}\Vert \widetilde{X}^j\Vert^2_{\mathbb{L}^{\infty}}\right]\nonumber\\
&\leq C\mathbb{E}\left[\max_{1\leq j\leq J}1\!\!1_{\Omega_{\mathcal{E}}}\Vert \Delta\widehat{X}^j\Vert^2\right]+C\eps^{(\gamma-\eta-1)2}\leq C\eps^{-2\theta-7}.
\end{align*}
\end{proof}
}

\revd{
We introduce the following subset of $\Omega$
\begin{align}
\label{SetXaa}
\Omega_{\infty}:=\left\{\omega\in\Omega:\; \max_{1\leq j\leq J}1\!\!1_{\Omega_{\mathcal{E}}}\Vert X^j\Vert_{\mathbb{L}^{\infty}}\leq \kappa\right\},\quad \text{where}\; \kappa=\eps^{-\theta-4}.
\end{align}
It follows by Markov's inequality that
\begin{align*}
\mathbb{P}[\Omega_{\infty}]=1-\mathbb{P}[\Omega_{\infty}^c]\geq 1-\frac{\mathbb{E}\left[\max\limits_{1\leq j\leq J}1\!\!1_{\Omega_{\mathcal{E}}}\Vert X^j\Vert^{2}_{\mathbb{L}^{\infty}}\right]}{\kappa^{2}}. 
\end{align*}
Noting \lemref{Linfinitya} we deduce that  $\lim\limits_{\eps\rightarrow 0}\mathbb{P}[\Omega_{\infty}]=1$. 

We also introduce the following subset of $\Omega$
\begin{align}
\label{SetXa}
\Omega_{\kappa, J}=\Omega_{\infty}\cap\Omega_{\mathcal{E}}.
\end{align}
From the identity $\Omega_{\infty}=(\Omega_{\infty}\cap\Omega_{\mathcal{E}})\cup(\Omega_{\infty}\cap\Omega_{\mathcal{E}}^c)$ we get that $\mathbb{P}[\Omega_{\mathcal{E}}\cap\Omega_{\infty}]=\mathbb{P}[\Omega_{\infty}]-\mathbb{P}[\Omega_{\infty}\cap\Omega_{\mathcal{E}}^c]\geq \mathbb{P}[\Omega_{\infty}]-\mathbb{P}[\Omega_{\mathcal{E}}^c]$. Since $\lim\limits_{\eps\rightarrow 0}\mathbb{P}[\Omega_{\mathcal{E}}^c]=0$ we conclude that $\lim\limits_{\eps\rightarrow 0}\mathbb{P}[\Omega_{\kappa,J}]=\lim\limits_{\eps\rightarrow 0}\mathbb{P}[\Omega_{\mathcal{E}}\cap\Omega_{\infty}]=1$. 

Along the same lines as above, we have $\lim\limits_{\eps\rightarrow 0}\mathbb{P}[\Omega_{\widetilde{W}}\cap\Omega_{\kappa, J}]=1$.

}

\subsection{ $\mathbb{L}^{\infty}$-error estimate}
\label{SectionLinfinityZ}
In this section we derive an estimate of
the error  ${Z}^j = X^j - X_{\mathrm{CH}}^j$ in the $\mathbb{L}^\infty$-norm on the subset $\Omega_{\widetilde{W}}\cap\Omega_{\kappa, J}$, see \thmref{LinfinityZ} below.
\revl{The estimate is obtained by using the identity $X^j = \widehat{X}^j + \widetilde{X}^j$ and splitting
${Z}^j = \widehat{X}^j + \widetilde{X}^j - X_{\mathrm{CH}}^j = \widehat{Z}^j + \widetilde{X}^j$, see \eqref{ErrorZ}, \eqref{translateddifference}.
We use \eqref{LinfinityXtilde} to control the perturbation $\widetilde{X}^j$ and the error $\widehat{Z}^j$ is estimated in several steps below.}

We start by estimating $\widehat{Z}^j$ in stronger norms on the subset $\Omega_{\widetilde{W}}$ in \lemref{LemmaNorm2} below.
\revl{
To ensure that the right-hand side in the estimate in the lemma vanishes for $\eps \rightarrow 0$,
we require additional conditions to be satisfied.
In comparison to  Assumptions \ref{assumption1} and \ref{assumption2} required for \thmref{mainresult1} we need
a smaller time-step size $\tau$ and larger value of $\sigma_0$, which implies that only larger value of $\gamma$ are admissible .}
\begin{Assumption}
\label{assumption3}
Let Assumptions \ref{assumption1} and \ref{assumption2} hold. In addition, assume that
\begin{align*}
\sigma_0>\kappa_0+4\mathfrak{n}_{\mathrm{CH}}+16,\quad \gamma>\max\left\{6d\eta+\kappa_0+4\mathfrak{n}_{\mathrm{CH}}+47, 6d\eta+\kappa_0+4\mathfrak{n}_{\mathrm{CH}}+45\right\},
\end{align*}
and that the time-step satisfies
\begin{align*}
\tau\leq C\min\left\{\eps^{\mathfrak{l}_{\mathrm{CH}}}, \eps^{\frac{\kappa_0}{2}+2\mathfrak{n}_{\mathrm{CH}}+10}, \eps^{2+\frac{\kappa_0}{2}+\beta}\right\}\quad \eps\in (0, \eps_0),
\end{align*}
\revl{for sufficiently small $\eps_0\equiv\eps_0(\sigma_0, \kappa_0)>0$,  $\mathfrak{l}_{\mathrm{CH}}\geq 3$  and an arbitrarily $0<\beta<\frac{1}{2}$).}
\end{Assumption}

\begin{lemma}
\label{LemmaNorm2}
Let  \assref{assumption3}  be fulfilled.  Then there exists a constant  $C$ such that
\begin{align*}
&\mathbb{E}\left[\max_{1\leq j\leq J}1\!\!1_{\Omega_{\widetilde{W}}}\Vert\widehat{Z}^j\Vert^2\right]+\mathbb{E}\left[\sum_{j=1}^J1\!\!1_{\Omega_{\widetilde{W}}}\Vert\widehat{Z}^j-\widehat{Z}^{j-1}\Vert^2+\eps\tau\sum_{j=1}^J1\!\!1_{\Omega_{\widetilde{W}}}\Vert\Delta\widehat{Z}^j\Vert^2\right]\nonumber\\
&\quad+\frac{\tau}{\eps}\sum_{j=1}^J\mathbb{E}\left[1\!\!1_{\Omega_{\widetilde{W}}}\Vert\widehat{Z}^j\nabla\widehat{Z}^j\Vert^2+1\!\!1_{\Omega_{\widetilde{W}}}\Vert X^j_{\mathrm{CH}}\nabla\widehat{Z}^j\Vert^2\right]
\\
&\leq \mathcal{F}_1(\tau, d, \eps; \sigma_0, \kappa_0, \gamma, \eta)
\\
&:=\left(\frac{C}{\eps^{\kappa_0+4\mathfrak{n}_{\mathrm{CH}}+16}}\max\left(\eps^{\sigma_0}, \eps^{2\gamma-2d\eta},\eps^{\gamma-d\eta+\frac{\sigma_0+1}{3}}, \frac{\tau^2}{\eps^{4}}\right)\right)^{\frac{1}{2}}\nonumber\\
&\quad+C\max\left\{ \eps^{2\gamma-2\eta-7}, \eps^{6\gamma-2\eta-d\eta-7}, \eps^{6\gamma-\frac{3d\eta}{2}-\frac{9}{2}}, \eps^{\frac{17}{2}\gamma-2d\eta-5},\eps^{\frac{\gamma}{2}-5},\eps^{\frac{7}{2}\gamma-d\eta-3}, \eps^{2\gamma-\frac{d\eta}{2}-2\mathfrak{n}_{\mathrm{CH}}-\frac{3}{2}}\right\}.
\end{align*}
\end{lemma}

\begin{proof} 
Recall that $\widehat{Z}^j=\widehat{X}^j-X^j_{\mathrm{CH}}$. Since $\widehat{X}^j$ satisfies \eqref{randompdescheme}, we deduce that $\widehat{Z}^j$ satisfies
\begin{align}
\label{scheme4} 
(\widehat{Z}^j-\widehat{Z}^{j-1}, \varphi)&=\tau(\nabla(\widehat{w}^j-w^j_{\mathrm{CH}}), \nabla\varphi) & \varphi\in\mathbb{H}^1\\
\label{scheme5}
(\widehat{w}^j-w^j_{\mathrm{CH}}, \psi)+\eps(\nabla \widehat{Z}^j, \nabla\psi)&=\frac{1}{\eps}(f(\widehat{X}^j+\widetilde{X}^j)-f(X^j_{\mathrm{CH}}), \psi) & \psi\in\mathbb{H}^1. 
\end{align}
Taking $\varphi=\widehat{Z}^j$ in \eqref{scheme4}, $\psi=\Delta\widehat{Z}^j$ in \eqref{scheme5}, integrating by parts and summing the resulting equations  yields
\begin{align}
\label{Norm2}
&\frac{1}{2}\left(\Vert \widehat{Z}^j\Vert^2-\Vert\widehat{Z}^{j-1}\Vert^2+\Vert\widehat{Z}^j-\widehat{Z}^{j-1}\Vert^2\right)+\eps\tau\Vert\Delta\widehat{Z}^j\Vert^2\nonumber\\
&+\frac{\tau}{\eps}\left(f(\widehat{X}^j+\widetilde{X}^j)-f(X^j_{\mathrm{CH}}), -\Delta\widehat{Z}^j\right)=0.
\end{align}
Splitting the term involving the nonlinearity in two parts, using Cauchy-Schwarz and Young's inequalities in the first  term yields
\begin{align}
\label{Roman0}
&\frac{\tau}{\eps}\left(f(\widehat{X}^j+\widetilde{X}^j)-f(X^j_{\mathrm{CH}}), -\Delta\widehat{Z}^j\right)\nonumber\\
&=\frac{\tau}{\eps}\left(f(\widehat{X}^j+\widetilde{X}^j)-f(\widehat{X}^j), -\Delta\widehat{Z}^j\right)+\frac{\tau}{\eps}\left(f(\widehat{X}^j)-f(X^j_{\mathrm{CH}}), -\Delta\widehat{Z}^j\right)\\
&\leq \frac{\eps\tau}{4}\Vert\Delta\widehat{Z}^j\Vert^2+\frac{C\tau}{\eps^3}\Vert f(\widehat{X}^j+\widetilde{X}^j)-f(\widehat{X}^j)\Vert^2+\frac{\tau}{\eps}\left(f(\widehat{X}^j)-f(X^j_{\mathrm{CH}}), -\Delta\widehat{Z}^j\right). \nonumber
\end{align}
Along the same lines as in \cite[Page 533]{Banas19}, one obtains
\begin{align}
\label{Roman1}
\frac{\tau}{\eps}\left(f(\widehat{X}^j)-f(X^j_{\mathrm{CH}}), -\Delta\widehat{Z}^j\right)\geq \frac{\tau}{2\eps}\left[\Vert\widehat{Z}^j\nabla\widehat{Z}^j\Vert^2+\Vert X^j_{\mathrm{CH}}\nabla \widehat{Z}^j\Vert^2\right]-\frac{C\tau}{\eps^{1+2\mathfrak{n}_{\mathrm{CH}}}}\Vert\nabla\widehat{Z}^j\Vert^2.
\end{align}
Using the identity \eqref{nonlinearident}, Young's inequality, the Sobolev embeddings $\mathbb{H}^1\hookrightarrow\mathbb{L}^q$ ($1\leq q\leq 6$) and Poincar\'{e}'s inequality, it follows that
\begin{align}
\label{Roman2}
&\Vert f(\widehat{X}^j+\widetilde{X}^j)-f(\widehat{X}^j) \Vert^2=\Vert 3\widetilde{X}^j(\widehat{X}^j)^2-\widetilde{X}^j+(\widetilde{X}^j)^3-3(\widetilde{X}^j)^2\widehat{X}^j\Vert^2\nonumber\\
&\leq C\Vert \widetilde{X}^j(\widehat{X}^j)^2\Vert^2+C\Vert\widetilde{X}^j\Vert^2+C\Vert(\widetilde{X}^j)^3\Vert^2+C\Vert(\widetilde{X}^j)^2\widehat{X}^j\Vert^2\nonumber\\
&\leq C\Vert\widetilde{X}^j\Vert^2_{\mathbb{L}^{\infty}}\Vert\widehat{X}^j\Vert^4_{\mathbb{L}^4}+C\Vert\widetilde{X}^j\Vert^2+ C\Vert \widetilde{X}^j\Vert^6_{\mathbb{L}^6}+C\eps^{-\frac{\gamma}{2}}\Vert \widetilde{X}^j\Vert^8_{\mathbb{L}^{\infty}}+C\eps^{\frac{\gamma}{2}}\Vert\widehat{X}^j\Vert^4\\
&\leq C\Vert\widetilde{X}^j\Vert^2_{\mathbb{L}^{\infty}}\Vert\nabla\widehat{X}^j\Vert^4+C\eps^{-\frac{\gamma}{2}}\Vert\widetilde{X}^j\Vert^8_{\mathbb{L}^{\infty}}+C\Vert\widetilde{X}^j\Vert^6_{\mathbb{H}^1}+C\eps^{\frac{\gamma}{2}}\Vert\widehat{X}^j\Vert^4.\nonumber
\end{align}
Substituting \eqref{Roman2} and \eqref{Roman1} in \eqref{Roman0}, substituting the resulting estimate in \eqref{Norm2}, summing over $1\leq j\leq J$,  multiplying both sides by $1\!\!1_{\Omega_{\widetilde{W}}}$, taking the maximum,  the expectation in both sides,    recalling the definition of $\Omega_{\widetilde{W}}$ \eqref{SetW} and using \lemref{MomentLemma}  leads to
\begin{align*}
&\mathbb{E}\left[\max_{1\leq j\leq J}1\!\!1_{\Omega_{\widetilde{W}}}\Vert \widehat{Z}^j\Vert^2\right]+\sum_{j=1}^J\mathbb{E}\left[1\!\!1_{\Omega_{\widetilde{W}}}\Vert \widehat{Z}^j-\widehat{Z}^{j-1}\Vert^2\right]+\frac{\eps\tau}{4}\sum_{j=1}^J\mathbb{E}\left[1\!\!1_{\Omega_{\widetilde{W}}}\Vert\Delta\widehat{Z}^j\Vert^2\right]\nonumber\\
&\quad+\frac{\tau}{2\eps}\sum_{j=1}^J\mathbb{E}\left[1\!\!1_{\Omega_{\widetilde{W}}}\Vert \widehat{Z}^j\nabla\widehat{Z}^j\Vert^2\right]+\frac{\tau}{2\eps}\sum_{j=1}^J\mathbb{E}\left[1\!\!1_{\Omega_{\widetilde{W}}}\Vert X^j_{\mathrm{CH}}\nabla\widehat{Z}^j\Vert^2\right]\nonumber\\
&\leq\frac{C\tau}{\eps^3}\sum_{j=1}^J\mathbb{E}\left[\Vert \widetilde{X}^j\Vert^2_{\mathbb{L}^{\infty}}\Vert \nabla\widehat{X}^j\Vert^4\right]+\frac{C\tau}{\eps^3}\sum_{j=1}^J\mathbb{E}\left[\Vert\widetilde{X}^j\Vert^6_{\mathbb{H}^1}\right]+\frac{C\tau}{\eps^{3+\frac{\gamma}{2}}}\sum_{j=1}^J\mathbb{E}\left[1\!\!1_{\Omega_{\widetilde{W}}}\Vert \widetilde{X}^j\Vert^8_{\mathbb{L}^{\infty}}\right]\nonumber\\
&\quad+\frac{C\tau}{\eps^{3-\frac{\gamma}{2}}}\sum_{j=1}^J\mathbb{E}\left[\Vert\widehat{X}^j\Vert^4\right]+\frac{C\tau}{\eps^{1+2\mathfrak{n}_{\mathrm{CH}}}}\sum_{j=1}^J\mathbb{E}\left[\Vert\nabla\widehat{Z}^j\Vert^2\right]\nonumber\\
&\leq C\left(\eps^{2\gamma-2\eta-7}+\eps^{6\gamma-2\eta-d\eta-7}+\eps^{6\gamma-\frac{3d\eta}{2}-\frac{9}{2}}+\eps^{\frac{17}{2}\gamma-2d\eta-5}+\eps^{\frac{\gamma}{2}-5}+\eps^{\frac{7}{2}\gamma-d\eta-3}\right)\nonumber\\
&\quad+\frac{C\tau}{\eps^{1+2\mathfrak{n}_{\mathrm{CH}}}}\sum_{j=1}^J\mathbb{E}[\Vert\nabla\widehat{Z}^j\Vert^2], 
\end{align*}
where at the last step we used the inequalities $\Vert\widehat{X}^j\Vert^4=\Vert X^j-\widetilde{X}^j\Vert^4\leq 8\Vert X^j\Vert^4+ 8\Vert\widetilde{X}^j\Vert^4$,  $\Vert\nabla\widehat{X}^j\Vert^4\leq  8\Vert\nabla X^j\Vert^4+ 8\Vert\nabla\widetilde{X}^j\Vert^4$, Poincar\'{e}'s inequality $\Vert X^j\Vert\leq C_{\mathcal{D}}\Vert\nabla X^j\Vert$, Lemmas \ref{MomentLemma} and \ref{regularitynoise} with $\alpha=0,1$ and noting that $h=\eps^{\eta}$. 

Using the inequality $\Vert\nabla\widehat{Z}^j\Vert^2\leq 2\Vert \nabla Z^j\Vert^2+2\Vert \nabla\widetilde{X}^j\Vert^2$, Lemmas \ref{EstiZ} and \ref{regularitynoise} with $\alpha=1$ and noting again that $h=\eps^{\eta}$  yields the desired result. 
\end{proof}

Using \lemref{LemmaNorm2} we estimate
the error $\widehat{Z}^j$ in stronger norms on a smaller probability space $\Omega_{\widetilde{W}}\cap\Omega_{\kappa, J}$.
\begin{lemma}
\label{EstinablaZ}
Let \assref{assumption3} hold. Then the following error estimate holds
\begin{align*}
&\mathbb{E}\left[\max_{1\leq j\leq J}1\!\!1_{\Omega_{\widetilde{W}}\cap\Omega_{\kappa, J}}\Vert\nabla\widehat{Z}^j\Vert^2\right]+\sum_{j=1}^J\mathbb{E}\left[1\!\!1_{\Omega_{\widetilde{W}}\cap\Omega_{\kappa, J}}\Vert\widehat{Z}^j-\widehat{Z}^{j-1}\Vert^2+\eps\tau1\!\!1_{\Omega_{\widetilde{W}}\cap\Omega_{\kappa, J}}\Vert\nabla\Delta\widehat{Z}^j\Vert^2\right]\nonumber\\
&\leq \frac{C(1+\kappa^4+\eps^{-2\mathfrak{n}_{\mathrm{CH}}})}{\eps^3}\eps^{2\gamma-\frac{1}{2}-\frac{d\eta}{2}}
   + C\left(\frac{(1+\kappa^2)}{\eps^4}\eps^{-2\mathfrak{n}_{\mathrm{CH}}}+\frac{C(1+\kappa^2)}{\eps^2}\right)\mathcal{F}_1(\tau, d, \eps, \sigma_0, \kappa_0, \gamma, \eta)\nonumber\\
&=:\mathcal{F}_2(\tau, d, \eps; \sigma_0, \kappa_0, \gamma, \eta).
\end{align*}
\end{lemma}

\begin{proof} 
Taking $\varphi=-\Delta\widehat{Z}^j(\omega)$ in \eqref{scheme4}, $\psi=\Delta^2\widehat{Z}^j(\omega)$ in \eqref{scheme5}, with $\omega\in \Omega$ fixed, integrating by parts and summing the resulting equations yields
\begin{align}
\label{scheme6}
&\frac{1}{2}\left(\Vert\nabla\widehat{Z}^j\Vert^2-\Vert\nabla\widehat{Z}^{j-1}\Vert^2+\Vert\nabla(\widehat{Z}^j-\widehat{Z}^{j-1})\Vert^2\right)+\eps\tau\Vert\nabla\Delta\widehat{Z}^j\Vert^2\nonumber\\
&=\frac{\tau}{\eps}\left(\nabla\left(f(\widehat{X}^j+\widetilde{X}^j)-f(X^j_{\mathrm{CH}})\right), \nabla\Delta\widehat{Z}^j\right).
\end{align}
Splitting the term involving $f$ and using Cauchy-Schwarz's inequality leads to
\begin{align}
\label{Sam6}
&\frac{\tau}{\eps}\left(\nabla\left(f(\widehat{X}^j+\widetilde{X}^j)-f(X^j_{\mathrm{CH}})\right), \nabla\Delta\widehat{Z}^j\right)\nonumber\\
&=\frac{\tau}{\eps}\left(\nabla\left(f(\widehat{X}^j+\widetilde{X}^j)-f(X^j_{\mathrm{CH}}+\widetilde{X}^j)\right), \nabla\Delta\widehat{Z}^j\right)\nonumber\\
&\quad+\frac{\tau}{\eps}\left(\nabla\left(f(X^j_{\mathrm{CH}}+\widetilde{X}^j)-f(X^j_{\mathrm{CH}})\right), \nabla\Delta\widehat{Z}^j\right)\nonumber\\
&\leq \frac{\eps\tau}{4}\Vert \nabla\Delta\widehat{Z}^j\Vert^2+\frac{C\tau}{\eps^3}\left\Vert\nabla\left(f(\widehat{X}^j+\widetilde{X}^j)-f(X^j_{\mathrm{CH}}+\widetilde{X}^j)\right)\right\Vert^2\\
&\quad+\frac{C\tau}{\eps^3}\left\Vert \nabla\left(f(X^j_{\mathrm{CH}}+\widetilde{X}^j)-f(X^j_{\mathrm{CH}})\right)\right\Vert^2=: \frac{\eps\tau}{4}\Vert \nabla\Delta\widehat{Z}^j\Vert^2+I+II.\nonumber
\end{align}
Let us start with the estimate of $II$. 
Using the identity \eqref{nonlinearident} and the elementary inequality $(a+b+c+d)^2\leq 4a^2+4b^2+4c^2+4d^2$ leads to
\begin{align}
\label{Vidal1}
II=&\frac{C\tau}{\eps^3}\Vert\nabla(f(X^j_{\mathrm{CH}}+\widetilde{X}^j)-f(X^j_{\mathrm{CH}}))\Vert^2=\frac{C\tau}{\eps^3}\int_{\mathcal{D}}\vert \nabla(f(X^j_{\mathrm{CH}}+\widetilde{X}^j)-f(X^j_{\mathrm{CH}}))\vert^2dx\nonumber\\
\leq& \frac{C\tau}{\eps^3}\int_{\mathcal{D}}\vert\nabla(\widetilde{X}^j(X^j_{\mathrm{CH}})^2)\vert^2dx+\frac{C\tau}{\eps^3}\int_{\mathcal{D}}\vert \nabla\widetilde{X}^j\vert^2dx+\frac{C\tau}{\eps^3}\int_{\mathcal{D}}\vert\nabla((\widetilde{X}^j)^3)\vert^2dx\\
&+\frac{C\tau}{\eps^3}\int_{\mathcal{D}}\vert\nabla((\widetilde{X}^j)^2X^j_{\mathrm{CH}})\vert^2dx=:II_1+II_2+II_3+II_4.\nonumber
\end{align}
Using the bounds $\Vert X^j_{\mathrm{CH}}\Vert_{\mathbb{L}^{\infty}}\leq C$ and $\mathcal{E}(X^j_{\mathrm{CH}})=\displaystyle\frac{\eps}{2}\Vert \nabla X^j_{\mathrm{CH}}\Vert^2+\frac{1}{\eps}\Vert F(X^i_{\mathrm{CH}})\Vert_{\mathbb{L}^1}\leq C$ (see  \lemref{LemmaLubo19} (i) \& (iii)), we estimate
\begin{align}
\label{Vidal2}
II_1&=\frac{C\tau}{\eps^3}\int_{\mathcal{D}}\vert\nabla(\widetilde{X}^j(X^j_{\mathrm{CH}})^2)\vert^2dx\nonumber\\
&\leq \frac{C\tau}{\eps^3}\int_{\mathcal{D}}\vert\nabla\widetilde{X}^j\vert^2\vert X^j_{\mathrm{CH}}\vert^4dx+\frac{C\tau}{\eps^3}\int_{\mathcal{D}}\vert\widetilde{X}^jX^j_{\mathrm{CH}}\nabla X^j_{\mathrm{CH}}\vert^2dx\\
&\leq \frac{C\tau}{\eps^3}\Vert X^j_{\mathrm{CH}}\Vert^4_{\mathbb{L}^{\infty}}\Vert\nabla\widetilde{X}^j\Vert^2+\frac{C\tau}{\eps^3}\Vert \widetilde{X}^j\Vert^2_{\mathbb{L}^{\infty}}\Vert X^j_{\mathrm{CH}}\Vert^2_{\mathbb{L}^{\infty}}\Vert \nabla X^j_{\mathrm{CH}}\Vert^2\nonumber\\
&\leq C\tau\eps^{-3}\Vert\nabla\widetilde{X}^j\Vert^2+C\tau\eps^{-4}\Vert\widetilde{X}^j\Vert^2_{\mathbb{L}^{\infty}}. \nonumber
\end{align}
We easily estimate $II_3$ as follows
\begin{align}
\label{Vidal3}
II_3=\frac{C\tau}{\eps^3}\int_{\mathcal{D}}\vert\nabla[(\widetilde{X}^j)^3]\vert^2dx\leq C\tau\eps^{-3}\int_{\mathcal{D}}\vert\widetilde{X}^j\vert^4\vert\nabla \widetilde{X}^j\vert^2dx\leq C\tau\eps^{-3}\Vert \widetilde{X}^j\Vert^4_{\mathbb{L}^{\infty}}\Vert \nabla\widetilde{X}^j\Vert^2.
\end{align}
Using   again \lemref{LemmaLubo19} (iii) \& (i), we estimate $II_4$ as follows 
\begin{align}
\label{Vidal4}
II_4&=\frac{C\tau}{\eps^3}\int_{\mathcal{D}}\vert\nabla[(\widetilde{X}^j)^2X^j_{\mathrm{CH}}]\vert^2dx\leq \frac{C\tau}{\eps^3}\int_{\mathcal{D}}\vert X^j_{\mathrm{CH}}\widetilde{X}^j\nabla\widetilde{X}^j\vert^2dx+\frac{C\tau}{\eps^3}\int_{\mathcal{D}}\vert   \widetilde{X}^j\vert^4\vert\nabla X^j_{\mathrm{CH}}\vert^2dx\nonumber\\
&\leq C\tau\eps^{-3}\Vert X^j_{\mathrm{CH}}\Vert^2_{\mathbb{L}^{\infty}}\Vert \widetilde{X}^j\Vert^2_{\mathbb{L}^{\infty}}\Vert\nabla\widetilde{X}^j\Vert^2+C\tau\eps^{-3}\Vert \widetilde{X}^j\Vert^4_{\mathbb{L}^{\infty}}\Vert\nabla X^j_{\mathrm{CH}}\Vert^2\\
&\leq C\tau\eps^{-3}\Vert \widetilde{X}^j\Vert^2_{\mathbb{L}^{\infty}}\Vert\nabla\widetilde{X}^j\Vert^2+C\tau\eps^{-4}\Vert \widetilde{X}^j\Vert^4_{\mathbb{L}^{\infty}}.\nonumber
\end{align}
Substituting \eqref{Vidal2}--\eqref{Vidal4} in \eqref{Vidal1} and noting that $II_2\leq C\tau\eps^{-3}\Vert\nabla\widetilde{X}^j\Vert^2$, leads to
\begin{align*}
II=&\frac{C\tau}{\eps^3}\Vert\nabla[f(X^j_{\mathrm{CH}}+\widetilde{X}^j)-f(X^j_{\mathrm{CH}})]\Vert^2\nonumber\\
\leq& C\tau\eps^{-3}\Vert\nabla\widetilde{X}^j\Vert^2+C\tau\eps^{-3}\Vert \widetilde{X}^j\Vert^2+ C\tau\eps^{-3}\Vert \widetilde{X}^j\Vert^4_{\mathbb{L}^{\infty}}\Vert \nabla\widetilde{X}^j\Vert^2+C\tau\eps^{-3}\Vert \widetilde{X}^j\Vert^2_{\mathbb{L}^{\infty}}\Vert\nabla\widetilde{X}^j\Vert^2\nonumber\\
&+C\tau\eps^{-4}\Vert \widetilde{X}^j\Vert^4_{\mathbb{L}^{\infty}}+C\tau\eps^{-4}\Vert \widetilde{X}^j\Vert^2_{\mathbb{L}^{\infty}}.
\end{align*}
Multiplying both sides of the above estimate by $1\!\!1_{\Omega_{\widetilde{W}}}$, using the embedding $\mathbb{L}^{\infty}\hookrightarrow\mathbb{L}^2$ and noting the definition of $\Omega_{\widetilde{W}}$  (i.e., $1\!\!1_{\Omega_{\widetilde{W}}}\Vert \widetilde{X}^j\Vert_{\mathbb{L}^{\infty}}\leq C\eps^{\gamma-\eta-1}$, cf. \eqref{SetW}) yields
\begin{align}
\label{Abdo2}
1\!\!1_{\Omega_{\widetilde{W}}}II\leq C\tau\left(\eps^{2\gamma-2\eta-6}+\eps^{4\gamma-4\eta-8}\right)+C\tau\left(\eps^{-3}+\eps^{2\gamma-2\eta-5}+\eps^{4\gamma-4\eta-7}\right)1\!\!1_{\Omega_{\widetilde{W}}}\Vert\nabla\widetilde{X}^j\Vert^2.
\end{align}
 Recalling that $\widehat{Z}^j=\widehat{X}^j-X^j_{\mathrm{CH}}$, using  \eqref{nonlinearident} and Young's inequality yields
\begin{align*}
I=&\frac{C\tau}{\eps^3}\Vert\nabla(f(\widehat{X}^j+\widetilde{X}^j)-f(X^j_{\mathrm{CH}}+\widetilde{X}^j))\Vert^2\nonumber\\
=&\frac{C\tau}{\eps^3}\int_{\mathcal{D}}\vert \nabla(f(\widehat{X}^j+\widetilde{X}^j)-f(X^j_{\mathrm{CH}}+\widetilde{X}^j))\vert^2dx\nonumber\\
=&\frac{C\tau}{\eps^3}\int_{\mathcal{D}}\left\vert\nabla\left(-3\widehat{Z}^j(X^j_{\mathrm{CH}}+\widetilde{X}^j)^2+\widehat{Z}^j-(\widehat{Z}^j)^3-3(\widehat{Z}^j)^2(X^j_{\mathrm{CH}}+\widetilde{X}^j)\right)\right\vert^2dx\nonumber\\
\leq& \frac{C\tau}{\eps^3}\int_{\mathcal{D}}\left\vert\nabla\left(\widehat{Z}^j(X^j_{\mathrm{CH}}+\widetilde{X}^j)^2\right)\right\vert^2dx+\frac{C\tau}{\eps^3}\int_{\mathcal{D}}\vert\nabla\widehat{Z}^j\vert^2dx+\frac{C\tau}{\eps^3}\int_{\mathcal{D}}\vert\nabla((\widehat{Z}^j)^3)\vert^2dx\nonumber\\
&+ \frac{C\tau}{\eps^3}\int_{\mathcal{D}}\vert\nabla((\widehat{Z}^j)^2(X^j_{\mathrm{CH}}+\widetilde{X}^j))\vert^2dx.
\end{align*}
Consequently
\begin{align}
\label{Recal1}
I\leq& \frac{C\tau}{\eps^3}\Vert X^j_{\mathrm{CH}}\Vert^2_{\mathbb{L}^{\infty}}\Vert\nabla \widehat{Z}^j\Vert^2+\frac{C\tau}{\eps^3}\Vert \widetilde{X}^j\Vert^2_{\mathbb{L}^{\infty}}\Vert \nabla\widehat{Z}^j\Vert^2+\frac{C\tau}{\eps^3}\Vert\nabla\widehat{Z}^j\Vert^2\nonumber\\
&+ \frac{C\tau}{\eps^3}\int_{\mathcal{D}}\vert \widehat{Z}^j\vert^2\vert \widehat{Z}^j\nabla\widehat{Z}^j\vert^2dx+\frac{C\tau}{\eps^3}\int_{\mathcal{D}}\vert\widehat{Z}^j\nabla\widehat{Z}^j (X^j_{\mathrm{CH}}+\widetilde{X}^j)\vert^2dx\\
&+\frac{C\tau}{\eps^3}\int_{\mathcal{D}}\vert \widehat{Z}^j(X^j_{\mathrm{CH}}+\widetilde{X}^j)(\nabla X^j_{\mathrm{CH}}+\nabla\widetilde{X}^j)\vert^2dx+\frac{C\tau}{\eps^3}\int_{\mathcal{D}}\vert (\widehat{Z}^j)^2(\nabla X^j_{\mathrm{CH}}+\nabla\widetilde{X}^j)\vert^2dx\nonumber\\
=:&\frac{C\tau}{\eps^3}\Vert X^j_{\mathrm{CH}}\Vert^2_{\mathbb{L}^{\infty}}\Vert\nabla \widehat{Z}^j\Vert^2+\frac{C\tau}{\eps^3}\Vert \widetilde{X}^j\Vert^2_{\mathbb{L}^{\infty}}\Vert \nabla\widehat{Z}^j\Vert^2+\frac{C\tau}{\eps^3}\Vert\nabla\widehat{Z}^j\Vert^2+I_1+I_2+I_3+I_4.\nonumber
\end{align}
Using triangle inequality, we split $I_3$ as follows
\begin{align}
\label{Abdo1}
I_3=&\frac{C\tau}{\eps^3}\int_{\mathcal{D}}\vert \widehat{Z}^j(X^j_{\mathrm{CH}}+\widetilde{X}^j)(\nabla X^j_{\mathrm{CH}}+\nabla\widetilde{X}^j)\vert^2dx\nonumber\\
\leq& \frac{C\tau}{\eps^3}\int_{\mathcal{D}}\vert \widehat{Z}^j\vert^2\vert X^j_{\mathrm{CH}}\vert^2\vert \nabla X^j_{\mathrm{CH}}\vert^2dx+\frac{C\tau}{\eps^3}\int_{\mathcal{D}}\vert\widehat{Z}^j\vert^2\vert X^j_{\mathrm{CH}}\vert^2\vert\nabla\widetilde{X}^j\vert^2dx\\
&+\frac{C\tau}{\eps^3}\int_{\mathcal{D}}\vert\widehat{Z}^j\vert^2\vert\widetilde{X}^j\vert^2\vert\nabla X^j_{\mathrm{CH}}\vert^2dx+\frac{C\tau}{\eps^3}\int_{\mathcal{D}}\vert\widehat{Z}^j\vert^2\vert \widetilde{X}^j\vert^2\vert \nabla\widetilde{X}^j\vert^2dx\nonumber\\
=:&I_{3,1}+I_{3,2}+I_{3,3}+I_{3,4}.\nonumber
\end{align}
Using the uniform boundedness of $X^j_{\mathrm{CH}}$ (see  \lemref{LemmaLubo19} (iii)), Cauchy-Schwarz's inequality, the  embedding $\mathbb{H}^1\hookrightarrow \mathbb{L}^4$, Poincar\'{e}'s inequality and  \lemref{LemmaLubo19} (ii) yields
\begin{align}
\label{Vend0}
I_{3,1}=&\frac{C\tau}{\eps^3}\int_{\mathcal{D}}\vert \widehat{Z}^j\vert^2\vert X^j_{\mathrm{CH}}\vert^2\vert \nabla X^j_{\mathrm{CH}}\vert^2dx\leq \frac{C\tau}{\eps^3}\Vert X^j_{\mathrm{CH}}\Vert^2_{\mathbb{L}^{\infty}}\int_{\mathcal{D}}\vert\widehat{Z}^j\vert^2\vert \nabla X^j_{\mathrm{CH}}\vert^2dx\nonumber\\
\leq& \frac{C\tau}{\eps^3}\Vert X^j_{\mathrm{CH}}\Vert^2_{\mathbb{L}^{\infty}}\Vert \widehat{Z}^j\Vert^2_{\mathbb{L}^4}\Vert \nabla X^j_{\mathrm{CH}}\Vert^2_{\mathbb{L}^4}\leq \frac{C\tau}{\eps^3} \Vert \nabla\widehat{Z}^j\Vert^2\Vert  X^j_{\mathrm{CH}}\Vert^2_{\mathbb{H}^2}\leq C\tau\eps^{-2\mathfrak{n}_{\mathrm{CH}}-3} \Vert \nabla\widehat{Z}^j\Vert^2.
\end{align}
Using again the uniform boundedness of $X^j_{\mathrm{CH}}$ (see  \lemref{LemmaLubo19} (iii)) yields
\begin{align}
\label{Vend1}
I_{3,2}\leq C\tau\eps^{-3} \Vert X^j_{\mathrm{CH}}\Vert^2_{\mathbb{L}^{\infty}}\Vert\widehat{Z}^j\Vert^2_{\mathbb{L}^{\infty}}\Vert \nabla\widetilde{X}^j\Vert^2\leq C\tau\eps^{-3}\Vert\widehat{Z}^j\Vert^2_{\mathbb{L}^{\infty}}\Vert \nabla\widetilde{X}^j\Vert^2.
\end{align}
Along the same lines as in \eqref{Vend0}  we obtain
\begin{align}
\label{Vend2}
I_{3,3}=\frac{C\tau}{\eps^3}\int_{\mathcal{D}}\vert\widehat{Z}^j\vert^2\vert\widetilde{X}^j\vert^2\vert\nabla X^j_{\mathrm{CH}}\vert^2dx\leq C\tau\eps^{-2n_{\mathrm{CH}}-3}\Vert\widetilde{X}^j\Vert^2_{\mathbb{L}^{\infty}}\Vert \nabla\widehat{Z}^j\Vert^2.
\end{align}
Similarly we estimate
\begin{align}
\label{Vend3}
I_{3,4}=\frac{C\tau}{\eps^3}\int_{\mathcal{D}}\vert\widehat{Z}^j\vert^2\vert\widetilde{X}^j\vert^2\vert\nabla\widetilde{X}^j\vert^2dx\leq C\tau\eps^{-3}\Vert \widehat{Z}^j\Vert^2_{\mathbb{L}^{\infty}}\Vert \widetilde{X}^j\Vert^2_{\mathbb{L}^{\infty}}\Vert\nabla\widetilde{X}^j\Vert^2.
\end{align}
Substituting \eqref{Vend0}--\eqref{Vend3} in \eqref{Abdo1} yields
\begin{align*}
I_3\leq C\tau\eps^{-2\mathfrak{n}_{\mathrm{CH}}-3}(1+\Vert\widetilde{X}^j\Vert^2_{\mathbb{L}^{\infty}})\Vert\nabla\widehat{Z}^j\Vert^2+C\tau\eps^{-3}\Vert\widehat{Z}^j\Vert^2_{\mathbb{L}^{\infty}}(1+\Vert\widetilde{X}^j\Vert^2_{\mathbb{L}^{\infty}})\Vert\nabla\widetilde{X}^j\Vert^2.
\end{align*}
Multiplying both sides of the preceding estimate by $1\!\!1_{\Omega_{\widetilde{W}}\cap\Omega_{\kappa, J}}$, using \eqref{SetXa} and \eqref{SetW}, we get
\begin{align}
\label{Vend4}
1\!\!1_{\Omega_{\widetilde{W}}\cap\Omega_{\kappa, J}}I_3=&1\!\!1_{\Omega_{\widetilde{W}}\cap\Omega_{\kappa, J}}\frac{C\tau}{\eps^3}\int_{\mathcal{D}}\vert \widehat{Z}^j(X^j_{\mathrm{CH}}+\widetilde{W}^j_{\Delta})(\nabla X^j_{\mathrm{CH}}+\nabla\widetilde{X}^j)\vert^2dx\nonumber\\
\leq& C\tau\eps^{-2\mathfrak{n}_{\mathrm{CH}}-3}(1+\eps^{2\gamma-2\eta-2})1\!\!1_{\Omega_{\widetilde{W}}\cap\Omega_{\kappa, J}}\Vert \nabla\widehat{Z}^j\Vert^2\\
&+C\tau\eps^{-3}(\kappa^2+1) (1+\eps^{2\gamma-2\eta-2})1\!\!1_{\Omega_{\widetilde{W}}\cap\Omega_{\kappa, J}}\Vert \nabla\widetilde{X}^j\Vert^2, \nonumber
\end{align}
where we  used the fact that $\widehat{Z}^j=\widehat{X}^j-X^j_{\mathrm{CH}}$ and  \lemref{LemmaLubo19} (iii). 
Similarly, we obtain
\begin{align}
\label{Vend5}
1\!\!1_{\Omega_{\widetilde{W}}\cap\Omega_{\kappa, J}}I_1&=1\!\!1_{\Omega_{\widetilde{W}}\cap\Omega_{\kappa, J}}\frac{C\tau}{\eps^3}\int_{\mathcal{D}}\vert \widehat{Z}^j\vert^2\vert \widehat{Z}^j\nabla\widehat{Z}^j\vert^2dx\leq C\tau\eps^{-3}1\!\!1_{\Omega_{\widetilde{W}}\cap\Omega_{\kappa, J}}\Vert \widehat{Z}^j\Vert^2_{\mathbb{L}^{\infty}}\Vert\widehat{Z}^j\nabla\widehat{Z}^j\Vert^2\nonumber\\
&\leq C\tau\eps^{-3}(1+\kappa^2)1\!\!1_{\Omega_{\widetilde{W}}\cap\Omega_{\kappa, J}}\Vert \widehat{Z}^j\nabla\widehat{Z}^j\Vert^2. 
\end{align}
Noting the definitions of $\Omega_{\widetilde{W}}$ \eqref{SetW} and $\Omega_{\kappa, J}$ \eqref{SetXa},  using the uniform boundedness of $X^j_{\mathrm{CH}}$ (see  \lemref{LemmaLubo19} (iii)), it follows that
\begin{align}
\label{Vend6}
1\!\!1_{\Omega_{\widetilde{W}}\cap\Omega_{\kappa, J}}I_2&=1\!\!1_{\Omega_{\widetilde{W}}\cap\Omega_{\kappa, J}}\frac{C\tau}{\eps^3}\int_{\mathcal{D}}\vert\widehat{Z}^j\nabla\widehat{Z}^j(X^j_{\mathrm{CH}}+\widetilde{X}^j)\vert^2dx\nonumber\\
&\leq 1\!\!1_{\Omega_{\widetilde{W}}\cap\Omega_{\kappa, J}}C\tau\eps^{-3}\left(\Vert X^j_{\mathrm{CH}}\Vert^2_{\mathbb{L}^{\infty}}+\Vert\widetilde{X}^j\Vert^2_{\mathbb{L}^{\infty}}\right)\Vert\widehat{Z}^j\nabla\widehat{Z}^j\Vert^2\\
&\leq C\tau\eps^{-3}(1+\eps^{2\gamma-2\eta-2})1\!\!1_{\Omega_{\widetilde{W}}\cap\Omega_{\kappa, J}}\Vert \widehat{Z}^j\nabla\widehat{Z}^j\Vert^2. \nonumber
\end{align}
Arguing as in \eqref{Vend0}, using the embedding $\mathbb{H}^1\hookrightarrow\mathbb{L}^4$ and Poincar\'{e}'s inequality we deduce that
\begin{align}
\label{Vend7}
1\!\!1_{\Omega_{\widetilde{W}}\cap\Omega_{\kappa, J}}I_4&=1\!\!1_{\Omega_{\widetilde{W}}\cap\Omega_{\kappa, J}}\frac{C\tau}{\eps^3}\int_{\mathcal{D}}\vert(\widehat{Z}^j)^2(\nabla X^j_{\mathrm{CH}}+\nabla\widetilde{X}^j)\vert^2dx\nonumber\\
&\revl{\leq} 1\!\!1_{\Omega_{\widetilde{W}}\cap\Omega_{\kappa, J}}\frac{C\tau}{\eps^3}\int_{\mathcal{D}}\vert\widehat{Z}^j\vert^4(\vert \nabla X^j_{\mathrm{CH}}\vert+\vert \nabla\widetilde{X}^j\vert)^2dx\nonumber\\
&\leq C1\!\!1_{\Omega_{\widetilde{W}}\cap\Omega_{\kappa, J}}\frac{C\tau}{\eps^3}\Vert  X^j_{\mathrm{CH}}\Vert^2_{\mathbb{H}^2}\Vert\widehat{Z}^j\Vert^2_{\mathbb{L}^{\infty}}\Vert \nabla\widehat{Z}^j\Vert^2+1\!\!1_{\Omega_{\widetilde{W}}\cap\Omega_{\kappa, J}}\frac{C\tau}{\eps^3}\Vert \widehat{Z}^j\Vert^4_{\mathbb{L}^{\infty}}\Vert \nabla\widetilde{X}^j\Vert^2\\
&\leq C\tau\kappa^2\eps^{-2\mathfrak{n}_{\mathrm{CH}}-3}1\!\!1_{\Omega_{\widetilde{W}}\cap\Omega_{\kappa, J}}\Vert \nabla\widehat{Z}^j\Vert^2+C\tau\eps^{-3}\kappa^41\!\!1_{\Omega_{\widetilde{W}}\cap\Omega_{\kappa, J}}\Vert\nabla\widetilde{X}^j\Vert^2. \nonumber
\end{align}
Substituting \eqref{Vend4}--\eqref{Vend7} in \eqref{Recal1},   using the fact  $0<\eps<1$, noting that from \assref{assumption3} we have $\gamma-\eta-1\geq 0$, $2\gamma-2\eta-3\geq 0$, $4\gamma-4\eta-5\geq 0$ (which implies  $\eps^{\gamma-\eta-1}+\eps^{2\gamma-2\eta-3}+\eps^{4\gamma-4\eta-5}\leq 1$), we obtain
\begin{align}
\label{Vend8}
1\!\!1_{\Omega_{\widetilde{W}}\cap\Omega_{\kappa, J}}I=&1\!\!1_{\Omega_{\widetilde{W}}\cap\Omega_{\kappa, J}}\frac{C\tau}{\eps^3}\Vert\nabla(f(\widehat{X}^j+\widetilde{X}^j)-f(X^j_{\mathrm{CH}}+\widetilde{X}^j))\Vert^2\nonumber\\
\leq& C\tau\eps^{-2\mathfrak{n}_{\mathrm{CH}}-3}(1+\kappa^2)1\!\!1_{\Omega_{\widetilde{W}}\cap\Omega_{\kappa, J}}\Vert \nabla\widehat{Z}^j\Vert^2\\
&+C\tau\eps^{-3}(1+\kappa^4)1\!\!1_{\Omega_{\widetilde{W}}\cap\Omega_{\kappa, J}}\Vert \nabla\widetilde{X}^j\Vert^2+C\tau\eps^{-3}(1+\kappa^2)1\!\!1_{\Omega_{\widetilde{W}}\cap\Omega_{\kappa, J}}\Vert \widehat{Z}^j\nabla\widehat{Z}^j\Vert^2.\nonumber
\end{align}
Adding \eqref{Vend8} and \eqref{Abdo2}, summing the resulting estimate over $j=1,\cdots, J$ and using the fact that $\gamma-\eta-1\geq 0$ (see \assref{assumption3})yelds
\begin{align}
\label{Sam5}
&\frac{\tau}{\eps^3}\sum_{j=1}^J1\!\!1_{\Omega_{\widetilde{W}}\cap\Omega_{\kappa, J}}\Vert \nabla(f(\widehat{X}^j+\widetilde{X}^j)-f(X^j_{\mathrm{CH}}+\widetilde{X}^j))\Vert^2\nonumber\\
&\quad+\frac{\tau}{\eps^3}\sum_{j=1}^J1\!\!1_{\Omega_{\widetilde{W}}\cap\Omega_{\kappa, J}}\Vert\nabla(f(X^j_{\mathrm{CH}}+\widetilde{X}^j)-f(X^j_{\mathrm{CH}}))\Vert^2\nonumber\\
&\leq \frac{C(1+\kappa^2)\tau}{\eps^3}\eps^{-2\mathfrak{n}_{\mathrm{CH}}}\sum_{j=1}^J1\!\!1_{\Omega_{\widetilde{W}}\cap\Omega_{\kappa, J}}\Vert\nabla\widehat{Z}^j\Vert^2\\
&\quad+\frac{C(1+\kappa^4+\eps^{-2\mathfrak{n}_{\mathrm{CH}}})\tau}{\eps^3}\sum_{j=1}^J1\!\!1_{\Omega_{\widetilde{W}}\cap\Omega_{\kappa, J}}\Vert\nabla\widetilde{X}^j\Vert^2\nonumber\\
&\quad+\frac{C(1+\kappa^2)\tau}{\eps^3}\sum_{j=1}^J 1\!\!1_{\Omega_{\widetilde{W}}\cap\Omega_{\kappa, J}}\Vert \widehat{Z}^j\nabla\widehat{Z}^j\Vert^2.\nonumber
\end{align}
Summing \eqref{Sam6} over $j=1,\cdots, J$, multiplying by $1\!\!1_{\Omega_{\widetilde{W}}\cap\Omega_{\kappa, J}}$, taking the expectation,  using \eqref{Sam5}, Poincar\'{e}'s inequality yields
\begin{align}
\label{Sam7a} 
&\frac{\tau}{\eps}\sum_{j=1}^J\mathbb{E}\left[1\!\!1_{\Omega_{\widetilde{W}}\cap\Omega_{\kappa, J}}\left(\nabla[f(\widehat{X}^j+\widetilde{X}^j)-f(X^j_{\mathrm{CH}})], \nabla\Delta\widehat{Z}^j\right)\right]\nonumber\\
&\leq \frac{\eps\tau}{4}\sum_{j=1}^J\mathbb{E}\left[1\!\!1_{\Omega_{\widetilde{W}}\cap\Omega_{\kappa, J}}\Vert\nabla\Delta\widehat{Z}^j\Vert^2\right]+\frac{C(1+\kappa^2)\tau}{\eps^3}\eps^{-2\mathfrak{n}_{\mathrm{CH}}}\sum_{j=1}^J\mathbb{E}\left[1\!\!1_{\Omega_{\widetilde{W}}\cap\Omega_{\kappa, J}}\Vert\nabla\widehat{Z}^j\Vert^2\right]\\
&\quad+\frac{C(1+\kappa^4+\eps^{-2\mathfrak{n}_{\mathrm{CH}}})\tau}{\eps^3}\sum_{j=1}^J\mathbb{E}\left[1\!\!1_{\Omega_{\widetilde{W}}\cap\Omega_{\kappa, J}}\Vert\nabla\widetilde{X}^j\Vert^2\right]\nonumber\\
&\quad+\frac{C(1+\kappa^2)\tau}{\eps^3}\sum_{j=1}^J \mathbb{E}\left[1\!\!1_{\Omega_{\widetilde{W}}\cap\Omega_{\kappa, J}}\Vert \widehat{Z}^j\nabla\widehat{Z}^j\Vert^2\right].\nonumber
\end{align}
Using  Lemmas \ref{EstiZ}, \ref{LemmaNorm2} and  \ref{regularitynoise}  with $\alpha=1$ and recalling that $h=\eps^{\eta}$ yields
\begin{align}
\label{Sam7}
&\frac{\tau}{\eps}\sum_{j=1}^J\mathbb{E}\left[1\!\!1_{\Omega_{\widetilde{W}}\cap\Omega_{\kappa, J}}\left(\nabla[f(\widehat{X}^j+\widetilde{X}^j)-f(X^j_{\mathrm{CH}})], \nabla\Delta\widehat{Z}^j\right)\right]\nonumber\\
&\leq  \frac{\eps\tau}{4}\sum_{j=1}^J\mathbb{E}\left[1\!\!1_{\Omega_{\widetilde{W}}\cap\Omega_{\kappa, J}}\Vert\nabla\Delta\widehat{Z}^j\Vert^2\right]+\frac{C(1+\kappa^4+\eps^{-2n_{\mathrm{CH}}})}{\eps^3}\eps^{2\gamma-\frac{1}{2}-\frac{d\eta}{2}}\\
&\quad+\frac{C(1+\kappa^2)}{\eps^4}\eps^{-2\mathfrak{n}_{\mathrm{CH}}}\mathcal{F}_1(\tau, s, \eps, \sigma_0, \kappa_0, \gamma, \eta)+\frac{C(1+\kappa^2)}{\eps^2}\mathcal{F}_1(\tau, d, \eps, \sigma_0, \kappa_0, \gamma, \eta).\nonumber
\end{align}
Hence, summing \eqref{scheme6} over $j$, multiplying by $1\!\!1_{\Omega_{\widetilde{W}}\cap\Omega_{\kappa, J}}$, taking the maximum,  the expectation, using \eqref{Sam7} and absorbing the term
 $ \mathbb{E}\left[1\!\!1_{\Omega_{\widetilde{W}}\cap\Omega_{\kappa, J}}\Vert\nabla\Delta\widehat{Z}^j\Vert^2\right]$ in the left-hand side, yields that
\begin{align*}
&\mathbb{E}\left[\max_{1\leq j\leq J}1\!\!1_{\Omega_{\widetilde{W}}\cap\Omega_{\kappa, J}}\Vert\nabla\widehat{Z}^j\Vert^2\right]+\sum_{j=1}^J\mathbb{E}\left[1\!\!1_{\Omega_{\widetilde{W}}\cap\Omega_{\kappa, J}}\Vert\widehat{Z}^j-\widehat{Z}^{j-1}\Vert^2+\eps\tau 1\!\!1_{\Omega_{\widetilde{W}}\cap\Omega_{\kappa, J}}\Vert\nabla\Delta\widehat{Z}^j\Vert^2\right]\nonumber\\
&\leq \frac{C(1+\kappa^4+\eps^{-2\mathfrak{n}_{\mathrm{CH}}})}{\eps^3}\eps^{2\gamma-\frac{1}{2}-\frac{d\eta}{2}}+C\left\{\frac{(1+\kappa^2)}{\eps^4}\eps^{-2\mathfrak{n}_{\mathrm{CH}}}+\frac{C(1+\kappa^2)}{\eps^2}\right\}\mathcal{F}_1(\tau, d, \eps, \sigma_0, \kappa_0, \gamma, \eta).
\end{align*}
\end{proof}

{To ensure that the the right-hand side in the estimate in \lemref{EstinablaZ} vanishes for $\eps\rightarrow 0$ we require the following assumption.}
\begin{Assumption}
\label{assumption4}
Let \assref{assumption3} hold and assume in addition that
  $\sigma_0, \kappa_0, \gamma, \eta$ and $\tau$ are such that
\begin{align}
\label{F2}
\lim_{\eps\rightarrow 0}\mathcal{F}_2(\tau, d, \eps; \sigma_0, \kappa_0, \gamma, \eta)=0,
\end{align}
where $\mathcal{F}_2(\tau,d, \eps; \sigma_0, \kappa_0, \gamma, \eta)$ is defined in \lemref{EstinablaZ}.
\end{Assumption}
\begin{remark}
\label{Scenario1}
\revd{A strategy to identify  admissible quadruples $(\sigma_0, \kappa_0, \gamma, \tau)$ which meet \assref{assumption4} is as follows: 
\begin{itemize}
\item[(1)] \assref{assumption3} establishes $\lim\limits_{\eps\rightarrow 0}\mathcal{F}_1(\tau, d, \eps; \sigma_0, \kappa_0, \gamma, \eta)=0$, which appears as a factor in the second term on the right-hand side in \lemref{EstinablaZ}.
\item[(2)] The leading factor in $\mathcal{F}_2$ (cf. in \lemref{EstinablaZ}) is 
\begin{align*}
\eps^{-2\mathfrak{n}_{\mathrm{CH}}}\frac{\kappa^2}{\eps^4}\leq \eps^{-4-2\mathfrak{n}_{\mathrm{CH}}-2\theta-8}.
\end{align*}
To meet \eqref{F2}, we therefore require that
\begin{align}
\label{F2a}
\eps^{-4-2\mathfrak{n}_{\mathrm{CH}}-2\theta-8}\mathcal{F}_1(\tau,d, \eps; \sigma_0, \kappa_0, \gamma, \eta)\rightarrow 0\quad \text{as}\; \eps\rightarrow 0.
\end{align}
This can be achieved by taking  $\tau=\eps^{\varrho}$, with $\varrho>0$ sufficiently large.  This implies that only larger values of $\gamma$ and $\sigma_0$ are admissible.
\item[(3)] We  proceed analogously for the first term on the right hand-side in \lemref{EstinablaZ}. 
\end{itemize}
}
\end{remark}

\revl{Using the result of \lemref{EstinablaZ} we deduce a $\mathbb{L}^\infty$-estimate for the corresponding error.}
\begin{lemma}
\label{maintheorem2a}
\revd{
Let  \assref{assumption4}  hold \revd{and let $d<p<q\leq 6$}.    Then it holds  that
\begin{align*}
\mathbb{E}\left[\max_{1\leq j\leq J}1\!\!1_{\Omega_{\widetilde{W}}\cap\Omega_{\kappa,J}}\Vert\widehat{Z}^j\Vert^{\frac{p}{q}}_{\mathbb{L}^{\infty}}\right]\leq C\eps^{\frac{(2-p)}{2(q-2)}\left(2\theta+7+2\mathfrak{n}_{\mathrm{CH}}\right)}\left(\mathcal{F}_2(\tau, d,\eps; \sigma_0, \kappa_0, \gamma, \eta)\right)^{\frac{q-p}{q(q-2)}},
\end{align*}
}
where $\eta$ is as in \assref{assumption2}.
\end{lemma}
\begin{proof} 
\revd{
 Using the Sobolev embedding $\mathbb{W}^{1,p}\hookrightarrow \mathbb{L}^{\infty}$ $(p>d)$ (cf. \cite[Corollary 9.14]{Brezis2010functional}), and the interpolation inequality (cf. \cite[Proposition 6.10]{Folland})
\begin{align*}
\Vert u\Vert_{\mathbb{L}^{q'}}\leq \Vert u\Vert^{\lambda}_{\mathbb{L}^{p'}}\Vert u\Vert^{1-\lambda}_{\mathbb{L}^{r'}},\quad p'<q'<r',\quad \lambda=\frac{p'}{q'}\frac{r'-q'}{r'-p'}, \quad u\in\mathbb{L}^{r'},
\end{align*}
we obtain by using Poincar\'{e}'s inequality
\begin{align*}
\Vert X^j\Vert_{\mathbb{L}^{\infty}}\leq C\Vert \nabla X^j\Vert_{\mathbb{L}^p}\leq C\Vert \nabla X^j\Vert_{\mathbb{L}^2}^{\frac{2(q-p)}{p(q-2)}}\Vert\nabla X^j\Vert^{\frac{q(p-2)}{p(q-2)}}_{\mathbb{L}^q}\quad d<p<q.
\end{align*}
Using the Sobolev embedding $\mathbb{H}^1\hookrightarrow\mathbb{L}^q$ $(q\leq 6)$ and  the elliptic regularity, we obtain
\begin{align}
\label{Gagliardo1a}
\Vert X^j\Vert^{\frac{p}{q}}_{\mathbb{L}^{\infty}}\leq C\Vert \nabla X^j\Vert_{\mathbb{L}^2}^{\frac{2(q-p)}{q(q-2)}}\Vert\Delta X^j\Vert^{\frac{p-d}{q-2}}\quad d<p<q\leq 6.
\end{align}

Using the inequality \eqref{Gagliardo1a},   H\"{o}lder's inequality with exponents $\frac{q-2}{q-p}$ and $\frac{q-2}{p-2}$, Cauchy-Schwarz's inequality, noting that $\Omega_{\kappa,J}=\Omega_{\infty}\cap\Omega_{\mathcal{E}}$, using Lemmas \ref{EstiXhat} and \ref{LemmaLubo19} iii), we obtain
\begin{align*}
&\mathbb{E}\left[\max_{1\leq j\leq J}1\!\!1_{\Omega_{\widetilde{W}}\cap\Omega_{\kappa,J}}\Vert \widehat{Z}^j\Vert^{\frac{p}{q}}_{\mathbb{L}^{\infty}}\right]\nonumber\\
&\quad\leq C\mathbb{E}\left[\max_{1\leq j\leq J}1\!\!1_{\Omega_{\widetilde{W}}\cap\Omega_{\kappa,J}}\Vert \nabla \widehat{Z}^j\Vert_{\mathbb{L}^2}^{\frac{2(q-p)}{q(q-2)}}\Vert\Delta \widehat{Z}^j\Vert^{\frac{p-2}{q-2}}\right]\nonumber\\
&\quad\leq C\mathbb{E}\left[\max_{1\leq j\leq J}1\!\!1_{\Omega_{\widetilde{W}}\cap\Omega_{\kappa,J}}\Vert \nabla \widehat{Z}^j\Vert_{\mathbb{L}^2}^{\frac{2}{q}}\right]^{\frac{q-p}{q-2}}\mathbb{E}\left[\max_{1\leq j\leq J}1\!\!1_{\Omega_{\widetilde{W}}\cap\Omega_{\kappa, J}}\Vert\Delta \widehat{Z}^j\Vert\right]^{
\frac{p-2}{q-2}}\nonumber\\
&\quad\leq C\eps^{\frac{(2-p)}{2(q-2)}\left(2\theta+7+2\mathfrak{n}_{\mathrm{CH}}\right)}\mathbb{E}\left[\max_{1\leq j\leq J}1\!\!1_{\Omega_{\widetilde{W}}\cap\Omega_{\kappa, J}}\Vert \nabla \widehat{Z}^j\Vert_{\mathbb{L}^2}^{\frac{2}{q}}\right]^{\frac{q-p}{q-2}}.
\end{align*} 
Using  H\"{o}lder's inequality with exponents $q$ and $\frac{q}{q-1}$ and \lemref{EstinablaZ},  we obtain 
\begin{align*}
\mathbb{E}\left[\max_{1\leq j\leq J}1\!\!1_{\Omega_{\widetilde{W}}\cap\Omega_{\kappa,J}}\Vert \widehat{Z}^j\Vert^{\frac{p}{q}}_{\mathbb{L}^{\infty}}\right]&\leq C\eps^{\frac{(2-p)}{2(q-2)}\left(2\theta+7+2\mathfrak{n}_{\mathrm{CH}}\right)}\mathbb{E}\left[\max_{1\leq j\leq J}1\!\!1_{\Omega_{\widetilde{W}}\cap\Omega_{\infty}}\Vert \nabla \widehat{Z}^j\Vert_{\mathbb{L}^2}^{2}\right]^{\frac{q-p}{q(q-2)}}\nonumber\\
&\leq C\eps^{\frac{(2-p)}{2(q-2)}\left(2\theta+7+2\mathfrak{n}_{\mathrm{CH}}\right)}\left(\mathcal{F}_2(\tau, d,\eps; \sigma_0, \kappa_0, \gamma, \eta)\right)^{\frac{q-p}{q(q-2)}}.
\end{align*}
}
\end{proof}

\revl{To establish convergence of the numerical scheme to the sharp-interface limit  \eqref{HeleShaw1} we require that the right-hand side in the above $\mathbb{L}^\infty$-estimate
vanishes for $\eps\rightarrow 0$.
To this end, we impose the following assumption, which is stronger than \assref{assumption4}.}
\begin{Assumption}
\label{assumption5}
Let \assref{assumption4} be fulfilled. Let $\sigma_0, \kappa_0, \gamma, \eta$ and $\tau$ be such that 
\revd{
\begin{align}
\label{F4}
\lim_{\eps\rightarrow 0}\left[\eps^{\frac{(2-p)}{2(q-2)}\left(2\theta+7+2\mathfrak{n}_{\mathrm{CH}}\right)}\left(\mathcal{F}_2(\tau, d,\eps; \sigma_0, \kappa_0, \gamma, \eta)\right)^{\frac{q-p}{q(q-2)}}\right]=0.
\end{align}
}
\end{Assumption}

\begin{remark}
\revl{To identify admissible $(\sigma_0, \kappa_0, \gamma, \tau, \eta)$ which meet \eqref{F4}, it is enough to limit ourselves to a discussion of the leading term inside the maximum which defines $\mathcal{F}_2$. 
 To meet \eqref{F4}, we have to ensure that for some $d<p<q\leq 6$
\begin{align*}
\eps^{\frac{(2-p)}{2(q-2)}\left(2\theta+7+2\mathfrak{n}_{\mathrm{CH}}\right)}\left[\eps^{-4-2\mathfrak{n}_{\mathrm{CH}}-2\theta-8}\mathcal{F}_1(\tau, d,\eps; \sigma_0, \kappa_0, \gamma, \eta)\right]^{\frac{q-p}{q(q-2)}}\rightarrow 0\quad (\eps\rightarrow 0).
\end{align*}
This can be achieved by taking  $\tau=\eps^{\varrho}$, with $\varrho>0$ sufficiently large.  This implies that only larger values of $\gamma$ and $\sigma_0$ are admissible.
}
\end{remark}

\revl{

\revl{The theorem below, which provides the estimate of the error $Z^j=X^j-X^j_{\mathrm{CH}}$ in the $\mathbb{L}^\infty$-norm,
is a crucial ingredient in the convergence proof of the sharp-interface limit and is a straightforward consequence of \lemref{maintheorem2a}.
}
\begin{theorem}
\label{LinfinityZ}
Let \assref{assumption5} be fulfilled {and let $d<p<q\leq 6$}. Then
\begin{align*}
\mathbb{E}\left[\max_{1\leq j\leq J}1\!\!1_{\Omega_{\widetilde{W}}\cap\Omega_{\kappa, J}}\Vert Z^j\Vert^{\frac{p}{q}}_{\mathbb{L}^{\infty}}\right]\rightarrow 0 \quad (\text{as}\; \varepsilon\rightarrow 0).
\end{align*}
\end{theorem}
\begin{proof}
 Noting $Z^j=X^j-X^j_{\mathrm{CH}}-\widetilde{X}^j+\widetilde{X}^j=\widehat{Z}^j+\widetilde{X}^j$ and using triangle inequality, we obtain
\begin{align*}
\max_{1\leq j\leq J}1\!\!1_{\Omega_{\widetilde{W}}\cap\Omega_{\kappa, J}}\Vert Z^j\Vert^{\frac{p}{q}}_{\mathbb{L}^{\infty}}\leq C\max_{1\leq j\leq J}1\!\!1_{\Omega_{\widetilde{W}}\cap\Omega_{\kappa, J}}\Vert \widehat{Z}^j\Vert^{\frac{p}{q}}_{\mathbb{L}^{\infty}}+C\max_{1\leq j\leq J}1\!\!1_{\Omega_{\widetilde{W}}\cap\Omega_{\kappa, J}}\Vert \widetilde{X}^j\Vert^{\frac{p}{q}}_{\mathbb{L}^{\infty}}.
\end{align*}
We take expectation in the above inequality and use \lemref{maintheorem2a}, \assref{assumption5} and \eqref{LinfinityXtilde} to estimate the right-hand side to conclude the proof.
\end{proof}
}

\subsection{Convergence to the sharp-interface limit}
For each $\eps\in (0, \eps_0)$ we consider the piecewise affine time-interpolation of the solution $\{X^j\}_{j=0}^J$ of the numerical scheme  \eqref{scheme1b} as
\begin{align}
\label{scheme1f}
X^{\eps, \tau}(t):=\frac{t-t_{j-1}}{\tau}X^j+\frac{t_j-t}{\tau}X^{j-1}\quad \text{for } t_{j-1}\leq t\leq t_j.
\end{align}
Let $\Gamma_{00}\subset\mathcal{D}$ be a smooth closed curve \revl{ if $d=2$ or a smooth closed surface if $d=3$}, and $(v_{\mathrm{MS}}, \Gamma^{\mathrm{MS}})$ be a smooth solution of \eqref{HeleShaw1} starting from $\Gamma_{00}$, where $\Gamma^{\mathrm{MS}}:=\cup_{0\leq t\leq T}\{t\}\times\Gamma^{\mathrm{MS}}_t$. Let $d(t,x)$ be the signed distance function to $\Gamma_t^{\mathrm{MS}}$ such that $d(t,x)<0$ in $\mathcal{I}_t^{\mathrm{MS}}$ (the inside of $\Gamma_t^{\mathrm{MS}}$) and $d(t,x)>0$ on $\mathcal{O}_t^{\mathrm{MS}}:=\mathcal{D}\setminus(\Gamma_t^{\mathrm{MS}}\cap\mathcal{I}_t^{\mathrm{MS}})$, the outside of $\Gamma_t^{\mathrm{MS}}$. In other words, the inside $\mathcal{I}_t^{\mathrm{MS}}$ and outside $\mathcal{O}_t^{\mathrm{MS}}$ of  $\Gamma_t^{\mathrm{MS}}$ are defined as
\begin{align*}
\mathcal{I}_t^{\mathrm{MS}}:=\left\{(t,x)\in\overline{\mathcal{D}}_T:\; d(t,x)<0\right\},\quad \mathcal{O}_t^{\mathrm{MS}}:=\left\{(t,x)\in\overline{\mathcal{D}}_T:\; d(t,x)>0\right\}.
\end{align*}
For the numerical interpolant $X^{\eps, \tau}$ we denote the zero level set at time $t$ by $\Gamma_t^{\eps,\tau}$, that is,
\begin{align*}
\Gamma_t^{\eps, \tau}:=\left\{x\in\mathcal{D}:\; X^{\eps,\tau}(t,x)=0\},\quad 0\leq t\leq T\right\}. 
\end{align*}

In order to establish the convergence in probability of iterates $\{X^j\}_{j=0}^J$ in the sets $\mathcal{I}^{\mathrm{MS}}$ and $\mathcal{O}^{\mathrm{MS}}$, we further need the following requirement. 

\begin{Assumption}
\label{assumption6}
Let $\mathcal{D}\subset\mathbb{R}^d$ be a smooth domain. There exists a classical solution $(v_{\mathrm{MS}}, \Gamma^{\mathrm{MS}})$ of \eqref{HeleShaw1} evolving from $\Gamma_{00}\subset\mathcal{D}$, such that $\Gamma_t^{\mathrm{MS}}\subset\mathcal{D}$ for all $t\in[0, T]$. 
\end{Assumption}

Under \assref{assumption6}, it was proved in \cite[Theorem 5.1]{abc94} that there exits a family of smooth functions $\{u^{\eps}_0\}_{0\leq \eps\leq 1}$, which are uniformly bounded in $\eps$ and $(t, x)$ and such that if $u^{\eps}_{\mathrm{CH}}$ is the solution to the deterministic Cahn-Hilliard equation (i.e., Eq. \eqref{model1} with $W\equiv 0$) with initial value $u^{\eps}_0$. Then
\begin{itemize}
\item[(i)] $\lim\limits_{\eps\rightarrow 0} u^{\eps}_{\mathrm{CH}}(t,x)=\left\{\begin{array}{ll}
+1,\quad \text{if}\; (t,x)\in \mathcal{O}^{\mathrm{MS}},\\
-1\quad \text{if}\; (t,x)\in \mathcal{I}^{\mathrm{MS}},
\end{array}
\text{uniformly on compacts subsets of $\mathcal{D}_T$},
\right.$
\item[(ii)] $\lim\limits_{\eps\rightarrow 0}\left(\frac{1}{\eps}f(u^{\eps}_{\mathrm{CH}})-\eps\Delta u^{\eps}_{\mathrm{CH}}\right)(t,x)=v^{\mathrm{MS}}(t,x)$ uniformly on $\mathcal{D}_T$.
\end{itemize}

The theorem below establishes uniform convergence of the numerical approximation \eqref{scheme1b}
in probability on the space-time sets $\mathcal{I}^{\mathrm{MS}}$, $\mathcal{O}^{\mathrm{MS}}$.
\begin{theorem}
\label{maintheorem2}
Let Assumptions \ref{assumption5} and \ref{assumption6} be fulfilled.
\revl{Let $\{X^{\eps, \tau}\}_{0\leq \eps\leq\eps_0}$ in \eqref{scheme1f} be obtained from the solutions of \eqref{scheme1b}}. Then it hold that
\begin{itemize}
\item[(i)] $\lim\limits_{\eps\rightarrow 0}\mathbb{P}\left[\{\Vert X^{\eps, \tau}-1\Vert_{C(\mathcal{A})}>\alpha\quad \text{for all}\; \mathcal{A}\Subset\mathcal{O}^{\mathrm{MS}}\}\right]=0$ for all $\alpha>0$,
\item[(ii)] $\lim\limits_{\eps\rightarrow 0}\mathbb{P}\left[\{\Vert X^{\eps, \tau}+1\Vert_{C(\mathcal{A})}>\alpha\quad \text{for all}\; \mathcal{A}\Subset\mathcal{I}^{\mathrm{MS}}\}\right]=0$ for all $\alpha>0$, 
\end{itemize}
\revl{ where $C(\mathcal{A})$ is the space of continuous functions on $\mathcal{A}$.}
\end{theorem}

\begin{proof}
We decompose our error as follows: $X^{\eps, \tau}\pm 1=(X^{\eps, \tau}-X^{\eps, \tau}_{\mathrm{CH}})+(X^{\eps, \tau}_{\mathrm{CH}}\pm 1)$, where $X^{\eps, \tau}_{\mathrm{CH}}$ is the piecewise  affine interpolant of $\{X^j_{\mathrm{CH}}\}_{j=0}^J$. We also write $\overline{\mathcal{D}}_T\setminus\Gamma=\mathcal{I}^{\mathrm{MS}}\cup\mathcal{O}^{\mathrm{MS}}$. From \cite[Theorem 4.2]{fp04}, the piecewise affine interpolant satisfies
\begin{itemize}
\item[(i')] $X^{\eps, \tau}_{\mathrm{CH}}\longrightarrow +1$ uniformly on compact subsets of $\mathcal{O}^{\mathrm{MS}}$ (as $\eps\rightarrow 0$),
\item[(ii')] $X^{\eps, \tau}_{\mathrm{CH}}\longrightarrow -1$ uniformly on compact subsets of $\mathcal{I}^{\mathrm{MS}}$ (as $\eps\rightarrow 0$). 
\end{itemize}
Since $\lim\limits_{\eps\rightarrow 0}\mathbb{P}[\Omega_{\kappa, J}\cap\Omega_{\widetilde{W}}]=1$, it  holds that $\lim\limits_{\eps\rightarrow 0}\mathbb{P}[(\Omega_{\kappa, J}\cap\Omega_{\widetilde{W}})^c]=0$. Using Chebyshev's inequality (see \cite[Theorem 3.14]{WalshBook2012}) and \revl{ \thmref{LinfinityZ}}, it follows \revd{for $d<p<q\leq 6$} that
\begin{align}
\label{Yan8} 
\mathbb{P}\left[\max_{1\leq j\leq J}\Vert Z^j\Vert_{\mathbb{L}^{\infty}}>\alpha\right]&\leq \mathbb{P}\left[\{\max_{1\leq j\leq J}\Vert Z^j\Vert_{\mathbb{L}^{\infty}}>\alpha\}\cap\Omega_{\kappa, J}\cap\Omega_{\widetilde{W}}\right]+\mathbb{P}\left[(\Omega_{\kappa, J}\cap\Omega_{\widetilde{W}})^c\right]\nonumber\\
&\leq \frac{1}{\alpha^{\frac{p}{q}}}\mathbb{E}\left[\max_{1\leq j\leq J}1\!\!1_{\Omega_{\widetilde{W}}\cap\Omega_{\kappa, J}}\Vert Z^j\Vert^{\frac{p}{q}}_{\mathbb{L}^{\infty}}\right]+ \mathbb{P}\left[(\Omega_{\kappa, J}\cap\Omega_{\widetilde{W}})^c\right]\\
&\longrightarrow 0 \; (\text{as}\; \eps\rightarrow 0).\nonumber
\end{align}
Using \eqref{Yan8} together with (i') and (ii') completes the proof of the theorem. 
\end{proof}

The following corollary gives the convergence in probability (for $\eps\rightarrow 0$) of the zero level set $\{\Gamma_t^{\eps, \tau}; \; t\geq 0\}$ to the interface $\Gamma_t^{\mathrm{MS}}$ of the Hele-Shaw/Mullins-Sekerka problem \eqref{HeleShaw1}. 

\begin{corollary}
Let the assumptions in \thmref{maintheorem2} be fulfilled
and \revl{let $\{X^{\eps, \tau}\}_{0\leq \eps\leq\eps_0}$ in \eqref{scheme1f} be obtained from the solutions of \eqref{scheme1b}}.
Then it holds that
\begin{align*}
\lim_{\eps\rightarrow 0}\mathbb{P}\left[\left\{\sup_{(t,x)\in[0, T]\times\Gamma_t^{\eps,\tau}}\mathrm{dist}(x, \Gamma_t^{\mathrm{MS}})>\alpha\right\}\right]=0\quad \text{for all}\; \alpha>0.
\end{align*}
\end{corollary}

\begin{proof}
Owing to \thmref{maintheorem2}, the proof  goes along the same lines as that of \cite[Corollary 5.8]{Banas19}. 
\end{proof}

\section*{Acknowledgement}
Funded by the Deutsche Forschungsgemeinschaft (DFG, German Research Foundation) -- Project-ID 317210226 -- SFB 1283.

\bibliographystyle{plain}
\bibliography{references}

\end{document}